\documentclass[hidelinks,onefignum,onetabnum]{siamart220329}
\usepackage{mathtools}
\usepackage{algorithm}
\usepackage{algorithmic}
\usepackage{subfigure}
\usepackage{enumitem}
\usepackage{amsmath}
\usepackage{amssymb}

\usepackage{microtype}

\makeatletter
\def\cref@override@label@type#1\@nil#2{#1}
\makeatother

\newlist{questions}{enumerate}{1}
\setlist[questions]{label=(Q\arabic*), leftmargin=*, align=left}

\usepackage{lipsum}
\usepackage{amsfonts}
\usepackage{graphicx}
\usepackage{epstopdf}
\usepackage{algorithmic}
\ifpdf
  \DeclareGraphicsExtensions{.eps,.pdf,.png,.jpg}
\else
  \DeclareGraphicsExtensions{.eps}
\fi

\newsiamremark{remark}{Remark}
\newsiamremark{hypothesis}{Hypothesis}
\crefname{hypothesis}{Hypothesis}{Hypotheses}
\newsiamthm{claim}{Claim}

\headers{2nd order explicit splitting for Stokes-Biot System}{}

\title{Second Order Explicit Splitting Scheme for Fluid-Poroelastic Structure Interaction Problems \thanks{Submitted to the editors DATE.
\funding{\v{C}ani\'{c}'s research has been supported in part by the
National Science Foundation under grants DMS-2408928, DMS-2247000 and by the  U.S. Department of Energy, Office of Science, Office of Advanced Scientific Computing Research's Applied Mathematics Competitive Portfolios program under Contract No. AC02-05CH11231. Yifan Wang’s research has been supported in part by the National Science Foundation under grant DMS-2247001, CPRIT Texas under grant RP260780, and by a Simons Foundation Travel Award.}}}

\author{
Yifan Wang \thanks{Department of Mathematics and Statistics, Texas Tech University, Lubbock, TX, USA (\email{yifan.wang@ttu.edu}).}
\and 
Jeonghun Lee \thanks{Department of Mathematics, Baylor University, Waco, TX, USA (\email{Jeonghun\_Lee@baylor.edu}).}
\and 
Sun\v{C}ica \v{C}ani\'{c} \thanks{\textbf{Corresponding author.} Department of Mathematics, University of California, Berkeley, Berkeley, CA, USA (\email{canics@berkeley.edu}).}
}

\usepackage{amsopn}

\allowdisplaybreaks[4]
\ifpdf
\hypersetup{
  pdftitle={Error Analysis of the Explicit Splitting Scheme for Fluid-Poroelastic Structure Interaction Problems},
  pdfauthor={Yifan Wang, Jeonghun Lee Sun\v{C}ica \v{C}ani\'{c}}
}
\fi

\newcommand{\sunny}{\color{black}}

\begin{document}

\maketitle
\tableofcontents

\begin{abstract}
Efficient and provably accurate partitioned methods for fluid--poroelastic structure interaction remain challenging because explicit treatment of the Stokes--Biot interface coupling condition can easily compromise stability. In this work,
we develop and analyze a fully discrete, second-order, explicit splitting scheme for the time-dependent Stokes--Biot problem on fixed domains. The method combines BDF2
time stepping with second-order Adams--Bashforth (AB2) extrapolation of the
interface data through a Robin reformulation, yielding a partitioned algorithm
in which the Stokes and Biot subproblems are solved independently and in parallel at each time step.

The main analytical contribution is a rigorous stability and error analysis for this second-order explicit coupling strategy. Using BDF2 energy identities, a sharp decomposition of the extrapolated interface terms, and discrete trace estimates, we prove a closed stability bound under a parabolic CFL condition. We then derive an \emph{a priori} error estimate through a projection-based framework involving a Fortin projection for the fluid variables and Ritz-type projections for the poroelastic variables. The analysis identifies the consistency defects arising from BDF2 time discretization, AB2 interface extrapolation, and the projected kinematic relation, and shows that the total errors in the fluid velocity, structure velocity, pore pressure, and elastic displacement are bounded by
\(C\,(h^k+\Delta t^2)\) in the bulk energy norms for \(1\le k\le 3\).
Numerical experiments with manufactured solutions confirm the predicted second-order temporal convergence and optimal-order spatial convergence. In addition, we include a moving-domain numerical example in which the fluid is governed by the Navier–Stokes equations, demonstrating the applicability of the proposed partitioned strategy beyond the fixed-domain Stokes--Biot setting analyzed. These results show that the scheme provides a parallelizable, second-order accurate, and mathematically justified explicit partitioned method for fluid--poroelastic interaction, with strong potential for more general moving-domain flow problems. 

\end{abstract}


\section{Introduction}

Fluid-poroelastic structure interaction problems arise in many settings where a free fluid interacts with a deformable porous medium, including biological tissue perfusion, flow through bioartificial devices, and transport in deformable porous materials \cite{Buka202404,fluids7070222,Benjamin2014,Martina2014,CanicSiam2021}. 
A standard mathematical description of these processes couples the incompressible time-dependent Stokes equations in the fluid region with the Biot system in the poroelastic region \cite{biot1941general,biot1955theory}. 
This Stokes-Biot framework has been studied extensively, both from the modeling and numerical points of view, and has led to a wide range of monolithic and partitioned discretization strategies \cite{ambartsumyan2018lagrange,wen2020strongly,guo2022decoupled,cesmelioglu2017analysis,CanicBook,BociuMulti,SingularLimit,guo2025uncond,Andrew2026,Yotov2023,NitCoup,MultiLayer,Martina2014,RobRob,MACscheme,CLR-Stokes-Biot:2023,CLR-NSE-Biot:2024}. 

A central difficulty in Stokes-Biot coupling is the treatment of the interface conditions. 
These conditions encode continuity of normal flux together with balance of normal and tangential stresses, and they create a strong coupling between the free-fluid and poroelastic dynamics. 
Monolithic methods handle this coupling robustly, but they typically require solving a large coupled system at every time step. 
Partitioned methods are computationally more attractive because they preserve the structure of the fluid and poroelastic subproblems and allow the use of specialized solvers, but explicit or loosely coupled formulations must be designed carefully in order to maintain stability and accuracy. 
Weak interface enforcement based on Nitsche- or Robin-type formulations has proved particularly effective in this context, since it leads to consistent coupling mechanisms that are well suited to partitioned and parallel algorithms \cite{BADIA20097986,BUKAC2015138,MartinaOyekole,cesmelioglu2016optimization, Wang2026ExplicitSplitting, wang2025}.

Motivated by these considerations, we study a fully discrete explicit splitting method for the time-dependent Stokes-Biot problem on fixed domains. 
The method is based on a Robin reformulation of the interface conditions, with tangential and normal coupling controlled by the parameters $\gamma>0$ and $L>0$, respectively. 
At the fully discrete level, the unknowns at time $t_{n+1}$ are computed by combining a BDF2 discretization of the subdomain evolution equations with AB2 extrapolation of the interface data from previous time levels. 
As a result, the Stokes and Biot subproblems can be solved independently at each time step, which makes the scheme naturally parallelizable while retaining a consistent treatment of the interface conditions.

The main contributions of this paper are: (i) a discrete stability estimate
under a parabolic CFL condition linking the time step and mesh size; (ii) an
\emph{a priori} error analysis showing that the total errors in fluid velocity,
structure velocity, pore pressure, and elastic displacement are bounded by
\(C\,(h^k+\Delta t^2)\) in the bulk energy norms for \(1\le k\le 3\); and
(iii) numerical experiments confirming the predicted second-order temporal and
optimal-order spatial convergence rates.

The remainder of the paper is organized as follows. 
Section 2 introduces the coupled Stokes-Biot model, its interface conditions, and the Robin reformulation used for splitting. 
Section 3 presents the fully discrete BDF2-AB2 partitioned scheme and establishes the discrete stability estimate. 
Section 4 develops the projection framework and proves the a priori error bounds. 
Section 5 reports numerical experiments that validate the theoretical convergence results.

\section{Continuous Problem}
\label{sec:continuous}

\subsection{Notation and function spaces}
\label{subsec:notation}

Let $\Omega_f\subset\mathbb{R}^d$ ($d=2,3$) denote the fluid domain and
$\Omega_p\subset\mathbb{R}^d$ the poroelastic domain. We assume throughout
that both subdomains are fixed in time. Their common interface is denoted by
\[
\Gamma=\partial\Omega_f\cap\partial\Omega_p.
\]
We write $\boldsymbol{n}_f$ for the unit outward normal on $\partial\Omega_f$
and $\boldsymbol{n}_p$ for the unit outward normal on $\partial\Omega_p$, so that
\[
\boldsymbol{n}_p=-\boldsymbol{n}_f \qquad \text{on } \Gamma.
\]
We also define the tangential projection operators
\[
\boldsymbol{P}_f\boldsymbol{v}
= \boldsymbol{v}-(\boldsymbol{v}\cdot\boldsymbol{n}_f)\boldsymbol{n}_f,
\qquad
\boldsymbol{P}_p\boldsymbol{v}
= \boldsymbol{v}-(\boldsymbol{v}\cdot\boldsymbol{n}_p)\boldsymbol{n}_p.
\]
The outer boundaries of the fluid and poroelastic domains are denoted by
$\Sigma_{D,f}$ and $\Sigma_{D,p}$, respectively.

For a Banach space $X$, a final time $T>0$, and an exponent $1\le q\le\infty$,
the Bochner space $L^q(0,T;X)$ consists of (equivalence classes of)
strongly measurable functions $v:(0,T)\to X$ such that
\[
\|v\|_{L^q(0,T;X)}
:=
\left(\int_0^T \|v(t)\|_X^q\,dt\right)^{1/q}<\infty
\qquad(1\le q<\infty),
\]
with the usual essential-supremum norm when $q=\infty$. For a nonnegative
integer $m$, the space $H^m(0,T;X)$ consists of all $v\in L^2(0,T;X)$ whose
weak time derivatives $\partial_t^j v$, $j=1,\dots,m$, exist in the sense of
$X$-valued distributions on $(0,T)$ and belong to $L^2(0,T;X)$, equipped with
the norm
\[
\|v\|_{H^m(0,T;X)}^2
:=
\sum_{j=0}^{m}\|\partial_t^j v\|_{L^2(0,T;X)}^2.
\]
The space $W^{m,\infty}(0,T;X)$ is defined analogously, with the $L^2(0,T;X)$
norms in each term replaced by $L^\infty(0,T;X)$ norms. When $X=H^k(\Omega)$
for a spatial domain $\Omega$ and integer $k\ge0$, functions
$v\in L^q(0,T;H^k(\Omega))$ may be identified with functions
$v=v(\boldsymbol x,t)$ on $\Omega\times(0,T)$ that, for a.e.\ $t\in(0,T)$,
belong to $H^k(\Omega)$ as functions of $\boldsymbol x$, with $t\mapsto
\|v(\cdot,t)\|_{H^k(\Omega)}$ in $L^q(0,T)$. We use these spaces, in
particular $L^2(0,T;X)$ and $H^1(0,T;X)$ with $X=L^2(\Omega_f)$,
$L^2(\Omega_p)$, $\boldsymbol{V}_f$, and $\boldsymbol{V}_p$ (the latter two
defined in Section~\ref{subsec:weak}), to state the weak formulation of the
coupled problem and the regularity assumptions on the exact solution used in
the error analysis of Section~\ref{sec:error}.

\subsection{Governing equations}
\label{subsec:governing}

The coupled Stokes-Biot system consists of the time-dependent incompressible
Stokes equations in $\Omega_f$ and the Biot poroelasticity equations in
$\Omega_p$:
\begin{subequations}\label{eq:stokes-biot}
\begin{align}
\rho_f \partial_t \boldsymbol{u}
- \nabla\cdot \boldsymbol{\sigma}_f(\boldsymbol{u},p_f)
&= \boldsymbol{f}_f
\qquad &&\text{in }\Omega_f\times(0,T), \label{eq:stokes_momentum}\\
\nabla\cdot \boldsymbol{u}
&= 0
\qquad &&\text{in }\Omega_f\times(0,T), \label{eq:stokes_incompressibility}\\
\rho_p \partial_t \boldsymbol{\xi}
- \nabla\cdot \boldsymbol{\sigma}_p(\boldsymbol{\eta},\phi)
&= \boldsymbol{f}_p
\qquad &&\text{in }\Omega_p\times(0,T), \label{eq:biot_momentum}\\
\boldsymbol{\xi} &= \partial_t \boldsymbol{\eta}
\qquad &&\text{in }\Omega_p\times(0,T), \label{eq:kinematic_relation}\\
C_0 \partial_t \phi + \alpha \nabla\cdot \boldsymbol{\xi}
- \nabla\cdot(K\nabla\phi)
&= g_p
\qquad &&\text{in }\Omega_p\times(0,T), \label{eq:biot_mass}\\
\boldsymbol{u}_p &= -K\nabla\phi
\qquad &&\text{in }\Omega_p\times(0,T). \label{eq:darcy_velocity}
\end{align}
\end{subequations}
Here $\boldsymbol{u}$ and $p_f$ denote the fluid velocity and pressure,
$\boldsymbol{\eta}$ and $\boldsymbol{\xi}$ denote the poroelastic
displacement and velocity, $\phi$ is the pore pressure, and
$\boldsymbol{u}_p$ is the Darcy filtration velocity.

The stress tensors are given by
\begin{equation}\label{eq:stresses}
\boldsymbol{\sigma}_f(\boldsymbol{u},p_f)
= 2\mu_f \boldsymbol{D}(\boldsymbol{u}) - p_f \boldsymbol{I},
\qquad
\boldsymbol{\sigma}_p(\boldsymbol{\eta},\phi)
= 2\mu_p \boldsymbol{D}(\boldsymbol{\eta})
+ \lambda_p (\nabla\cdot\boldsymbol{\eta})\boldsymbol{I}
- \alpha \phi \boldsymbol{I},
\end{equation}
where
\[
\boldsymbol{D}(\boldsymbol{v})
= \frac{\nabla\boldsymbol{v}+\nabla\boldsymbol{v}^{\top}}{2}.
\]
The parameters $\rho_f,\mu_f$ are the fluid density and viscosity,
$\rho_p,\mu_p,\lambda_p$ are the poroelastic density and Lam\'e constants,
$\alpha$ is the Biot-Willis coefficient, $K$ is the hydraulic conductivity
tensor, and $C_0 > 0$ is the storage coefficient.

Throughout the paper we assume that the physical parameters are constants
satisfying
\[
\rho_f,\ \mu_f,\ \rho_p,\ \mu_p,\ \lambda_p,\ C_0 > 0,
\qquad
\alpha\in(0,1],
\]
and that the hydraulic conductivity tensor $K$ is symmetric and uniformly
positive definite, i.e.\ there exists $k_0>0$ such that
\begin{equation}\label{eq:K_coercive_intro}
(K\nabla \psi,\nabla \psi)_p \ge k_0 \|\nabla \psi\|_{L^2(\Omega_p)}^2
\qquad
\forall \psi\in H^1(\Omega_p).
\end{equation}
The interface coupling parameters $\gamma>0$ and $L>0$, introduced in
Section~\ref{subsec:interface} below, are also fixed positive constants. All
constants appearing in the stability and error estimates of Sections~3--4
depend on these parameters, but not on the discretization parameters $h$ and
$\Delta t$.

\subsection{Interface conditions}
\label{subsec:interface}

The fluid and poroelastic subproblems are coupled through the following
conditions on $\Gamma\times(0,T)$:
\begin{subequations}\label{eq:interface_conditions}
\begin{align}
(\boldsymbol{\xi}+\boldsymbol{u}_p)\cdot\boldsymbol{n}_f
&= \boldsymbol{u}\cdot\boldsymbol{n}_f,
\label{eq:interface_flux}\\
\boldsymbol{\tau}_{f,j}\cdot
\boldsymbol{\sigma}_f(\boldsymbol{u},p_f)\boldsymbol{n}_f
&=
-\gamma (\boldsymbol{u}-\boldsymbol{\xi})\cdot \boldsymbol{\tau}_{f,j},
\qquad j=1,\dots,d-1,
\label{eq:interface_bjs}\\
\boldsymbol{n}_f\cdot
\boldsymbol{\sigma}_f(\boldsymbol{u},p_f)\boldsymbol{n}_f
&= -\phi,
\label{eq:interface_normal_stress}\\
\boldsymbol{\sigma}_f(\boldsymbol{u},p_f)\boldsymbol{n}_f
&=
\boldsymbol{\sigma}_p(\boldsymbol{\eta},\phi)\boldsymbol{n}_f.
\label{eq:interface_stress_balance}
\end{align}
\end{subequations}
Here $\{\boldsymbol{\tau}_{f,j}\}_{j=1}^{d-1}$ denotes an orthonormal basis
of tangential vectors on $\Gamma$, and $\gamma>0$ is the
Beavers-Joseph-Saffman coefficient.

\subsection{Weak formulation}
\label{subsec:weak}

We introduce the spaces
\[
\begin{aligned}
\boldsymbol{V}_f
&=
\{\boldsymbol{v}\in H^1(\Omega_f)^d:\boldsymbol{v}=0
\text{ on }\Sigma_{D,f}\},
\qquad
Q_f = L^2(\Omega_f),\\
\boldsymbol{V}_p
&=
\{\boldsymbol{\zeta}\in H^1(\Omega_p)^d:\boldsymbol{\zeta}=0
\text{ on }\Sigma_{D,p}\},
\qquad
Q_p=\left\{\psi \in H^1\left(\Omega_p\right): \psi=0 \text { on } \Sigma_{D, p}\right\}.
\end{aligned}
\]
We use $(\cdot,\cdot)_f$ and $(\cdot,\cdot)_p$ for the $L^2$ inner products on
$\Omega_f$ and $\Omega_p$, and $\langle\cdot,\cdot\rangle_\Gamma$ for the
duality pairing on $\Gamma$.

The weak formulation reads as follows: find
\[
\boldsymbol{u}\in L^2(0,T;\boldsymbol{V}_f)\cap H^1(0,T;L^2(\Omega_f)^d),
\qquad
p_f\in L^2(0,T;Q_f),
\]
\[
(\boldsymbol{\eta},\boldsymbol{\xi})
\in L^2(0,T;\boldsymbol{V}_p)^2\cap H^1(0,T;L^2(\Omega_p)^d)^2,
\qquad
\phi\in L^2(0,T;Q_p)\cap H^1(0,T;L^2(\Omega_p)),
\]
with $\boldsymbol{u}(0)=\boldsymbol{u}_0$, $\boldsymbol{\eta}(0)=\boldsymbol{\eta}_0$,
$\boldsymbol{\xi}(0)=\boldsymbol{\xi}_0$, and $\phi(0)=\phi_0$, such that for
almost every $t\in(0,T)$,
\begin{align}
&\rho_f(\partial_t\boldsymbol{u},\boldsymbol{v})_f
+2\mu_f(\boldsymbol{D}(\boldsymbol{u}),\boldsymbol{D}(\boldsymbol{v}))_f
-(p_f,\nabla\cdot\boldsymbol{v})_f
+(\nabla\cdot\boldsymbol{u},q)_f
\notag\\
&\quad
+\langle \phi,\boldsymbol{v}\cdot\boldsymbol{n}_f\rangle_\Gamma
+\gamma\langle \boldsymbol{P}_f(\boldsymbol{u}-\boldsymbol{\xi}),
\boldsymbol{P}_f(\boldsymbol{v})\rangle_\Gamma
\notag\\
&\quad
+\rho_p(\partial_t\boldsymbol{\xi},\boldsymbol{\zeta})_p
+(\boldsymbol{\xi} - \partial_t \boldsymbol{\eta}, \boldsymbol{v}_p)_p
+2\mu_p(\boldsymbol{D}(\boldsymbol{\eta}),\boldsymbol{D}(\boldsymbol{\zeta}))_p
\notag\\
&\quad
+\lambda_p(\nabla\cdot\boldsymbol{\eta},\nabla\cdot\boldsymbol{\zeta})_p
-\alpha(\phi,\nabla\cdot\boldsymbol{\zeta})_p
\notag\\
&\quad
+C_0(\partial_t\phi,\psi)_p
+\alpha(\nabla\cdot\boldsymbol{\xi},\psi)_p
+(K\nabla\phi,\nabla\psi)_p
\notag\\
&\quad
-\gamma\langle \boldsymbol{P}_p(\boldsymbol{u}-\boldsymbol{\xi}),
\boldsymbol{P}_p(\boldsymbol{\zeta})\rangle_\Gamma
+\langle (\boldsymbol u-\boldsymbol\xi)\cdot\boldsymbol n_p,\psi\rangle_\Gamma
+\langle \phi,\boldsymbol{\zeta}\cdot\boldsymbol{n}_p\rangle_\Gamma
\notag\\
&=
(\boldsymbol{f}_f,\boldsymbol{v})_f
+(\boldsymbol{f}_p,\boldsymbol{\zeta})_p
+(g_p,\psi)_p
\label{eq:weak_coupled}
\end{align}
for all time-independent test functions
\[
(\boldsymbol{v},q,\boldsymbol{\zeta},\psi,\boldsymbol{v}_p)
\in \boldsymbol{V}_f\times Q_f\times \boldsymbol{V}_p\times Q_p\times \boldsymbol{V}_p,
\]
where $\boldsymbol{V}_f, Q_f, \boldsymbol{V}_p, Q_p$ are the spatial spaces
introduced above.

In what follows, we suppress external forcing terms when deriving the
stability and error estimates, in order to simplify the presentation.

\subsection{Decoupled continuous Stokes-Biot system}
\label{subsec:decoupled_continuous}

To derive the fully explicit splitting scheme, we rewrite all the interface terms in Robin
form. For a given penalty parameter $L>0$, the fluid-side interface conditions can be written as
\begin{equation}\label{eq:robin_fluid_cont}
\begin{array}{ll}
\boldsymbol{n}_f\cdot
\big(\boldsymbol{\sigma}_f(\boldsymbol{u},p_f)\boldsymbol{n}_f\big)
+L\,\boldsymbol{u}\cdot\boldsymbol{n}_f
=
L\,\boldsymbol{u}\cdot\boldsymbol{n}_f-\phi,
\\[2mm]
\boldsymbol{P}_f\big(\boldsymbol{\sigma}_f(\boldsymbol{u},p_f)\boldsymbol{n}_f\big)
+\gamma \boldsymbol{P}_f(\boldsymbol{u})
=
\gamma \boldsymbol{P}_f(\boldsymbol{\xi}).
\end{array}
\end{equation}
Similarly, the Biot-side interface conditions can be written as
\begin{equation}\label{eq:robin_biot_cont}
\begin{array}{ll}
\boldsymbol{n}_p\cdot
\big(\boldsymbol{\sigma}_p(\boldsymbol{\eta},\phi)\boldsymbol{n}_p\big)
+\phi
+\boldsymbol{\xi}\cdot\boldsymbol{n}_p
=
\boldsymbol{\xi}\cdot\boldsymbol{n}_p,
\\[2mm]
K\nabla\phi\cdot\boldsymbol{n}_p
+\phi/L
-\boldsymbol{\xi}\cdot\boldsymbol{n}_p
=
\phi/L-\boldsymbol{u}\cdot\boldsymbol{n}_p,
\\[2mm]
\boldsymbol{P}_p\big(\boldsymbol{\sigma}_p(\boldsymbol{\eta},\phi)\boldsymbol{n}_p\big)
+\gamma \boldsymbol{P}_p(\boldsymbol{\xi})
=
\gamma \boldsymbol{P}_p(\boldsymbol{u}).
\end{array}
\end{equation}
In the fully discrete scheme presented below, the left-hand sides of the above Robin boundary conditions
will be imposed at time $t_{n+1}$, while the right-hand sides will be approximated explicitly by AB2 extrapolation from previous time levels, which eventually leads to a parallelizable, fully explicit splitting algorithm.

The decoupled continuous problem can now be written as follows.

{\textbf{Fluid subproblem.}}
Find $(\boldsymbol{u},p_f)\in \boldsymbol{V}_f\times Q_f$ such that,
for all $(\boldsymbol{v},q)\in \boldsymbol{V}_f\times Q_f$,
\begin{align}
&\rho_f \big( \partial_t \boldsymbol{u} , \boldsymbol{v} \big)_{f}
+ 2\mu_f \big( \boldsymbol{D}(\boldsymbol{u}), \boldsymbol{D}(\boldsymbol{v}) \big)_{f}
- (p_f, \nabla \cdot \boldsymbol{v} )_{f}
+ (\nabla \cdot \boldsymbol{u}, q )_{f} \notag\\
&\quad
+ \gamma \langle \boldsymbol{P}_f(\boldsymbol{u}), \boldsymbol{P}_f(\boldsymbol{v}) \rangle_{\Gamma}
+ L \langle \boldsymbol{u} \cdot \boldsymbol{n}_f,\; \boldsymbol{v} \cdot \boldsymbol{n}_f \rangle_{\Gamma} \notag\\
&= \gamma \langle  \boldsymbol{P}_f(\boldsymbol{\xi}), \boldsymbol{P}_f(\boldsymbol{v}) \rangle_{\Gamma}
+ \langle L \boldsymbol{u} \cdot \boldsymbol{n}_f - \phi, \boldsymbol{v} \cdot \boldsymbol{n}_f \rangle_{\Gamma}.
\label{eq:continuous-Stokes-eq}
\end{align}

{\textbf{Biot subproblem.}}
Find $(\boldsymbol{\eta},\boldsymbol{\xi},\phi)\in \boldsymbol{V}_p\times \boldsymbol{V}_p\times Q_p$ such that,
for all $(\boldsymbol{v}_p,\boldsymbol{\zeta},\psi)\in \boldsymbol{V}_p\times \boldsymbol{V}_p\times Q_p$,
\begin{align}
&\rho_p \big( \partial_t \boldsymbol{\xi}, \boldsymbol{\zeta} \big)_{p}
+ 2\mu_p \big( \boldsymbol{D}(\boldsymbol{\eta}), \boldsymbol{D}(\boldsymbol{\zeta}) \big)_{p}
+ \lambda_p \big( \nabla \cdot \boldsymbol{\eta}, \nabla \cdot \boldsymbol{\zeta} \big)_{p}
- \alpha \big( \phi, \nabla \cdot \boldsymbol{\zeta} \big)_{p} \notag\\
&\quad
+(\boldsymbol{\xi}-\partial_t\boldsymbol{\eta},\boldsymbol{v}_p)_p
+ C_0 \big( \partial_t \phi, \psi \big)_{p}
+ \alpha \big( \nabla \cdot \boldsymbol{\xi}, \psi \big)_{p}
+ \big( K \nabla \phi, \nabla \psi \big)_{p} \notag\\
&\quad
+ \gamma \langle \boldsymbol{P}_p(\boldsymbol{\xi}), \boldsymbol{P}_p(\boldsymbol{\zeta}) \rangle_{\Gamma}
+ \langle \boldsymbol{\xi} \cdot \boldsymbol{n}_p,\; \boldsymbol{\zeta} \cdot \boldsymbol{n}_p \rangle_{\Gamma}
+ \langle \phi, \boldsymbol{\zeta} \cdot \boldsymbol{n}_p \rangle_{\Gamma} \notag\\
&\quad
+ \frac{1}{L} \langle \phi, \psi \rangle_{\Gamma}
- \langle \boldsymbol{\xi} \cdot \boldsymbol{n}_p, \psi \rangle_{\Gamma} \notag\\
&= \gamma \langle \boldsymbol{P}_p (\boldsymbol{u}), \boldsymbol{P}_p(\boldsymbol{\zeta}) \rangle_{\Gamma}
+ \langle \boldsymbol{\xi} \cdot \boldsymbol{n}_p, \boldsymbol{\zeta} \cdot \boldsymbol{n}_p \rangle_{\Gamma}
+ \langle -\boldsymbol{u} \cdot \boldsymbol{n}_p + \phi/L, \psi \rangle_{\Gamma}
\label{eq:continuous-poroelasticity-eq}
\end{align}

\subsection{Fully discrete second-order parallel splitting algorithm}
\label{subsec:scheme_bdf2}

We now introduce the fully discrete second-order partitioned scheme for the
fixed-domain Stokes-Biot problem.

Let $\{\mathcal{T}_{f,h}\}$ and $\{\mathcal{T}_{p,h}\}$ be conforming,
shape-regular triangulations of the fixed domains $\Omega_f$ and $\Omega_p$,
respectively, chosen so that the discrete interface coincides with $\Gamma$.
We define the finite element spaces
\[
\boldsymbol{V}_{f,h}\subset \boldsymbol{V}_f,
\qquad
Q_{f,h}\subset Q_f,
\qquad
\boldsymbol{V}_{p,h}\subset \boldsymbol{V}_p,
\qquad
Q_{p,h}\subset Q_p.
\]
Since the domains are fixed, these spaces do not depend on time.

Let the time interval $[0,T]$ be partitioned uniformly by
\[
0=t_0<t_1<\cdots<t_N=T,
\qquad
\Delta t=t_{n+1}-t_n.
\]
For a sequence $\{a^n\}_{n\ge 0}$, we define the BDF2 difference operator
\begin{equation}
\delta_t a^{n+1}
:=
\frac{3a^{n+1}-4a^n+a^{n-1}}{2\Delta t},
\qquad n\ge 1,
\label{BDF2}
\end{equation}
and the AB2 extrapolation
\begin{equation}
a^{\star,n+1}:=2a^n-a^{n-1}.
\label{AB2}
\end{equation}
The displacement and structure velocity are linked by the BDF2 kinematic relation
\begin{equation}\label{eq:xi_equals_bdf2_eta}
\boldsymbol{\xi}_h^{n+1}=\delta_t \boldsymbol{\eta}_h^{n+1}
=
\frac{3\boldsymbol{\eta}_h^{n+1}-4\boldsymbol{\eta}_h^{n}+\boldsymbol{\eta}_h^{n-1}}{2\Delta t}.
\end{equation}

Given $(\boldsymbol{u}_h^n,p_{f,h}^n,\boldsymbol{\eta}_h^n,\boldsymbol{\xi}_h^n,\phi_h^n)$ and $(\boldsymbol{u}_h^{n-1},p_{f,h}^{n-1},\boldsymbol{\eta}_h^{n-1},\boldsymbol{\xi}_h^{n-1},\phi_h^{n-1})$,
we compute the solution at time $t_{n+1}$ by solving the following two subproblems independently and in parallel.

\begin{enumerate}[leftmargin=12pt,labelsep=0.5em]

\item{\textbf{Discrete Stokes subproblem (BDF2).}}

Find $(\boldsymbol{u}_h^{n+1},p_{f,h}^{n+1})\in \boldsymbol{V}_{f,h}\times Q_{f,h}$
such that for all $(\boldsymbol{v}_h,q_h)\in \boldsymbol{V}_{f,h}\times Q_{f,h}$,
\begin{equation}\label{eq:fluid-discrete-bdf2-fixed}
\begin{aligned}
&\rho_f\Big(\delta_t \boldsymbol{u}_h^{n+1},\boldsymbol{v}_h\Big)_{f}
+2\mu_f\Big(\boldsymbol{D}(\boldsymbol{u}_h^{n+1}),\boldsymbol{D}(\boldsymbol{v}_h)\Big)_{f}
-\Big(p_{f,h}^{n+1},\nabla\!\cdot\!\boldsymbol{v}_h\Big)_{f}
+\Big(\nabla\!\cdot\!\boldsymbol{u}_h^{n+1},q_h\Big)_{f}
\\
&\quad
+\gamma\Big\langle \boldsymbol{P}_f\boldsymbol{u}_h^{n+1},
\boldsymbol{P}_f\boldsymbol{v}_h\Big\rangle_{\Gamma}
+L\Big\langle \boldsymbol{u}_h^{n+1}\!\cdot\!\boldsymbol{n}_f,
\boldsymbol{v}_h\!\cdot\!\boldsymbol{n}_f\Big\rangle_{\Gamma}
\\
&=
\gamma\Big\langle\boldsymbol{P}_f\boldsymbol{\xi}_h^{\star,n+1},
\boldsymbol{P}_f\boldsymbol{v}_h\Big\rangle_{\Gamma}
+\Big\langle
L\big(\boldsymbol{u}_h^{\star,n+1}\!\cdot\!\boldsymbol{n}_f\big)-\phi_h^{\star,n+1},
\boldsymbol{v}_h\!\cdot\!\boldsymbol{n}_f\Big\rangle_{\Gamma}.
\end{aligned}
\end{equation}

\item{\textbf{Discrete Biot subproblem (BDF2).}}

Find $(\boldsymbol{\eta}_h^{n+1},\boldsymbol{\xi}_h^{n+1},\phi_h^{n+1})
\in \boldsymbol{V}_{p,h}\times \boldsymbol{V}_{p,h}\times Q_{p,h}$
such that for all $(\boldsymbol{v}_{p,h},\boldsymbol{\zeta}_h,\psi_h)
\in \boldsymbol{V}_{p,h}\times \boldsymbol{V}_{p,h}\times Q_{p,h}$,
\begin{equation}\label{eq:poro-discrete-bdf2-fixed}
\begin{aligned}
&\rho_p\Big(\delta_t \boldsymbol{\xi}_h^{n+1},\boldsymbol{\zeta}_h\Big)_p
+2\mu_p\Big(\boldsymbol{D}(\boldsymbol{\eta}_h^{n+1}),\boldsymbol{D}(\boldsymbol{\zeta}_h)\Big)_p
+\lambda_p\Big(\nabla\!\cdot\!\boldsymbol{\eta}_h^{n+1},\nabla\!\cdot\!\boldsymbol{\zeta}_h\Big)_p
-\alpha\Big(\phi_h^{n+1},\nabla\!\cdot\!\boldsymbol{\zeta}_h\Big)_p
\\
&\quad
+\Big(\boldsymbol{\xi}_h^{n+1},\boldsymbol{v}_{p,h}\Big)_p
-\Big(\delta_t \boldsymbol{\eta}_h^{n+1},\boldsymbol{v}_{p,h}\Big)_p
\\
&\quad
+C_0\Big(\delta_t \phi_h^{n+1},\psi_h\Big)_p
+\alpha\Big(\nabla\!\cdot\!\boldsymbol{\xi}_h^{n+1},\psi_h\Big)_p
+\Big(K\nabla\phi_h^{n+1},\nabla\psi_h\Big)_p
\\
&\quad
+\gamma\Big\langle \boldsymbol{P}_p\boldsymbol{\xi}_h^{n+1},
\boldsymbol{P}_p\boldsymbol{\zeta}_h\Big\rangle_{\Gamma}
+\Big\langle \boldsymbol{\xi}_h^{n+1}\!\cdot\!\boldsymbol{n}_p,
\boldsymbol{\zeta}_h\!\cdot\!\boldsymbol{n}_p\Big\rangle_{\Gamma}
+\Big\langle \phi_h^{n+1},\boldsymbol{\zeta}_h\!\cdot\!\boldsymbol{n}_p\Big\rangle_{\Gamma}
\\
&\quad
+\frac{1}{L}\Big\langle\phi_h^{n+1},\psi_h\Big\rangle_{\Gamma}
-\Big\langle \boldsymbol{\xi}_h^{n+1}\!\cdot\!\boldsymbol{n}_p,\psi_h\Big\rangle_{\Gamma}
\\
&=
\gamma\Big\langle\boldsymbol{P}_p\boldsymbol{u}_h^{\star,n+1},
\boldsymbol{P}_p\boldsymbol{\zeta}_h\Big\rangle_{\Gamma}
+\Big\langle
\boldsymbol{\xi}_h^{\star,n+1}\!\cdot\!\boldsymbol{n}_p,
\boldsymbol{\zeta}_h\!\cdot\!\boldsymbol{n}_p\Big\rangle_{\Gamma}
+\Big\langle
-\boldsymbol{u}_h^{\star,n+1}\!\cdot\!\boldsymbol{n}_p
+\tfrac{1}{L}\phi_h^{\star,n+1},
\psi_h\Big\rangle_{\Gamma}.
\end{aligned}
\end{equation}

\end{enumerate}

\begin{remark} For the first time step, the BDF2-AB2 scheme is initialized by a second-order accurate starting procedure. In the numerical experiments, this
is done by prescribing the exact solution at $t_0$ and $t_1$ when a manufactured solution is available.
\end{remark}

\begin{remark}
For the displacement update, enforcing the kinematic relation implicitly within the Biot subproblem is not the only possible choice. One option is to impose it weakly through the auxiliary term
\[
\left(\boldsymbol{\xi}_h^{n+1}, \boldsymbol{v}_{p,h}\right)_p
-
\left(\delta_t \boldsymbol{\eta}_h^{n+1}, \boldsymbol{v}_{p,h}\right)_p,
\]
where the displacement and structure velocity are linked by the BDF2 kinematic relation, as in \eqref{eq:xi_equals_bdf2_eta}.
Alternatively, after computing $\boldsymbol{\xi}_h^{n+1}$, one may update the displacement by the trapezoidal rule,
\[
\boldsymbol{\eta}_h^{n+1}
=
\boldsymbol{\eta}_h^n
+\frac{\Delta t}{2}
\left(\boldsymbol{\xi}_h^{n+1}+\boldsymbol{\xi}_h^n\right).
\]
Both approaches are formally second-order accurate in time, so neither has an inherent advantage purely at the level of consistency order.
However, for strongly coupled fluid-structure interactions, the trapezoidal rule is typically less dissipative, whereas BDF2 is more dissipative and often more robust. Hence, for stiff poroelastic type problems, the BDF2 approach may be a safer choice from a computational viewpoint.
\end{remark}

\section{Stability analysis}
\label{sec:stability}

In this section we derive a discrete stability bound for the fully discrete
second-order partitioned scheme introduced in
Section~\ref{subsec:scheme_bdf2}. Since the domains are fixed, all norms and
inner products are taken on the fixed domains $\Omega_f$,
$\Omega_p$, and the interface $\Gamma$.
In addition, for simplicity of presentation, we suppress external forcing terms in the
stability proof. Their contribution can be incorporated by standard
Cauchy-Schwarz and Young inequalities.

\subsection{Preliminaries}
\label{subsec:stability_preliminaries}

We first recall the BDF2 identity: for any inner product space with inner
product $(\cdot,\cdot)$ and induced norm $\|\cdot\|$,
\begin{equation}\label{eq:bdf2_identity}
2(3a-4b+c,a)
=
\|a\|^2+\|2a-b\|^2-\|b\|^2-\|2b-c\|^2+\|a-2b+c\|^2.
\end{equation}
Since,
$$
\delta_t a^{n+1}=\frac{3 a^{n+1}-4 a^n+a^{n-1}}{2 \Delta t},
$$
we have
\begin{equation}
\begin{aligned}
\left(\delta_t a^{n+1},a^{n+1}\right)
=
\frac{1}{4\Delta t}
\Big(
&\|a^{n+1}\|^2+\|2a^{n+1}-a^n\|^2
-\|a^n\|^2-\|2a^n-a^{n-1}\|^2\\
&+\|a^{n+1}-2a^n+a^{n-1}\|^2
\Big).\label{eq:bdf2_energy_identity}
\end{aligned}
\end{equation}

Next, we apply the same identity to the elastic bilinear form
\[
a_e(\boldsymbol{\eta},\boldsymbol{\zeta})
:=
2\mu_p\big(\boldsymbol{D}(\boldsymbol{\eta}),\boldsymbol{D}(\boldsymbol{\zeta})\big)_p
+
\lambda_p\big(\nabla\cdot\boldsymbol{\eta},\nabla\cdot\boldsymbol{\zeta}\big)_p.
\]
Since $a_e(\cdot,\cdot)$ is symmetric, it induces the elastic energy seminorm
\[
\|\boldsymbol{\eta}\|_{a_e}^2:=a_e(\boldsymbol{\eta},\boldsymbol{\eta}).
\]
Therefore, viewing $a_e(\cdot,\cdot)$ as the underlying inner product and applying \eqref{eq:bdf2_energy_identity}, we obtain:
\begin{align}\label{eq:bdf2_elastic_identity}
a_e(\boldsymbol{\eta}^{n+1},\delta_t\boldsymbol{\eta}^{n+1})
=&\frac{1}{2 \Delta t} a_e\left(\eta^{n+1}, 3 \eta^{n+1}-4 \eta^n+\eta^{n-1}\right)\nonumber \\
=&\frac{1}{4\Delta t}
\Big(
\|\boldsymbol{\eta}^{n+1}\|_{a_e}^2
+\|2\boldsymbol{\eta}^{n+1}-\boldsymbol{\eta}^{n}\|_{a_e}^2
-\|\boldsymbol{\eta}^{n}\|_{a_e}^2 \nonumber \\
&-\|2\boldsymbol{\eta}^{n}-\boldsymbol{\eta}^{n-1}\|_{a_e}^2
+\|\boldsymbol{\eta}^{n+1}-2\boldsymbol{\eta}^{n}+\boldsymbol{\eta}^{n-1}\|_{a_e}^2
\Big).
\end{align}

For the explicit interface terms, we also need a bound on the AB2
extrapolation $a^{\star,n+1}=2a^n-a^{n-1}$ in terms of $a^n$ and
$2a^n-a^{n-1}$. Writing $a^{\star,n+1}=a^n+(a^n-a^{n-1})$ and
$a^n-a^{n-1}=(2a^n-a^{n-1})-a^n$, two applications of
$|x+y|^2\le 2|x|^2+2|y|^2$ give
\[
\|a^{\star,n+1}\|^2
\le
2\|a^n\|^2+2\|a^n-a^{n-1}\|^2
\le
2\|a^n\|^2+4\|2a^n-a^{n-1}\|^2+4\|a^n\|^2,
\]
and hence
\begin{equation}\label{eq:ab2_bound}
\|a^{\star,n+1}\|^2
\le
C\Big(\|a^n\|^2+\|2a^n-a^{n-1}\|^2\Big).
\end{equation}
Thus the extrapolated quantity is controlled by solution values at the two
previous time levels; this estimate will be used later to control the
explicit AB2 interface traces through the augmented interface BDF2 energies.

For later use and motivated by the BDF2 identities \eqref{eq:bdf2_energy_identity} and
\eqref{eq:bdf2_elastic_identity}, we introduce the discrete BDF2 energies associated with the
fluid velocity, structure velocity, pore pressure, and elastic displacement, respectively: 
\begin{align}
E_u^{n+1}
&:=
\frac{\rho_f}{4}
\Big(
\|\boldsymbol{u}_h^{n+1}\|_{L^2 (\Omega_f)}^2
+
\|2\boldsymbol{u}_h^{n+1}-\boldsymbol{u}_h^n\|_{L^2 (\Omega_f)}^2
\Big),
\label{eq:def_Ef}
\\
E_\xi^{n+1}
&:=
\frac{\rho_p}{4}
\Big(
\|\boldsymbol{\xi}_h^{n+1}\|_{L^2 (\Omega_p)}^2
+
\|2\boldsymbol{\xi}_h^{n+1}-\boldsymbol{\xi}_h^n\|_{L^2 (\Omega_p)}^2
\Big),
\label{eq:def_Exi}
\\
E_\phi^{n+1}
&:=
\frac{C_0}{4}
\Big(
\|\phi_h^{n+1}\|_{L^2 (\Omega_p)}^2
+
\|2\phi_h^{n+1}-\phi_h^n\|_{L^2 (\Omega_p)}^2
\Big),
\label{eq:def_Ephi}
\\
E_\eta^{n+1}
&:=
\frac{1}{4}
\Big(
\|\boldsymbol{\eta}_h^{n+1}\|_{a_e}^2
+
\|2\boldsymbol{\eta}_h^{n+1}-\boldsymbol{\eta}_h^n\|_{a_e}^2
\Big).
\label{eq:def_Eeta}
\end{align}
In particular, these discrete energies are defined with the factor $\frac{1}{4}$, which arises from the BDF2 identity.

We also define the total bulk BDF2 energy
\begin{equation}\label{eq:def_total_energy}
\mathcal E^{n+1}:=E_u^{n+1}+E_\xi^{n+1}+E_\phi^{n+1}+E_\eta^{n+1}.
\end{equation}
To control the explicit interface terms, we introduce the current-step
interface energy as follows:
\begin{align}
\mathcal D_\Gamma^{n+1}
&:=
\gamma\|\boldsymbol{P}_f\boldsymbol{u}_h^{n+1}\|_{L^2(\Gamma)}^2
+\gamma\|\boldsymbol{P}_p\boldsymbol{\xi}_h^{n+1}\|_{L^2(\Gamma)}^2
\notag\\
&\quad
+L\|\boldsymbol{u}_h^{n+1}\cdot\boldsymbol{n}_f\|_{L^2(\Gamma)}^2
+\|\boldsymbol{\xi}_h^{n+1}\cdot\boldsymbol{n}_p\|_{L^2(\Gamma)}^2
+\frac1L\|\phi_h^{n+1}\|_{L^2(\Gamma)}^2 .
\label{eq:def_DGamma}
\end{align}
We also introduce the total bulk BDF2 dissipation
\begin{align}
\mathcal D_{\mathrm{bulk}}^{n+1}
&:=
\frac{\rho_f}{4\Delta t}
\|\boldsymbol{u}_h^{n+1}-2\boldsymbol{u}_h^n+\boldsymbol{u}_h^{n-1}\|_{L^2(\Omega_f)}^2
+\frac{\rho_p}{4\Delta t}
\|\boldsymbol{\xi}_h^{n+1}-2\boldsymbol{\xi}_h^n+\boldsymbol{\xi}_h^{n-1}\|_{L^2(\Omega_p)}^2
\notag\\
&\quad
+\frac{C_0}{4\Delta t}
\|\phi_h^{n+1}-2\phi_h^n+\phi_h^{n-1}\|_{L^2(\Omega_p)}^2
+\frac{1}{4\Delta t}
\|\boldsymbol{\eta}_h^{n+1}-2\boldsymbol{\eta}_h^n+\boldsymbol{\eta}_h^{n-1}\|_{a_e}^2
\notag\\
&\quad
+2\mu_f\|\boldsymbol{D}(\boldsymbol{u}_h^{n+1})\|_{L^2(\Omega_f)}^2
+\|K^{1/2}\nabla\phi_h^{n+1}\|_{L^2(\Omega_p)}^2.
\label{eq:def_bulk_diss_final}
\end{align}

For the fluid velocity space, we use the discrete trace inequality
\begin{equation}\label{eq:disc_trace_fluid}
\| \boldsymbol v_h \|_{L^2(\Gamma)}^2
\le
C_{\mathrm{tr}}
\Big(
h^{-1}\|\boldsymbol v_h\|_{L^2(\Omega_f)}^2
+
h\|\boldsymbol D(\boldsymbol v_h)\|_{L^2(\Omega_f)}^2
\Big),
\qquad
\forall \boldsymbol v_h\in \boldsymbol V_{f,h},
\end{equation}
so part of that interface term can be later absorbed into the $\mu_f\left\|\boldsymbol{D}\left(\boldsymbol{u}_h^{n+1}\right)\right\|_{L^2(\Omega_f)}^2$ term.
While, for the poroelastic variables, we use the discrete trace-inverse inequality:
\begin{equation}\label{eq:disc_trace_inverse_poro}
\| w_h \|_{L^2(\Gamma)}^2
\le
C_{\Gamma} h^{-1}\|w_h\|_{L^2(\Omega_p)}^2,
\qquad
\forall w_h \in W_{p,h},
\end{equation}
where \(W_{p,h}\) denotes either the discrete scalar or discrete vector-valued
poroelastic finite element space. Here, the constant $C_{\Gamma}>0$ is independent of $h$.

\begin{theorem}[Closed stability under a parabolic CFL condition]
\label{thm:stability_cfl}
There exist constants \(\Delta t_0>0\), \(c_\ast>0\), and \(C_T>0\),
independent of \(h\) and \(\Delta t\), such that if
\[
\Delta t \le \min\{\Delta t_0,\, c_\ast h^2\},
\]
the fully discrete BDF2-AB2 partitioned scheme satisfies
\[
\mathcal E^{m+1}
+\frac{\Delta t}{2}\sum_{n=1}^{m}\mathcal D_{\mathrm{bulk}}^{n+1}
\le
C_T\,\mathcal E^1,
\qquad m\ge 1.
\]
The constants \(\Delta t_0\) and \(c_\ast\) are made explicit in the proof.
\end{theorem}

\begin{proof}
We consider the fully discrete fixed-domain scheme
\eqref{eq:fluid-discrete-bdf2-fixed}-\eqref{eq:poro-discrete-bdf2-fixed},
with the BDF2 kinematic relation \eqref{eq:xi_equals_bdf2_eta} and the discrete AB2 predictors $a_h^{\star,n+1}=2a_h^n-a_h^{n-1}$.

For the Stokes subproblem, we test \eqref{eq:fluid-discrete-bdf2-fixed} with $\boldsymbol{v}_h=\boldsymbol{u}_h^{n+1}$ and
$q_h=p_{f,h}^{n+1}$. By the cancellation of the pressure divergence terms, we obtain:
\begin{align}
&\rho_f\big(\delta_t\boldsymbol{u}_h^{n+1},\boldsymbol{u}_h^{n+1}\big)_f
+2\mu_f\|\boldsymbol{D}(\boldsymbol{u}_h^{n+1})\|_{L^2 (\Omega_f)}^2
+\gamma\|\boldsymbol{P}_f\boldsymbol{u}_h^{n+1}\|_{L^2 (\Gamma)}^2
+L\|\boldsymbol{u}_h^{n+1}\cdot\boldsymbol{n}_f\|_{L^2 (\Gamma)}^2
\notag\\
&\qquad
=
\gamma\big\langle
\boldsymbol{P}_f\boldsymbol{\xi}_h^{\star,n+1},
\boldsymbol{P}_f\boldsymbol{u}_h^{n+1}
\big\rangle_\Gamma
+\big\langle
L(\boldsymbol{u}_h^{\star,n+1}\cdot\boldsymbol{n}_f)-\phi_h^{\star,n+1},
\boldsymbol{u}_h^{n+1}\cdot\boldsymbol{n}_f
\big\rangle_\Gamma.
\label{eq:stability_fluid_test}
\end{align}
Applying \eqref{eq:bdf2_energy_identity} to $\rho_f\left(\delta_t \boldsymbol{u}_h^{n+1}, \boldsymbol{u}_h^{n+1}\right)_f$ term, yields: 
\begin{align}
&\frac{1}{\Delta t}\big(E_u^{n+1}-E_u^n\big)
+\frac{\rho_f}{4\Delta t}
\|\boldsymbol{u}_h^{n+1}-2\boldsymbol{u}_h^n+\boldsymbol{u}_h^{n-1}\|_{L^2 (\Omega_f)}^2
\notag\\
&\qquad
+2\mu_f\|\boldsymbol{D}(\boldsymbol{u}_h^{n+1})\|_{L^2 (\Omega_f)}^2
+\gamma\|\boldsymbol{P}_f\boldsymbol{u}_h^{n+1}\|_{L^2 (\Gamma)}^2
+L\|\boldsymbol{u}_h^{n+1}\cdot\boldsymbol{n}_f\|_{L^2 (\Gamma)}^2
\notag\\
&=
\gamma\big\langle
\boldsymbol{P}_f\boldsymbol{\xi}_h^{\star,n+1},
\boldsymbol{P}_f\boldsymbol{u}_h^{n+1}
\big\rangle_\Gamma
+\big\langle
L(\boldsymbol{u}_h^{\star,n+1}\cdot\boldsymbol{n}_f)-\phi_h^{\star,n+1},
\boldsymbol{u}_h^{n+1}\cdot\boldsymbol{n}_f
\big\rangle_\Gamma.
\label{eq:stability_fluid_energy}
\end{align}

Next, we test the Biot subproblem \eqref{eq:poro-discrete-bdf2-fixed} with the choice of $\boldsymbol{v}_{p,h}=\delta_t\boldsymbol{\eta}_h^{n+1}$, $\boldsymbol{\zeta}_h=\boldsymbol{\xi}_h^{n+1}$, $
\psi_h=\phi_h^{n+1}
$. Since $
\delta_t\boldsymbol{\eta}_h^{n+1}=\boldsymbol{\xi}_h^{n+1}$, the Biot coupling terms canceled:
$$
-\alpha\left(\phi_h^{n+1}, \nabla \cdot \boldsymbol{\xi}_h^{n+1}\right)_p+\alpha\left(\nabla \cdot \boldsymbol{\xi}_h^{n+1}, \phi_h^{n+1}\right)_p=0,
$$
and the following interface kinetic terms canceled:
$$
\left\langle\phi_h^{n+1}, \boldsymbol{\xi}_h^{n+1} \cdot \boldsymbol{n}_p\right\rangle_{\Gamma}-\left\langle\boldsymbol{\xi}_h^{n+1} \cdot \boldsymbol{n}_p, \phi_h^{n+1}\right\rangle_{\Gamma}=0.
$$
Thus, we obtain:
\begin{align}
\label{eq:stability_biot_test}
&\quad\rho_p\big(\delta_t\boldsymbol{\xi}_h^{n+1},\boldsymbol{\xi}_h^{n+1}\big)_p
+a_e(\boldsymbol{\eta}_h^{n+1},\boldsymbol{\xi}_h^{n+1})
+C_0\big(\delta_t\phi_h^{n+1},\phi_h^{n+1}\big)_p
+\|K^{1/2}\nabla\phi_h^{n+1}\|_{L^2 (\Omega_p)}^2
\nonumber\\
&\quad
+\gamma\|\boldsymbol{P}_p\boldsymbol{\xi}_h^{n+1}\|_{L^2 (\Gamma)}^2
+\|\boldsymbol{\xi}_h^{n+1}\cdot\boldsymbol{n}_p\|_{L^2 (\Gamma)}^2
+\frac1L\|\phi_h^{n+1}\|_{L^2 (\Gamma)}^2
\nonumber\\
\qquad&=
\gamma\big\langle
\boldsymbol{P}_p\boldsymbol{u}_h^{\star,n+1},
\boldsymbol{P}_p\boldsymbol{\xi}_h^{n+1}
\big\rangle_\Gamma
+\big\langle
\boldsymbol{\xi}_h^{\star,n+1}\cdot\boldsymbol{n}_p,
\boldsymbol{\xi}_h^{n+1}\cdot\boldsymbol{n}_p
\big\rangle_\Gamma\nonumber
\nonumber\\
&\quad+\big\langle
-\boldsymbol{u}_h^{\star,n+1}\cdot\boldsymbol{n}_p+\tfrac1L\phi_h^{\star,n+1},
\phi_h^{n+1}
\big\rangle_\Gamma.
\end{align}
Using \eqref{eq:bdf2_energy_identity} for
$\rho_p\left(\delta_t \boldsymbol{\xi}_h^{n+1}, \boldsymbol{\xi}_h^{n+1}\right)_p$ and $C_0\left(\delta_t \phi_h^{n+1}, \phi_h^{n+1}\right)_p$, and
convert the elastic term 
$a_e(\boldsymbol{\eta}_h^{n+1},\delta_t\boldsymbol{\eta}_h^{n+1})$ into elastic displacement energy using BDF2 kinematics, we obtain:
\begin{align}
&\qquad\frac{1}{\Delta t}(E_\xi^{n+1}-E_\xi^n)
+\frac{\rho_p}{4\Delta t}
\|\boldsymbol{\xi}_h^{n+1}-2\boldsymbol{\xi}_h^n+\boldsymbol{\xi}_h^{n-1}\|_{L^2 (\Omega_p)}^2
\notag\\
&\quad
+\frac{1}{\Delta t}(E_\phi^{n+1}-E_\phi^n)
+\frac{C_0}{4\Delta t}
\|\phi_h^{n+1}-2\phi_h^n+\phi_h^{n-1}\|_{L^2 (\Omega_p)}^2
\notag\\
&\quad
+\frac{1}{\Delta t}(E_\eta^{n+1}-E_\eta^n)
+\frac{1}{4\Delta t}
\|\boldsymbol{\eta}_h^{n+1}-2\boldsymbol{\eta}_h^n+\boldsymbol{\eta}_h^{n-1}\|_{a_e}^2
\notag\\
&\quad
+\|K^{1/2}\nabla\phi_h^{n+1}\|_{L^2 (\Omega_p)}^2
+\gamma\|\boldsymbol{P}_p\boldsymbol{\xi}_h^{n+1}\|_{L^2 (\Gamma)}^2
+\|\boldsymbol{\xi}_h^{n+1}\cdot\boldsymbol{n}_p\|_{L^2 (\Gamma)}^2
+\frac1L\|\phi_h^{n+1}\|_{L^2 (\Gamma)}^2
\notag\\
\quad&=
\gamma\big\langle
\boldsymbol{P}_p\boldsymbol{u}_h^{\star,n+1},
\boldsymbol{P}_p\boldsymbol{\xi}_h^{n+1}
\big\rangle_\Gamma
+\big\langle
\boldsymbol{\xi}_h^{\star,n+1}\cdot\boldsymbol{n}_p,
\boldsymbol{\xi}_h^{n+1}\cdot\boldsymbol{n}_p
\big\rangle_\Gamma
\notag\\
&\quad+\big\langle
-\boldsymbol{u}_h^{\star,n+1}\cdot\boldsymbol{n}_p+\tfrac1L\phi_h^{\star,n+1},
\phi_h^{n+1}
\big\rangle_\Gamma.
\label{eq:stability_biot_energy}
\end{align}
Now, we sum the fluid and Biot contributions, namely \eqref{eq:stability_fluid_energy} and \eqref{eq:stability_biot_energy}, we obtain:
\begin{align}
&\frac{1}{\Delta t}\big(\mathcal E^{n+1}-\mathcal E^n\big)
+\mathcal D_{\mathrm{bulk}}^{n+1}
+\mathcal D_\Gamma^{n+1}
=
\mathcal I^{n+1},
\label{eq:stability_sum_final}
\end{align}
where \(\mathcal D_{\mathrm{bulk}}^{n+1}\) and \(\mathcal D_\Gamma^{n+1}\) are
defined in \eqref{eq:def_bulk_diss_final} and \eqref{eq:def_DGamma}, and
$\mathcal{I}^{n+1}$ corresponds to the sum of all the extrapolation interface terms:
\begin{align}
\mathcal{I}^{n+1}
&:=
\gamma\big\langle
\boldsymbol{P}_f\boldsymbol{\xi}_h^{\star,n+1},
\boldsymbol{P}_f\boldsymbol{u}_h^{n+1}
\big\rangle_\Gamma
+\big\langle
L(\boldsymbol{u}_h^{\star,n+1}\cdot\boldsymbol{n}_f)-\phi_h^{\star,n+1},
\boldsymbol{u}_h^{n+1}\cdot\boldsymbol{n}_f
\big\rangle_\Gamma
\notag\\
&\quad
+\gamma\big\langle
\boldsymbol{P}_p\boldsymbol{u}_h^{\star,n+1},
\boldsymbol{P}_p\boldsymbol{\xi}_h^{n+1}
\big\rangle_\Gamma
+\big\langle
\boldsymbol{\xi}_h^{\star,n+1}\cdot\boldsymbol{n}_p,
\boldsymbol{\xi}_h^{n+1}\cdot\boldsymbol{n}_p
\big\rangle_\Gamma
\notag\\
&\quad
+\big\langle
-\boldsymbol{u}_h^{\star,n+1}\cdot\boldsymbol{n}_p+\tfrac1L\phi_h^{\star,n+1},
\phi_h^{n+1}
\big\rangle_\Gamma .
\label{eq:def_I}
\end{align}

We now derive a closed stability estimate by exploiting the
exact AB2 residual identity:
\begin{equation}\label{eq:ab2_residual_identity}
a_h^{\star,n+1}-a_h^{n+1}
=
-\big(a_h^{n+1}-2a_h^n+a_h^{n-1}\big).
\end{equation}

By the uniform positive definiteness of $K$ assumed in
\eqref{eq:K_coercive_intro}, the same coercivity holds in particular on the
discrete poroelastic pressure space $Q_{p,h}\subset H^1(\Omega_p)$:
\begin{equation}\label{eq:K_coercive}
(K\nabla \psi_h,\nabla \psi_h)_p \ge k_0 \|\nabla \psi_h\|_{L^2(\Omega_p)}^2
\qquad
\forall \psi_h\in Q_{p,h}.
\end{equation}

Next, by defining the residual variables:
\begin{equation}\label{eq:second_differences}
\boldsymbol r_u^{n+1}:=\boldsymbol u_h^{n+1}-2\boldsymbol u_h^n+\boldsymbol u_h^{n-1},
\quad
\boldsymbol r_\xi^{n+1}:=\boldsymbol \xi_h^{n+1}-2\boldsymbol \xi_h^n+\boldsymbol \xi_h^{n-1},
\quad
r_\phi^{n+1}:=\phi_h^{n+1}-2\phi_h^n+\phi_h^{n-1},
\end{equation}
and from \eqref{AB2}, we have:
\begin{equation}\label{eq:ab2_residual_variables}
\boldsymbol u_h^{\star,n+1}=\boldsymbol u_h^{n+1}-\boldsymbol r_u^{n+1},
\quad
\boldsymbol \xi_h^{\star,n+1}=\boldsymbol \xi_h^{n+1}-\boldsymbol r_\xi^{n+1},
\quad
\phi_h^{\star,n+1}=\phi_h^{n+1}-r_\phi^{n+1}.
\end{equation}
Substituting \eqref{eq:ab2_residual_variables} into \(\mathcal I^{n+1}\), we obtain:
$$
\mathcal I^{n+1}=\mathcal I_{\mathrm{cur}}^{n+1}+\mathcal R_\Gamma^{n+1},
$$
where $\mathcal I_{\mathrm{cur}}^{n+1}$ denotes the current-step contribution:
\begin{align}
\mathcal I_{\mathrm{cur}}^{n+1}
&=
\gamma\left\langle \boldsymbol P_f\boldsymbol \xi_h^{n+1},\boldsymbol P_f\boldsymbol u_h^{n+1}\right\rangle_\Gamma
+\left\langle L(\boldsymbol u_h^{n+1}\cdot \boldsymbol n_f)-\phi_h^{n+1},
\boldsymbol u_h^{n+1}\cdot \boldsymbol n_f\right\rangle_\Gamma
\notag\\
&\quad
+\gamma\left\langle \boldsymbol P_p\boldsymbol u_h^{n+1},\boldsymbol P_p\boldsymbol \xi_h^{n+1}\right\rangle_\Gamma
+\left\langle \boldsymbol \xi_h^{n+1}\cdot \boldsymbol n_p,
\boldsymbol \xi_h^{n+1}\cdot \boldsymbol n_p\right\rangle_\Gamma
\notag\\
&\quad
+\left\langle -\boldsymbol u_h^{n+1}\cdot \boldsymbol n_p+\frac1L\phi_h^{n+1},
\phi_h^{n+1}\right\rangle_\Gamma,
\label{eq:Icur_alt}
\end{align}
and $\mathcal R_\Gamma^{n+1}$ denotes the residual part:
\begin{align}
\mathcal R_\Gamma^{n+1}
&=
-\gamma\left\langle \boldsymbol P_f\boldsymbol r_\xi^{n+1},\boldsymbol P_f\boldsymbol u_h^{n+1}\right\rangle_\Gamma
-\gamma\left\langle \boldsymbol P_p\boldsymbol r_u^{n+1},\boldsymbol P_p\boldsymbol \xi_h^{n+1}\right\rangle_\Gamma
\notag\\
&\quad
-L\left\langle \boldsymbol r_u^{n+1}\cdot \boldsymbol n_f,\boldsymbol u_h^{n+1}\cdot \boldsymbol n_f\right\rangle_\Gamma
+\left\langle r_\phi^{n+1},\boldsymbol u_h^{n+1}\cdot \boldsymbol n_f\right\rangle_\Gamma
\notag\\
&\quad
-\left\langle \boldsymbol r_\xi^{n+1}\cdot \boldsymbol n_p,\boldsymbol \xi_h^{n+1}\cdot \boldsymbol n_p\right\rangle_\Gamma
+\left\langle \boldsymbol r_u^{n+1}\cdot \boldsymbol n_p,\phi_h^{n+1}\right\rangle_\Gamma
-\frac1L\left\langle r_\phi^{n+1},\phi_h^{n+1}\right\rangle_\Gamma.
\label{eq:def_residual_alt}
\end{align}
Since we have normal vector
$\boldsymbol n_p=-\boldsymbol n_f$ and tangential projection operator satisfies $\boldsymbol P_f=\boldsymbol P_p$ on the interface $\Gamma$, we derive:
\begin{align}
\mathcal I_{\mathrm{cur}}^{n+1}
&=
2\gamma\left\langle \boldsymbol P_f\boldsymbol \xi_h^{n+1},\boldsymbol P_f\boldsymbol u_h^{n+1}\right\rangle_\Gamma
+
L\|\boldsymbol u_h^{n+1}\cdot \boldsymbol n_f\|_{L^2(\Gamma)}^2
+\|\boldsymbol \xi_h^{n+1}\cdot \boldsymbol n_p\|_{L^2(\Gamma)}^2
+\frac1L\|\phi_h^{n+1}\|_{L^2(\Gamma)}^2.
\label{eq:Icur_simplified_alt}
\end{align}
Based on $2\langle a,b\rangle\le \|a\|^2+\|b\|^2$ and \eqref{eq:def_DGamma}, we have:
\begin{equation}\label{eq:Icur_bound_alt}
\mathcal I_{\mathrm{cur}}^{n+1}
\le
\mathcal D_\Gamma^{n+1}.
\end{equation}
Hence, from \eqref{eq:stability_sum_final}, we obtain:
\begin{equation}\label{eq:energy_residual_identity_alt}
\frac{1}{\Delta t}\big(\mathcal E^{n+1}-\mathcal E^n\big)
+\mathcal D_{\mathrm{bulk}}^{n+1}
\le
\mathcal R_\Gamma^{n+1}.
\end{equation}

Next, we estimate each term in the residual \(\mathcal R_\Gamma^{n+1}\) by using Cauchy-Schwarz,
the discrete trace inequalities \eqref{eq:disc_trace_fluid}-\eqref{eq:disc_trace_inverse_poro},
and Young's inequality.
For the first tangential residual term, we have:
\begin{align}
\gamma\big|\left\langle \boldsymbol P_f\boldsymbol r_\xi^{n+1},\boldsymbol P_f\boldsymbol u_h^{n+1}\right\rangle_\Gamma\big|
&\le
\gamma
\|\boldsymbol P_f\boldsymbol r_\xi^{n+1}\|_{L^2(\Gamma)}
\|\boldsymbol P_f\boldsymbol u_h^{n+1}\|_{L^2(\Gamma)}
\notag\\
&\le
C\gamma h^{-1/2}\|\boldsymbol r_\xi^{n+1}\|_{L^2(\Omega_p)}
\Big(
h^{-1/2}\|\boldsymbol u_h^{n+1}\|_{L^2(\Omega_f)}
+
h^{1/2}\|\boldsymbol D(\boldsymbol u_h^{n+1})\|_{L^2(\Omega_f)}
\Big)
\notag\\
&=C \gamma h^{-1}\left\|\boldsymbol{r}_{\xi}^{n+1}\right\|_{L^2(\Omega_p)}
\left\|\boldsymbol{u}_h^{n+1}\right\|_{L^2(\Omega_f)}
+C \gamma\left\|\boldsymbol{r}_{\xi}^{n+1}\right\|_{L^2(\Omega_p)}
\left\|\boldsymbol{D}\left(\boldsymbol{u}_h^{n+1}\right)\right\|_{L^2(\Omega_f)} \notag\\
&\le
\frac{\rho_p}{32\Delta t}\|\boldsymbol r_\xi^{n+1}\|_{L^2(\Omega_p)}^2
+C\frac{\gamma^2\Delta t}{\rho_p h^2}\|\boldsymbol u_h^{n+1}\|_{L^2(\Omega_f)}^2
+C\frac{\gamma^2\Delta t}{\rho_p}\|\boldsymbol D(\boldsymbol u_h^{n+1})\|_{L^2(\Omega_f)}^2
\notag\\
&\le
\frac{\rho_p}{32\Delta t}\|\boldsymbol r_\xi^{n+1}\|_{L^2(\Omega_p)}^2
+C\frac{\gamma^2\Delta t}{\rho_p h^2}\|\boldsymbol u_h^{n+1}\|_{L^2(\Omega_f)}^2
+C\frac{\gamma^2\Delta t}{\rho_p}\|\boldsymbol D(\boldsymbol u_h^{n+1})\|_{L^2(\Omega_f)}^2
\notag\\
&\quad+\frac{\mu_f}{8}\|\boldsymbol D(\boldsymbol u_h^{n+1})\|_{L^2(\Omega_f)}^2.
\label{eq:R1_alt}
\end{align}
Similarly, we have:
\begin{align}
\gamma\big|\left\langle \boldsymbol P_p\boldsymbol r_u^{n+1},\boldsymbol P_p\boldsymbol \xi_h^{n+1}\right\rangle_\Gamma\big|
&\le
\frac{\rho_f}{32\Delta t}\|\boldsymbol r_u^{n+1}\|_{L^2(\Omega_f)}^2
+
C\frac{\gamma^2\Delta t}{\rho_f h^2}\|\boldsymbol \xi_h^{n+1}\|_{L^2(\Omega_p)}^2,
\label{eq:R2_alt}
\\
L\big|\left\langle \boldsymbol r_u^{n+1}\cdot \boldsymbol n_f,\boldsymbol u_h^{n+1}\cdot \boldsymbol n_f\right\rangle_\Gamma\big|
&\le
\frac{\rho_f}{32\Delta t}\|\boldsymbol r_u^{n+1}\|_{L^2(\Omega_f)}^2
+\frac{\mu_f}{8}\|\boldsymbol D(\boldsymbol u_h^{n+1})\|_{L^2(\Omega_f)}^2
\notag\\
&\quad
+C\frac{L^2\Delta t}{\rho_f h^2}\|\boldsymbol u_h^{n+1}\|_{L^2(\Omega_f)}^2
+C\frac{L^2\Delta t}{\rho_f}\|\boldsymbol D(\boldsymbol u_h^{n+1})\|_{L^2(\Omega_f)}^2,
\label{eq:R3_alt}
\\
\big|\left\langle r_\phi^{n+1},\boldsymbol u_h^{n+1}\cdot \boldsymbol n_f\right\rangle_\Gamma\big|
&\le
\frac{C_0}{32\Delta t}\|r_\phi^{n+1}\|_{L^2(\Omega_p)}^2
+\frac{\mu_f}{8}\|\boldsymbol D(\boldsymbol u_h^{n+1})\|_{L^2(\Omega_f)}^2
\notag\\
&\quad
+C\frac{\Delta t}{C_0 h^2}\|\boldsymbol u_h^{n+1}\|_{L^2(\Omega_f)}^2
+C\frac{\Delta t}{C_0}\|\boldsymbol D(\boldsymbol u_h^{n+1})\|_{L^2(\Omega_f)}^2,
\label{eq:R4_alt}
\\
\big|\left\langle \boldsymbol r_\xi^{n+1}\cdot \boldsymbol n_p,\boldsymbol \xi_h^{n+1}\cdot \boldsymbol n_p\right\rangle_\Gamma\big|
&\le
\frac{\rho_p}{32\Delta t}\|\boldsymbol r_\xi^{n+1}\|_{L^2(\Omega_p)}^2
+
C\frac{\Delta t}{\rho_p h^2}\|\boldsymbol \xi_h^{n+1}\|_{L^2(\Omega_p)}^2,
\label{eq:R5_alt}
\\
\big|\left\langle \boldsymbol r_u^{n+1}\cdot \boldsymbol n_p,\phi_h^{n+1}\right\rangle_\Gamma\big|
&\le
\frac{\rho_f}{32\Delta t}\|\boldsymbol r_u^{n+1}\|_{L^2(\Omega_f)}^2
+
C\frac{\Delta t}{\rho_f h^2}\|\phi_h^{n+1}\|_{L^2(\Omega_p)}^2
+
\frac{k_0}{8}\|\nabla \phi_h^{n+1}\|_{L^2(\Omega_p)}^2,
\label{eq:R6_alt}
\\
\frac1L\big|\left\langle r_\phi^{n+1},\phi_h^{n+1}\right\rangle_\Gamma\big|
&\le
\frac{C_0}{32\Delta t}\|r_\phi^{n+1}\|_{L^2(\Omega_p)}^2
+
C\frac{\Delta t}{C_0L^2 h^2}\|\phi_h^{n+1}\|_{L^2(\Omega_p)}^2
+
\frac{k_0}{8}\|\nabla \phi_h^{n+1}\|_{L^2(\Omega_p)}^2,
\label{eq:R7_alt}
\end{align}
where, all those fractions, such as $1/32$ and $1/8$ are chosen to ensure later that those terms can be absorbed into the left-hand side.

Combining \eqref{eq:R1_alt}-\eqref{eq:R7_alt}, and using
\eqref{eq:K_coercive} together with the definition of \(\mathcal E^{n+1}\),
we obtain:
\begin{align}
\mathcal R_\Gamma^{n+1}
&\le
\frac{1}{4}\mathcal D_{\mathrm{bulk}}^{n+1}
+
C_1\Delta t\,\|\boldsymbol D(\boldsymbol u_h^{n+1})\|_{L^2(\Omega_f)}^2
+
C_2\frac{\Delta t}{h^2}\,\mathcal E^{n+1},
\label{eq:R_total_alt}
\end{align}
where \(C_1,C_2>0\) depend only on the physical/interface parameters and the trace constants, but are independent of \(h\) and \(\Delta t\).

To bound the lower-order term $C_1 \Delta t\,\|\boldsymbol D(\boldsymbol u_h^{n+1})\|_{L^2(\Omega_f)}^2$
on the right-hand side of \eqref{eq:R_total_alt}, we absorb it into the viscous contribution
$2\mu_f\|\boldsymbol D(\boldsymbol u_h^{n+1})\|_{L^2(\Omega_f)}^2$, contained in the bulk dissipation \(\mathcal D_{\mathrm{bulk}}^{n+1}\).
Here $C_1>0$ collects all residual contributions proportional to \(\Delta t\,\|\boldsymbol D(\boldsymbol u_h^{n+1})\|_{L^2(\Omega_f)}^2\).
In particular, up to a generic multiplicative constant independent of \(h\) and \(\Delta t\), we may take:
\begin{equation}\label{eq:def_C1}
C_1
=
C\left(
\frac{\gamma^2}{\rho_p}
+
\frac{L^2}{\rho_f}
+
\frac{1}{C_0}
\right).
\end{equation}
These terms arise from the residual estimates
\eqref{eq:R1_alt}, \eqref{eq:R3_alt}, and \eqref{eq:R4_alt}.

Next, we choose \(\Delta t_0>0\) such that $C_1\Delta t_0\le \frac{\mu_f}{2}$.
Namely,
\begin{equation}\label{eq:def_dt0}
\Delta t_0
:=
\min\left\{1,\frac{\mu_f}{2C_1}\right\}.
\end{equation}
Then, for every \(\Delta t\le \Delta t_0\),
\begin{equation}\label{eq:small_dt_absorb_corrected}
C_1\Delta t\,\|\boldsymbol D(\boldsymbol u_h^{n+1})\|_{L^2(\Omega_f)}^2
\le
\frac{\mu_f}{2}\|\boldsymbol D(\boldsymbol u_h^{n+1})\|_{L^2(\Omega_f)}^2
\le
\frac14\,\mathcal D_{\mathrm{bulk}}^{n+1}.
\end{equation}

Next, \(C_2>0\) denotes a constant collecting all contributions in
\eqref{eq:R1_alt}--\eqref{eq:R7_alt} that are bounded by
\[
\frac{\Delta t}{h^2}\,\mathcal E^{n+1}.
\]
Here we use that the current \(L^2\)-terms
\(\|\boldsymbol u_h^{n+1}\|_{L^2(\Omega_f)}^2\),
\(\|\boldsymbol\xi_h^{n+1}\|_{L^2(\Omega_p)}^2\),
\(\|\phi_h^{n+1}\|_{L^2(\Omega_p)}^2\)
are controlled by the BDF2 energy \(\mathcal E^{n+1}\).
More precisely, after using the definition of \(\mathcal E^{n+1}\), one may
take
\begin{equation}\label{eq:def_C2}
C_2
=
C\left(
\frac{\gamma^2}{\rho_f\rho_p}
+\frac{L^2}{\rho_f^2}
+\frac{1}{\rho_p^2}
+\frac{1}{\rho_f C_0}
+\frac{1}{C_0^2L^2}
\right),
\end{equation}
again up to a generic constant \(C>0\) independent of \(h\) and \(\Delta t\).

Substituting \eqref{eq:R_total_alt} and
\eqref{eq:small_dt_absorb_corrected} into
\eqref{eq:energy_residual_identity_alt}, we obtain
\begin{equation}\label{eq:one_step_cfl_pre_corrected}
\frac{1}{\Delta t}\big(\mathcal E^{n+1}-\mathcal E^n\big)
+\frac12\,\mathcal D_{\mathrm{bulk}}^{n+1}
\le
C_2\frac{\Delta t}{h^2}\,\mathcal E^{n+1}.
\end{equation}
Multiplying by $\Delta t$ gives
\begin{equation}\label{eq:one_step_cfl_corrected}
\mathcal E^{n+1}-\mathcal E^n
+\frac{\Delta t}{2}\mathcal D_{\mathrm{bulk}}^{n+1}
\le
C_2\frac{\Delta t^2}{h^2}\,\mathcal E^{n+1}.
\end{equation}

Now choose $c_\ast>0$ such that
\begin{equation}\label{eq:def_cstar}
C_2 c_\ast \le \frac14.
\end{equation}
Then whenever
${\Delta t} \le c_\ast {h^2}$, we have:
$$
C_2\frac{\Delta t^2}{h^2}\le C_2 c_\ast \Delta t \le \frac{\Delta t}{4},
$$
and therefore \eqref{eq:one_step_cfl_corrected} yields the following inequality:
\begin{equation}\label{eq:one_step_cfl_rearranged}
\Bigl(1-\frac{\Delta t}{4}\Bigr)\mathcal E^{n+1}
+\frac{\Delta t}{2}\mathcal D_{\mathrm{bulk}}^{n+1}
\le
\mathcal E^n.
\end{equation}

Since when $\Delta t\le \Delta t_0\le 1$, we have:
$$
\Bigl(1-\frac{\Delta t}{4}\Bigr)^{-1}\le 1+\frac{\Delta t}{2},
$$
by dropping the nonnegative dissipation term in
\eqref{eq:one_step_cfl_rearranged}, we obtain:
\begin{equation}\label{eq:energy_growth_step}
\mathcal E^{n+1}
\le
\Bigl(1+\frac{\Delta t}{2}\Bigr)\mathcal E^n.
\end{equation}
Iterating the above inequality gives:
$$
\mathcal E^{m+1}
\le
\Bigl(1+\frac{\Delta t}{2}\Bigr)^m \mathcal E^1
\le
\exp\!\left(\frac{m\Delta t}{2}\right)\mathcal E^1
\le
e^{T/2}\mathcal E^1.
$$

Finally, summing \eqref{eq:one_step_cfl_corrected} from $n=1$ to $m$, we obtain
\begin{equation}\label{eq:sum_one_step_cfl}
\mathcal E^{m+1}-\mathcal E^1
+\frac{\Delta t}{2}\sum_{n=1}^{m}\mathcal D_{\mathrm{bulk}}^{n+1}
\le
C_2\frac{\Delta t^2}{h^2}\sum_{n=1}^{m}\mathcal E^{n+1}.
\end{equation}
Given the bound
$\mathcal E^{n+1}\le e^{T/2}\mathcal E^1$, we can derive:
$$
C_2\frac{\Delta t^2}{h^2}\sum_{n=1}^{m}\mathcal E^{n+1}
\le
C_2 c_\ast \Delta t\sum_{n=1}^{m}\mathcal E^{n+1}
\le
\frac14\,\Delta t\sum_{n=1}^{m}\mathcal E^{n+1}
\le
\frac{T}{4}e^{T/2}\mathcal E^1.
$$
Therefore
$$
\mathcal E^{m+1}
+\frac{\Delta t}{2}\sum_{n=1}^{m}\mathcal D_{\mathrm{bulk}}^{n+1}
\le
\left(1+\frac{T}{4}e^{T/2}\right)\mathcal E^1
\le
C_T\,\mathcal E^1,
\qquad
m\ge 1.
$$
where
\begin{equation}\label{eq:def_CT_final}
C_T:=1+\frac{T}{4}e^{T/2}.
\end{equation}
\end{proof}

\begin{remark}
\label{rem:cfl_stability_corrected}
The constant \(C_1\) collects the coefficients of the terms
\(\Delta t\|\boldsymbol D(\boldsymbol u_h^{n+1})\|_{L^2(\Omega_f)}^2\)
arising in the residual estimates, while \(C_2\) collects the coefficients of
the terms \(\frac{\Delta t}{h^2}\mathcal E^{n+1}\). The threshold
\(\Delta t_0\) controls the absorption of lower-order viscous contributions,
whereas the condition \(\Delta t\le c_\ast h^2\) is a
CFL-type restriction needed to control the explicit AB2 interface treatment.
\end{remark}

\section{Error estimate}
\label{sec:error}

In this section we derive an a priori error estimate for the fully discrete
fixed-domain BDF2-AB2 partitioned scheme
\eqref{eq:fluid-discrete-bdf2-fixed}-\eqref{eq:poro-discrete-bdf2-fixed}.
The proof is based on the CFL stability estimate of
Theorem~\ref{thm:stability_cfl}, together with suitable projection operators,
local time-discretization defect bounds in \(L^2\)-in-time form, and a
discrete Gronwall argument. In particular, the discrete trace inequalities and the parabolic CFL condition are assumed to hold.

\subsection{Regularity assumptions and projection operators}
\label{subsec:error_assumptions}

Let
$(\boldsymbol u,p_f,\boldsymbol\eta,\boldsymbol\xi,\phi)$ 
be the exact solution of
\eqref{eq:stokes-biot}-\eqref{eq:interface_conditions}, and let
$
(\boldsymbol u_h^n,p_{f,h}^n,\boldsymbol\eta_h^n,\boldsymbol\xi_h^n,\phi_h^n)
$
be the discrete solution generated by
\eqref{eq:fluid-discrete-bdf2-fixed}-\eqref{eq:poro-discrete-bdf2-fixed}.
For \(n\ge0\), we define:
\[
\boldsymbol u^n:=\boldsymbol u(t_n),\quad
p_f^n:=p_f(t_n),\quad
\boldsymbol\eta^n:=\boldsymbol\eta(t_n),\quad
\boldsymbol\xi^n:=\boldsymbol\xi(t_n),\quad
\phi^n:=\phi(t_n).
\]

We assume the following regularity of the exact solution:
\begin{align}
\boldsymbol u
&\in W^{2,\infty}(0,T;H^{k+1}(\Omega_f)^d)\cap H^3(0,T;L^2(\Omega_f)^d),
\label{eq:reg_u_error}
\\
p_f
&\in L^\infty(0,T;H^k(\Omega_f)),
\label{eq:reg_p_error}
\\
\boldsymbol\eta
&\in W^{2,\infty}(0,T;H^{k+1}(\Omega_p)^d)\cap H^3(0,T;H^{k+1}(\Omega_p)^d),
\label{eq:reg_eta_error}
\\
\boldsymbol\xi=\partial_t\boldsymbol\eta
&\in W^{2,\infty}(0,T;H^{k+1}(\Omega_p)^d)\cap H^3(0,T;L^2(\Omega_p)^d),
\label{eq:reg_xi_error}
\\
\phi
&\in W^{2,\infty}(0,T;H^{k+1}(\Omega_p))\cap H^3(0,T;L^2(\Omega_p)).
\label{eq:reg_phi_error}
\end{align}
In particular, we assume $H^3$ in time for corresponding variables is to obtain the second order in time estimate.

We assume that the Stokes finite element pair $(\boldsymbol V_{f,h},Q_{f,h})$
is a stable mixed pair, i.e.\ there exists a constant $\beta_0>0$,
independent of $h$, such that the discrete inf-sup (Ladyzhenskaya--Babu\v{s}ka--Brezzi) condition
\begin{equation}\label{eq:discrete_infsup}
\sup_{\boldsymbol v_h\in \boldsymbol V_{f,h}\setminus\{0\}}
\frac{(q_h,\nabla\!\cdot\!\boldsymbol v_h)_f}{\|\boldsymbol v_h\|_{H^1(\Omega_f)}}
\ge
\beta_0\|q_h\|_{L^2(\Omega_f)}
\qquad
\forall q_h\in Q_{f,h}
\end{equation}
holds. For velocity spaces with homogeneous Dirichlet data only on
\(\Sigma_{D,f}\) and free or interface traces on \(\Gamma\), we assume the
corresponding full-\(L^2\) inf-sup condition. This assumption is compatible
with standard Stokes-stable pairs, such as Taylor--Hood
\(\mathbb P_k/\mathbb P_{k-1}\) elements, under the usual shape-regularity
and boundary-condition hypotheses. The Fortin operator used below is assumed
to satisfy the stated divergence-orthogonality and approximation properties.

{\textbf{Fluid projection.}}
To eliminate the divergence residual in the fluid error equation, we assume that the
Stokes finite element pair $(\boldsymbol V_{f,h},Q_{f,h})$ admits a Fortin operator
\(\Pi_u:\boldsymbol V_f\to \boldsymbol V_{f,h}\) such that
\begin{equation}
(\nabla\!\cdot(\boldsymbol v-\Pi_u\boldsymbol v),q_h)_f=0
\qquad \forall q_h\in Q_{f,h},
\label{eq:Fortin_property}
\end{equation}
and
\begin{equation}
\|\Pi_u \boldsymbol v\|_{L^2(\Omega_f)}
\le C\|\boldsymbol v\|_{L^2(\Omega_f)}
\qquad \forall \boldsymbol v\in \boldsymbol V_f.
\label{eq:Fortin_L2_stability}
\end{equation}
In addition, we assume that \(\Pi_u\) satisfies the approximation and trace estimates:
\begin{align}
\left\|\boldsymbol{v}-\Pi_u \boldsymbol{v}\right\|_{L^2\left(\Omega_f\right)}+h\left\|\boldsymbol{v}-\Pi_u \boldsymbol{v}\right\|_{H^1\left(\Omega_f\right)} &\leq C h^{k+1}\|\boldsymbol{v}\|_{H^{k+1}\left(\Omega_f\right)},
\label{eq:Fortin_approx}
\\
\left\|\boldsymbol{v}-\Pi_u \boldsymbol{v}\right\|_{L^2(\Gamma)} 
&\leq C h^{k+\frac{1}{2}}\|\boldsymbol{v}\|_{H^{k+1}\left(\Omega_f\right)}.
\label{eq:Fortin_trace_approx}
\end{align}
\textcolor{black}{
We refer to \cite{Brezzi-Falk:1991,Girault-Scott:2003,Falk:2008,Diening-Storn-Tscherpel:2022} for existence of \(\Pi_u\) for the family of Taylor-Hood elements.}

Next, we define the fluid pressure projection $\Pi_p$ as the $L^2$ projection onto the discrete pressure space. Let \(\Pi_p:Q_f\to Q_{f,h}\) denote the \(L^2\)-projection:
\begin{equation}
(q-\Pi_p q,\chi_h)_f=0
\qquad \forall \chi_h\in Q_{f,h},
\label{eq:pressure_projection}
\end{equation}
with the corresponding approximation estimate
\begin{equation}
\left\|q-\Pi_p q\right\|_{L^2\left(\Omega_f\right)} \leq C h^k\|q\|_{H^k\left(\Omega_f\right)}
\label{eq:proj_p_error}
\end{equation}
We also assume the discrete Korn inequality on \(\boldsymbol V_{f,h}\) \cite{Brenner-Scott:book}. Namely,
there exists a constant $C_K>0$, independent of \(h\), such that the $H^1$ norm is controlled by the $L^2$ norm of symmetric gradient
\begin{equation}
\left\|\boldsymbol{v}_h\right\|_{H^1\left(\Omega_f\right)} \leq C_K\left\|\boldsymbol{D}\left(\boldsymbol{v}_h\right)\right\|_{L^2\left(\Omega_f\right)} \quad \forall \boldsymbol{v}_h \in \boldsymbol{V}_{f, h}.
\label{eq:discrete_korn}
\end{equation}
Therefore we also have:
\begin{equation}
\left\|\nabla \cdot \boldsymbol{v}_h\right\|_{L^2\left(\Omega_f\right)} \leq C\left\|\boldsymbol{D}\left(\boldsymbol{v}_h\right)\right\|_{L^2\left(\Omega_f\right)} \quad \forall \boldsymbol{v}_h \in \boldsymbol{V}_{f, h}.
\label{eq:div_from_korn}
\end{equation}

{\textbf{Poroelastic projections.}}
For the displacement $\boldsymbol{\eta}$, we define the elasticity Ritz projection
\(\Pi_\eta:\boldsymbol V_p\to \boldsymbol V_{p,h}\) by \cite{Wang2026ExplicitSplitting}
\begin{equation}
a_e(\boldsymbol\eta-\Pi_\eta\boldsymbol\eta,\boldsymbol\zeta_h)=0
\qquad \forall \boldsymbol\zeta_h\in \boldsymbol V_{p,h}.
\label{eq:eta_projection}
\end{equation}
Next, for the displacement velocity $\boldsymbol{\xi}$, we define the $H^1$-stable
projection \(\Pi_\xi:\boldsymbol V_p\to \boldsymbol V_{p,h}\) by 
\begin{equation}
(\nabla(\boldsymbol\xi-\Pi_\xi\boldsymbol\xi),\nabla\boldsymbol\zeta_h)_{\Omega_p}
+(\boldsymbol\xi-\Pi_\xi\boldsymbol\xi,\boldsymbol\zeta_h)_{\Omega_p}=0
\qquad \forall \boldsymbol\zeta_h\in \boldsymbol V_{p,h}.
\label{eq:xi_projection}
\end{equation}
For the error analysis below, it is not necessary to identify $\Pi_{\eta}$ and $\Pi_{\xi}$; the kinematic mismatch generated by using the projections together with the BDF2 difference is measured by the defect $\boldsymbol\kappa^{n+1}$ defined in \eqref{eq:kappa_def}.
Regarding the pore pressure $\phi$, we impose homogeneous Dirichlet conditions on the outer poroelastic boundary $
\Sigma_{D,p}$, but not on the interface $\Gamma$. We define the Darcy Ritz projection
\(\Pi_{\phi}:Q_p\to Q_{p,h}\) by
\begin{equation}
\left(K \nabla\left(\phi-\Pi_{\phi} \phi\right), \nabla \psi_h\right)_p+\frac{1}{L}\left\langle\phi-\Pi_{\phi} \phi, \psi_h\right\rangle_{\Gamma}=0,
\qquad \forall \psi_h\in Q_{p,h}.
\label{eq:phi_projection}
\end{equation}

We assume that the projections $\Pi_\eta$, $\Pi_\xi$, and $\Pi_{\phi}$ satisfy the standard approximation estimates:
\begin{align}
\left\|\boldsymbol \eta-\Pi_\eta\boldsymbol \eta\right\|_{L^2\left(\Omega_p\right)}
+h\left\|\boldsymbol \eta-\Pi_\eta\boldsymbol \eta\right\|_{H^1\left(\Omega_p\right)}
&\le Ch^{k+1}\|\boldsymbol \eta\|_{H^{k+1}\left(\Omega_p\right)},
\label{eq:proj_eta}
\\
\left\|\boldsymbol \xi-\Pi_\xi\boldsymbol \xi\right\|_{L^2\left(\Omega_p\right)}
+h\left\|\boldsymbol \xi-\Pi_\xi\boldsymbol \xi\right\|_{H^1\left(\Omega_p\right)}
&\le Ch^{k+1}\|\boldsymbol \xi\|_{H^{k+1}\left(\Omega_p\right)},
\label{eq:proj_xi}
\\
\left\|\phi-\Pi_{\phi}\phi\right\|_{L^2\left(\Omega_p\right)}
+h\left\|\phi-\Pi_{\phi}\phi\right\|_{H^1\left(\Omega_p\right)}
&\le C h^{k+1}\|\phi\|_{H^{k+1}\left(\Omega_p\right)}.
\label{eq:proj_phi}
\end{align}
Moreover, the corresponding trace estimates hold:
\begin{align}
\left\|\boldsymbol \eta-\Pi_\eta\boldsymbol \eta\right\|_{L^2(\Gamma)}
&\le Ch^{k+\frac12}\|\boldsymbol \eta\|_{H^{k+1}\left(\Omega_p\right)},
\label{eq:trace_proj_eta}
\\
\left\|\boldsymbol \xi-\Pi_\xi\boldsymbol \xi\right\|_{L^2(\Gamma)}
&\le Ch^{k+\frac12}\|\boldsymbol \xi\|_{H^{k+1}\left(\Omega_p\right)},
\label{eq:trace_proj_xi}
\\
\left\|\phi-\Pi_{\phi}\phi\right\|_{L^2(\Gamma)}
&\le Ch^{k+\frac12}\|\phi\|_{H^{k+1}\left(\Omega_p\right)},
\label{eq:trace_proj_phi}
\end{align}
in which the half-order loss comes from taking traces onto $\Gamma$.

\subsection{Error decomposition}
\label{subsec:error_decomp}

Since the projection operators are linear and time-independent, they commute with the discrete BDF2 operator, and we have:
\begin{align}
\delta_t \Pi_u\boldsymbol u^{n+1}=\Pi_u(\delta_t\boldsymbol u^{n+1}), \qquad 
\delta_t \Pi_\eta\boldsymbol\eta^{n+1}=\Pi_\eta(\delta_t\boldsymbol\eta^{n+1}),\nonumber \\
\delta_t \Pi_\xi\boldsymbol\xi^{n+1}=\Pi_\xi(\delta_t\boldsymbol\xi^{n+1}), \qquad
\delta_t \Pi_{\phi}\phi^{n+1}=\Pi_{\phi}(\delta_t\phi^{n+1}). \nonumber
\end{align}

Next, we split each total error into a projection error denoted by the superscript $I$ and a discrete error denoted by the superscript $h$:
\begin{align}
\boldsymbol e_u^{n}
:=\boldsymbol u^n-\boldsymbol u_h^n
&=(\boldsymbol u^n-\Pi_u\boldsymbol u^n)+(\Pi_u\boldsymbol u^n-\boldsymbol u_h^n)
=:-\boldsymbol e_u^{I,n}+\boldsymbol e_u^{h,n},
\label{eq:errsplit_u_new}
\\
e_p^{n}
:=p_f^n-p_{f,h}^n
&=(p_f^n-\Pi_p p_f^n)+(\Pi_p p_f^n-p_{f,h}^n)
=-e_p^{I,n}+e_p^{h,n},
\label{eq:errsplit_p_new}
\\
\boldsymbol e_\eta^{n}
:=\boldsymbol\eta^n-\boldsymbol\eta_h^n
&=(\boldsymbol\eta^n-\Pi_\eta\boldsymbol\eta^n)+(\Pi_\eta\boldsymbol\eta^n-\boldsymbol\eta_h^n)
=:-\boldsymbol e_\eta^{I,n}+\boldsymbol e_\eta^{h,n},
\label{eq:errsplit_eta_new}
\\
\boldsymbol e_\xi^{n}
:=\boldsymbol\xi^n-\boldsymbol\xi_h^n
&=(\boldsymbol\xi^n-\Pi_\xi\boldsymbol\xi^n)+(\Pi_\xi\boldsymbol\xi^n-\boldsymbol\xi_h^n)
=:-\boldsymbol e_\xi^{I,n}+\boldsymbol e_\xi^{h,n},
\label{eq:errsplit_xi_new}
\\
e_\phi^{n}
:=\phi^n-\phi_h^n
&=(\phi^n-\Pi_{\phi}\phi^n)+(\Pi_{\phi}\phi^n-\phi_h^n)
=-e_\phi^{I,n}+e_\phi^{h,n}.
\label{eq:errsplit_phi_new}
\end{align}
Equivalently, the projection errors are
\[
\boldsymbol e_u^{I,n}:=\Pi_u\boldsymbol u^n-\boldsymbol u^n,\quad
 e_p^{I,n}:=\Pi_p p_f^n-p_f^n,\quad
\boldsymbol e_\eta^{I,n}:=\Pi_\eta\boldsymbol\eta^n-\boldsymbol\eta^n,\quad
\boldsymbol e_\xi^{I,n}:=\Pi_\xi\boldsymbol\xi^n-\boldsymbol\xi^n,\quad
 e_\phi^{I,n}:=\Pi_{\phi}\phi^n-\phi^n,
\]
and the discrete errors are
\[
\boldsymbol e_u^{h,n}:=\Pi_u\boldsymbol u^n-\boldsymbol u_h^n,\quad
 e_p^{h,n}:=\Pi_p p_f^n-p_{f,h}^n,\quad
\boldsymbol e_\eta^{h,n}:=\Pi_\eta\boldsymbol\eta^n-\boldsymbol\eta_h^n,\quad
\boldsymbol e_\xi^{h,n}:=\Pi_\xi\boldsymbol\xi^n-\boldsymbol\xi_h^n,\quad
 e_\phi^{h,n}:=\Pi_{\phi}\phi^n-\phi_h^n.
\]

We also define the kinematic projection defect as follows
\begin{equation}
\boldsymbol\kappa^{n+1}:=\Pi_\xi\boldsymbol\xi^{n+1}-\delta_t\Pi_\eta\boldsymbol\eta^{n+1},
\label{eq:kappa_def}
\end{equation}
Although the exact continuum relation satisfies $\boldsymbol{\xi}=\partial_t\boldsymbol{\eta}$, the projected variables do not in general satisfy the same relation after replacing $\partial_t$ by the BDF2 difference.
Since $\boldsymbol\xi_h^{n+1}=\delta_t\boldsymbol\eta_h^{n+1}$, we have:
\begin{equation}
\boldsymbol e_\xi^{h,n+1}
=
\delta_t\boldsymbol e_\eta^{h,n+1}
+\boldsymbol\kappa^{n+1}.
\label{eq:error_kinematic_relation}
\end{equation}

\subsection{Time-discretization and extrapolation defects}
\label{subsec:time_defects}

Since BDF2 and AB2 were used, we will quantify the local defects of these stencils. For \(I_n:=(t_{n-1},t_{n+1})\), we use the following local defect bounds.

\begin{lemma}[BDF2 defect]
\label{lem:bdf2_defect}
Let \(X\) be a Banach space and \(\chi\in H^3(I_n;X)\). Then the local BDF2 truncation error is:
\begin{equation}
\|\delta_t\chi^{n+1}-\partial_t\chi(t_{n+1})\|_X^2
\le
C\Delta t^3\int_{I_n}\|\partial_t^3\chi(s)\|_X^2\,ds.
\label{eq:bdf2_defect_new}
\end{equation}
\end{lemma}

\begin{proof}
Since the BDF2 formula is exact for polynomials of degree at most two, the
Peano kernel representation yields
\[
\delta_t\chi^{n+1}-\partial_t\chi(t_{n+1})
=
\frac{1}{2\Delta t}\int_{I_n}\omega_{n+1}(s)\,\partial_t^3\chi(s)\,ds,
\]
where \(|\omega_{n+1}(s)|\le C\Delta t^2\) for \(s\in I_n\). Hence
\[
\|\delta_t\chi^{n+1}-\partial_t\chi(t_{n+1})\|_X
\le
C\Delta t\int_{I_n}\|\partial_t^3\chi(s)\|_X\,ds.
\]
Applying Cauchy-Schwarz in time gives \eqref{eq:bdf2_defect_new}.
\end{proof}

\begin{lemma}[AB2 defect]
\label{lem:ab2_defect}
Let \(X\) be a Hilbert space and \(\chi\in H^2(I_n;X)\). Then
\begin{equation}
\|\chi^{\star}-\chi(t_{n+1})\|_X^2
\le
C\Delta t^3\int_{I_n}\|\partial_t^2\chi(s)\|_X^2\,ds.
\label{eq:ab2_defect_new}
\end{equation}
\end{lemma}

\begin{proof}
Using the fundamental theorem of calculus twice,
\[
\chi^{\star}-\chi(t_{n+1})
=
-\int_{t_n}^{t_{n+1}}(t_{n+1}-s)\,\partial_t^2\chi(s)\,ds
-\int_{t_{n-1}}^{t_n}(s-t_{n-1})\,\partial_t^2\chi(s)\,ds.
\]
Therefore
\[
\|\chi^{\star}-\chi(t_{n+1})\|_X
\le
C\Delta t\int_{I_n}\|\partial_t^2\chi(s)\|_X\,ds,
\]
Applying Cauchy-Schwarz yields
\eqref{eq:ab2_defect_new}.
\end{proof}

\begin{lemma}[Kinematic projection defect] \label{lem:kappa_defect} As a consequence of \(\boldsymbol\xi=\partial_t\boldsymbol\eta\), \eqref{eq:proj_eta}--\eqref{eq:proj_xi}, and Lemma~\ref{lem:bdf2_defect}, the kinematic defect satisfies \begin{align} \|\boldsymbol{\kappa}^{n+1}\|_{L^2(\Omega_p)}^2 &\le C h^{2k+2} + C \Delta t^3 \int_{I_n}\|\partial_t^3 \boldsymbol{\eta}(s)\|_{L^2(\Omega_p)}^2\,ds, \label{eq:kappa_bound_L2} \\ \|\boldsymbol{\kappa}^{n+1}\|_{H^1(\Omega_p)}^2 &\le C h^{2k} + C \Delta t^3 \int_{I_n}\|\partial_t^3 \boldsymbol{\eta}(s)\|_{H^1(\Omega_p)}^2\,ds. \label{eq:kappa_bound_H1} \end{align} \end{lemma}
\begin{proof}
Using the definition \eqref{eq:kappa_def}, the exact kinematic relation
\(\boldsymbol\xi=\partial_t\boldsymbol\eta\), and the projected quantities
\(\Pi_\eta\boldsymbol\eta^n\) and \(\Pi_\xi\boldsymbol\xi^n\) introduced in
Section~\ref{subsec:error_decomp}, we write
\begin{align*}
\boldsymbol\kappa^{n+1}
&=
\Pi_\xi\boldsymbol\xi^{n+1}-\delta_t\Pi_\eta\boldsymbol\eta^{n+1} \\
&=
(\Pi_\xi\boldsymbol\xi^{n+1}-\boldsymbol\xi^{n+1})
+
(\boldsymbol\xi^{n+1}-\delta_t\boldsymbol\eta^{n+1})
+
(\delta_t\boldsymbol\eta^{n+1}-\delta_t\Pi_\eta\boldsymbol\eta^{n+1}).
\end{align*}
Therefore, for \(X=L^2(\Omega_p)^d\) or \(X=H^1(\Omega_p)^d\),
\begin{equation}
\|\boldsymbol\kappa^{n+1}\|_X^2
\le
3\Big(
\|\Pi_\xi\boldsymbol\xi^{n+1}-\boldsymbol\xi^{n+1}\|_X^2
+\|\boldsymbol\xi^{n+1}-\delta_t\boldsymbol\eta^{n+1}\|_X^2
+\|\delta_t\boldsymbol\eta^{n+1}-\delta_t\Pi_\eta\boldsymbol\eta^{n+1}\|_X^2
\Big).
\label{eq:kappa_split_bound}
\end{equation}

We estimate the three terms separately. For
\(\Pi_\xi\boldsymbol\xi^{n+1}-\boldsymbol\xi^{n+1}\), the projection estimate
\eqref{eq:proj_xi} yields:
\begin{align}
\|\Pi_\xi\boldsymbol\xi^{n+1}-\boldsymbol\xi^{n+1}\|_{L^2(\Omega_p)}^2
&\le
C h^{2k+2},
\label{eq:kappa_A_L2}
\\
\|\Pi_\xi\boldsymbol\xi^{n+1}-\boldsymbol\xi^{n+1}\|_{H^1(\Omega_p)}^2
&\le
C h^{2k}.
\label{eq:kappa_A_H1}
\end{align}
For the term \(\boldsymbol\xi^{n+1}-\delta_t\boldsymbol\eta^{n+1}\),
by the continuous kinematic relation $\boldsymbol\xi=\partial_t\boldsymbol\eta$,
we have:
$$
\boldsymbol\xi^{n+1}-\delta_t\boldsymbol\eta^{n+1}
=
\partial_t\boldsymbol\eta(t_{n+1})-\delta_t\boldsymbol\eta^{n+1}.
$$
Hence Lemma~\ref{lem:bdf2_defect}, applied with \(\chi=\boldsymbol\eta\), gives
\begin{align}
\|\boldsymbol\xi^{n+1}-\delta_t\boldsymbol\eta^{n+1}\|_{L^2(\Omega_p)}^2
&\le
C\Delta t^3\int_{I_n}
\|\partial_t^3\boldsymbol\eta(s)\|_{L^2(\Omega_p)}^2\,ds,
\label{eq:kappa_B_L2}
\\
\|\boldsymbol\xi^{n+1}-\delta_t\boldsymbol\eta^{n+1}\|_{H^1(\Omega_p)}^2
&\le
C\Delta t^3\int_{I_n}
\|\partial_t^3\boldsymbol\eta(s)\|_{H^1(\Omega_p)}^2\,ds.
\label{eq:kappa_B_H1}
\end{align}
Regarding \(\delta_t\boldsymbol\eta^{n+1}-\delta_t\Pi_\eta\boldsymbol\eta^{n+1}\),
since the projection operators are linear and time-independent, they commute with the BDF2 operator, namely:
\[
\delta_t\Pi_\eta\boldsymbol\eta^{n+1}
=
\Pi_\eta(\delta_t\boldsymbol\eta^{n+1}).
\]
Therefore,
\[
\delta_t\boldsymbol\eta^{n+1}-\delta_t\Pi_\eta\boldsymbol\eta^{n+1}
=
(I-\Pi_\eta)\,\delta_t\boldsymbol\eta^{n+1}.
\]
Applying the approximation property \eqref{eq:proj_eta} to the function
\(\delta_t\boldsymbol\eta^{n+1}\), we obtain
\begin{equation}
\|(I-\Pi_\eta)\,\delta_t\boldsymbol\eta^{n+1}\|_{L^2(\Omega_p)}
+
h\|(I-\Pi_\eta)\,\delta_t\boldsymbol\eta^{n+1}\|_{H^1(\Omega_p)}
\le
Ch^{k+1}\|\delta_t\boldsymbol\eta^{n+1}\|_{H^{k+1}(\Omega_p)}.
\label{eq:kappa_C_projection}
\end{equation}
It remains to bound \(\delta_t\boldsymbol\eta^{n+1}\) in \(H^{k+1}(\Omega_p)\).
Using the BDF2 formula,
\[
\delta_t\boldsymbol\eta^{n+1}
=
\frac{3}{2\Delta t}\big(\boldsymbol\eta^{n+1}-\boldsymbol\eta^n\big)
-\frac{1}{2\Delta t}\big(\boldsymbol\eta^n-\boldsymbol\eta^{n-1}\big),
\]
and
\[
\boldsymbol\eta^{n+1}-\boldsymbol\eta^n
=
\int_{t_n}^{t_{n+1}}\partial_t\boldsymbol\eta(s)\,ds,
\qquad
\boldsymbol\eta^n-\boldsymbol\eta^{n-1}
=
\int_{t_{n-1}}^{t_n}\partial_t\boldsymbol\eta(s)\,ds,
\]
we have:
\begin{align*}
\|\delta_t\boldsymbol\eta^{n+1}\|_{H^{k+1}(\Omega_p)}
&\le
\frac{3}{2\Delta t}\int_{t_n}^{t_{n+1}}
\|\partial_t\boldsymbol\eta(s)\|_{H^{k+1}(\Omega_p)}\,ds
+\frac{1}{2\Delta t}\int_{t_{n-1}}^{t_n}
\|\partial_t\boldsymbol\eta(s)\|_{H^{k+1}(\Omega_p)}\,ds \\
&\le
2\,\|\partial_t\boldsymbol\eta\|_{L^\infty(I_n;H^{k+1}(\Omega_p))}.
\end{align*}
This quantity is absorbed into the generic constant \(C\). Inserting this into
\eqref{eq:kappa_C_projection} gives
\begin{align}
\|(I-\Pi_\eta)\,\delta_t\boldsymbol\eta^{n+1}\|_{L^2(\Omega_p)}^2
&\le
Ch^{2k+2},
\label{eq:kappa_C_L2}
\\
\|(I-\Pi_\eta)\,\delta_t\boldsymbol\eta^{n+1}\|_{H^1(\Omega_p)}^2
&\le
Ch^{2k}.
\label{eq:kappa_C_H1}
\end{align}
Combining \eqref{eq:kappa_split_bound} with
\eqref{eq:kappa_A_L2}, \eqref{eq:kappa_B_L2}, and \eqref{eq:kappa_C_L2}, we obtain
\[
\|\boldsymbol{\kappa}^{n+1}\|_{L^2(\Omega_p)}^2
\le
C h^{2k+2}
+
C \Delta t^3 \int_{I_n}
\|\partial_t^3 \boldsymbol{\eta}(s)\|_{L^2(\Omega_p)}^2\,ds.
\]
Similarly, combining \eqref{eq:kappa_split_bound} with
\eqref{eq:kappa_A_H1}, \eqref{eq:kappa_B_H1}, and \eqref{eq:kappa_C_H1}, we get
\[
\|\boldsymbol{\kappa}^{n+1}\|_{H^1(\Omega_p)}^2
\le
C h^{2k}
+
C \Delta t^3 \int_{I_n}
\|\partial_t^3 \boldsymbol{\eta}(s)\|_{H^1(\Omega_p)}^2\,ds.
\]
This proves \eqref{eq:kappa_bound_L2}--\eqref{eq:kappa_bound_H1}.
\end{proof}

\subsection{Projected equations and discrete error equations}
\label{subsec:error_equations_new}

We first write the time discrete continuous equations at $t_{n+1}$ by adding and
subtracting the projected quantities, we obtain:

{\textbf{Projected fluid equation.}}
For all \((\boldsymbol v_h,q_h)\in \boldsymbol V_{f,h}\times Q_{f,h}\),
\begin{align}
&\rho_f(\delta_t \Pi_u\boldsymbol u^{n+1},\boldsymbol v_h)_f
+2\mu_f(\boldsymbol D(\Pi_u\boldsymbol u^{n+1}),\boldsymbol D(\boldsymbol v_h))_f
-(\Pi_p p_f^{n+1},\nabla\!\cdot\!\boldsymbol v_h)_f
+(\nabla\!\cdot\!\Pi_u\boldsymbol u^{n+1},q_h)_f
\nonumber\\
&\quad
+\gamma\langle \boldsymbol P_f\Pi_u\boldsymbol u^{n+1},\boldsymbol P_f\boldsymbol v_h\rangle_\Gamma
+L\langle \Pi_u\boldsymbol u^{n+1}\!\cdot\!\boldsymbol n_f,\boldsymbol v_h\!\cdot\!\boldsymbol n_f\rangle_\Gamma
\nonumber\\
&=
\gamma\langle \boldsymbol P_f\Pi_\xi\boldsymbol\xi^{\star},\boldsymbol P_f\boldsymbol v_h\rangle_\Gamma
+\langle L\,\Pi_u\boldsymbol u^{\star}\!\cdot\!\boldsymbol n_f-\Pi_{\phi}\phi^{\star},
\boldsymbol v_h\!\cdot\!\boldsymbol n_f\rangle_\Gamma
+\mathfrak R_f^{n+1}(\boldsymbol v_h),
\label{eq:proj_fluid_eq}
\end{align}
where $\mathfrak R_f^{n+1}$ contains the time defect, space projection defects, AB2 interface defects:
\begin{align}
\mathfrak R_f^{n+1}(\boldsymbol v_h)
&:=
\rho_f(\delta_t\Pi_u\boldsymbol u^{n+1}-\partial_t\boldsymbol u^{n+1},\boldsymbol v_h)_f
+2\mu_f(\boldsymbol D\boldsymbol e_u^{I,n+1},\boldsymbol D\boldsymbol v_h)_f
\notag\\
&\quad
-(e_p^{I,n+1},\nabla\!\cdot\!\boldsymbol v_h)_f
+\gamma\langle \boldsymbol P_f\boldsymbol e_u^{I,n+1},
\boldsymbol P_f\boldsymbol v_h\rangle_\Gamma
+L\langle \boldsymbol e_u^{I,n+1}\!\cdot\!\boldsymbol n_f,
\boldsymbol v_h\!\cdot\!\boldsymbol n_f\rangle_\Gamma
\notag\\
&\quad
-\gamma\langle \boldsymbol P_f(\Pi_\xi\boldsymbol\xi^{\star}-\boldsymbol\xi^{n+1}),
\boldsymbol P_f\boldsymbol v_h\rangle_\Gamma
\notag\\
&\quad
-\langle L(\Pi_u\boldsymbol u^{\star}-\boldsymbol u^{n+1})\!\cdot\!\boldsymbol n_f
-(\Pi_{\phi}\phi^{\star}-\phi^{n+1}),
\boldsymbol v_h\!\cdot\!\boldsymbol n_f\rangle_\Gamma.
\label{eq:Rf_new}
\end{align}
By \eqref{eq:Fortin_property} and \(\nabla\!\cdot\!\boldsymbol u^{n+1}=0\),
\[
(\nabla\!\cdot\!\boldsymbol e_u^{I,n+1},q_h)_f=0,
\]
so there is no \(q_h\)-residual in \(\mathfrak R_f^{n+1}\).

{\textbf{Projected Biot equation.}}
For all
\((\boldsymbol v_{p,h},\boldsymbol\zeta_h,\psi_h)
\in \boldsymbol V_{p,h}\times\boldsymbol V_{p,h}\times Q_{p,h}\),
\begin{align}
&\rho_p(\delta_t\Pi_\xi\boldsymbol\xi^{n+1},\boldsymbol\zeta_h)_p
+a_e(\Pi_\eta\boldsymbol\eta^{n+1},\boldsymbol\zeta_h)
-\alpha(\Pi_{\phi}\phi^{n+1},\nabla\!\cdot\!\boldsymbol\zeta_h)_p
\nonumber\\
&\quad
+(\Pi_\xi\boldsymbol\xi^{n+1},\boldsymbol v_{p,h})_p
-(\delta_t\Pi_\eta\boldsymbol\eta^{n+1},\boldsymbol v_{p,h})_p
+C_0(\delta_t\Pi_{\phi}\phi^{n+1},\psi_h)_p
+\alpha(\nabla\!\cdot\!\Pi_\xi\boldsymbol\xi^{n+1},\psi_h)_p
\nonumber\\
&\quad
+(K\nabla\Pi_{\phi}\phi^{n+1},\nabla\psi_h)_p
+\gamma\langle \boldsymbol P_p\Pi_\xi\boldsymbol\xi^{n+1},\boldsymbol P_p\boldsymbol\zeta_h\rangle_\Gamma
+\langle \Pi_\xi\boldsymbol\xi^{n+1}\!\cdot\!\boldsymbol n_p,\boldsymbol\zeta_h\!\cdot\!\boldsymbol n_p\rangle_\Gamma
\nonumber\\
&\quad
+\langle \Pi_{\phi}\phi^{n+1},\boldsymbol\zeta_h\!\cdot\!\boldsymbol n_p\rangle_\Gamma
+\frac1L\langle \Pi_{\phi}\phi^{n+1},\psi_h\rangle_\Gamma
-\langle \Pi_\xi\boldsymbol\xi^{n+1}\!\cdot\!\boldsymbol n_p,\psi_h\rangle_\Gamma
\nonumber\\
&=
\gamma\langle \boldsymbol P_p\Pi_u\boldsymbol u^{\star},\boldsymbol P_p\boldsymbol\zeta_h\rangle_\Gamma
+\langle \Pi_\xi\boldsymbol\xi^{\star}\!\cdot\!\boldsymbol n_p,\boldsymbol\zeta_h\!\cdot\!\boldsymbol n_p\rangle_\Gamma
\nonumber\\
&\quad
+\langle -\Pi_u\boldsymbol u^{\star}\!\cdot\!\boldsymbol n_p
+L^{-1}\Pi_{\phi}\phi^{\star},\psi_h\rangle_\Gamma
+\mathfrak R_p^{n+1}(\boldsymbol v_{p,h},\boldsymbol\zeta_h,\psi_h),
\label{eq:proj_biot_eq}
\end{align}
Similarly, all the mismatch terms are grouped into $\mathfrak R_p^{n+1}$:
\begin{align}
\mathfrak R_p^{n+1}(\boldsymbol v_{p,h},\boldsymbol\zeta_h,\psi_h)
&:=
\rho_p(\delta_t\Pi_\xi\boldsymbol\xi^{n+1}-\partial_t\boldsymbol\xi^{n+1},\boldsymbol\zeta_h)_p
+(\boldsymbol\kappa^{n+1},\boldsymbol v_{p,h})_p
\nonumber\\
&\quad
-\alpha(e_\phi^{I,n+1},\nabla\!\cdot\!\boldsymbol\zeta_h)_p
+C_0(\delta_t\Pi_{\phi}\phi^{n+1}-\partial_t\phi^{n+1},\psi_h)_p
\notag\\
&\quad
+\alpha(\nabla\!\cdot\!\boldsymbol e_\xi^{I,n+1},\psi_h)_p
-\gamma\langle \boldsymbol P_p(\Pi_u\boldsymbol u^{\star}-\boldsymbol u^{n+1}),
\boldsymbol P_p\boldsymbol\zeta_h\rangle_\Gamma
\notag\\
&\quad
-\langle (\Pi_\xi\boldsymbol\xi^{\star}-\boldsymbol\xi^{n+1})\!\cdot\!\boldsymbol n_p,
\boldsymbol\zeta_h\!\cdot\!\boldsymbol n_p\rangle_\Gamma
\notag\\
&\quad
-\langle -(\Pi_u\boldsymbol u^{\star}-\boldsymbol u^{n+1})\!\cdot\!\boldsymbol n_p
+L^{-1}(\Pi_{\phi}\phi^{\star}-\phi^{n+1}),\psi_h\rangle_\Gamma
\notag\\
&\quad
+\gamma\langle \boldsymbol P_p\boldsymbol e_\xi^{I,n+1},
\boldsymbol P_p\boldsymbol\zeta_h\rangle_\Gamma
+\langle \boldsymbol e_\xi^{I,n+1}\!\cdot\!\boldsymbol n_p,
\boldsymbol\zeta_h\!\cdot\!\boldsymbol n_p\rangle_\Gamma
\notag\\
&\quad
+\langle e_\phi^{I,n+1},
\boldsymbol\zeta_h\!\cdot\!\boldsymbol n_p\rangle_\Gamma
-\langle \boldsymbol e_\xi^{I,n+1}\!\cdot\!\boldsymbol n_p,
\psi_h\rangle_\Gamma.
\label{eq:Rp_new}
\end{align}
The elasticity Ritz projection \eqref{eq:eta_projection} eliminates the
\(
a_e(\Pi_\eta\boldsymbol\eta^{n+1}-\boldsymbol\eta^{n+1},\boldsymbol\zeta_h)
\) term,
and the Darcy Ritz projection \eqref{eq:phi_projection} eliminates
$$
\left(K \nabla\left(\Pi_{\phi}\phi^{n+1}-\phi^{n+1}\right), \nabla \psi_h\right)_p+\frac{1}{L}\left\langle\Pi_{\phi}\phi^{n+1}-\phi^{n+1}, \psi_h\right\rangle_{\Gamma} .
$$

Next, subtracting the fully discrete scheme from equations
\eqref{eq:proj_fluid_eq}-\eqref{eq:proj_biot_eq} yields the discrete error equations:

{\textbf{Discrete fluid error equation.}}
For all \((\boldsymbol v_h,q_h)\in\boldsymbol V_{f,h}\times Q_{f,h}\),
\begin{align}
&\rho_f(\delta_t \boldsymbol e_u^{h,n+1},\boldsymbol v_h)_f
+2\mu_f(\boldsymbol D(\boldsymbol e_u^{h,n+1}),\boldsymbol D(\boldsymbol v_h))_f
-(e_p^{h,n+1},\nabla\!\cdot\!\boldsymbol v_h)_f
+(\nabla\!\cdot\!\boldsymbol e_u^{h,n+1},q_h)_f
\nonumber\\
&\quad
+\gamma\langle \boldsymbol P_f\boldsymbol e_u^{h,n+1},\boldsymbol P_f\boldsymbol v_h\rangle_\Gamma
+L\langle \boldsymbol e_u^{h,n+1}\!\cdot\!\boldsymbol n_f,\boldsymbol v_h\!\cdot\!\boldsymbol n_f\rangle_\Gamma
\nonumber\\
&=
\gamma\langle \boldsymbol P_f\boldsymbol e_\xi^{h,\star},\boldsymbol P_f\boldsymbol v_h\rangle_\Gamma
+\langle L\,\boldsymbol e_u^{h,\star}\!\cdot\!\boldsymbol n_f-e_\phi^{h,\star},
\boldsymbol v_h\!\cdot\!\boldsymbol n_f\rangle_\Gamma
+\mathfrak R_f^{n+1}(\boldsymbol v_h).
\label{eq:fluid_error_new}
\end{align}

{\textbf{Discrete Biot error equation.}}
For all
\((\boldsymbol v_{p,h},\boldsymbol\zeta_h,\psi_h)
\in\boldsymbol V_{p,h}\times\boldsymbol V_{p,h}\times Q_{p,h}\),
\begin{align}
&\rho_p(\delta_t\boldsymbol e_\xi^{h,n+1},\boldsymbol\zeta_h)_p
+a_e(\boldsymbol e_\eta^{h,n+1},\boldsymbol\zeta_h)
-\alpha(e_\phi^{h,n+1},\nabla\!\cdot\!\boldsymbol\zeta_h)_p
\nonumber\\
&\quad
+(\boldsymbol e_\xi^{h,n+1},\boldsymbol v_{p,h})_p
-(\delta_t\boldsymbol e_\eta^{h,n+1},\boldsymbol v_{p,h})_p
+C_0(\delta_t e_\phi^{h,n+1},\psi_h)_p
+\alpha(\nabla\!\cdot\!\boldsymbol e_\xi^{h,n+1},\psi_h)_p
\nonumber\\
&\quad
+(K\nabla e_\phi^{h,n+1},\nabla\psi_h)_p
+\gamma\langle \boldsymbol P_p\boldsymbol e_\xi^{h,n+1},\boldsymbol P_p\boldsymbol\zeta_h\rangle_\Gamma
+\langle \boldsymbol e_\xi^{h,n+1}\!\cdot\!\boldsymbol n_p,\boldsymbol\zeta_h\!\cdot\!\boldsymbol n_p\rangle_\Gamma
\nonumber\\
&\quad
+\langle e_\phi^{h,n+1},\boldsymbol\zeta_h\!\cdot\!\boldsymbol n_p\rangle_\Gamma
+\frac1L\langle e_\phi^{h,n+1},\psi_h\rangle_\Gamma
-\langle \boldsymbol e_\xi^{h,n+1}\!\cdot\!\boldsymbol n_p,\psi_h\rangle_\Gamma
\nonumber\\
&=
\gamma\langle \boldsymbol P_p\boldsymbol e_u^{h,\star},\boldsymbol P_p\boldsymbol\zeta_h\rangle_\Gamma
+\langle \boldsymbol e_\xi^{h,\star}\!\cdot\!\boldsymbol n_p,\boldsymbol\zeta_h\!\cdot\!\boldsymbol n_p\rangle_\Gamma
\nonumber\\
&\quad
+\langle -\boldsymbol e_u^{h,\star}\!\cdot\!\boldsymbol n_p+L^{-1}e_\phi^{h,\star},\psi_h\rangle_\Gamma
+\mathfrak R_p^{n+1}(\boldsymbol v_{p,h},\boldsymbol\zeta_h,\psi_h).
\label{eq:biot_error_new}
\end{align}

We now introduce the discrete error energy \(\mathcal E_h^n\) and the associated dissipation terms, defined by
\begin{align}
\mathcal E_h^{n}
&:=
\frac{\rho_f}{4}\Big(\|\boldsymbol e_u^{h,n}\|_{L^2\left(\Omega_f\right)}^2
+\|2\boldsymbol e_u^{h,n}-\boldsymbol e_u^{h,n-1}\|_{L^2\left(\Omega_f\right)}^2\Big)
+\frac{\rho_p}{4}\Big(\|\boldsymbol e_\xi^{h,n}\|_{L^2\left(\Omega_p\right)}^2
+\|2\boldsymbol e_\xi^{h,n}-\boldsymbol e_\xi^{h,n-1}\|_{L^2\left(\Omega_p\right)}^2\Big)
\notag\\
&\quad
+\frac{C_0}{4}\Big(\|e_\phi^{h,n}\|_{L^2\left(\Omega_p\right)}^2
+\|2e_\phi^{h,n}-e_\phi^{h,n-1}\|_{L^2\left(\Omega_p\right)}^2\Big)
+\frac{1}{4}\Big(\|\boldsymbol e_\eta^{h,n}\|_{a_e}^2
+\|2\boldsymbol e_\eta^{h,n}-\boldsymbol e_\eta^{h,n-1}\|_{a_e}^2\Big),
\label{eq:Eh_def}
\\
\mathcal D_{h,\mathrm{bulk}}^{n+1}
&:=
\frac{\rho_f}{4\Delta t}\|\boldsymbol e_u^{h,n+1}-2\boldsymbol e_u^{h,n}+\boldsymbol e_u^{h,n-1}\|_{L^2\left(\Omega_f\right)}^2
+\frac{\rho_p}{4\Delta t}\|\boldsymbol e_\xi^{h,n+1}-2\boldsymbol e_\xi^{h,n}+\boldsymbol e_\xi^{h,n-1}\|_{L^2\left(\Omega_p\right)}^2
\notag\\
&\quad
+\frac{C_0}{4\Delta t}\|e_\phi^{h,n+1}-2e_\phi^{h,n}+e_\phi^{h,n-1}\|_{L^2\left(\Omega_p\right)}^2
+\frac{1}{4\Delta t}\|\boldsymbol e_\eta^{h,n+1}-2\boldsymbol e_\eta^{h,n}+\boldsymbol e_\eta^{h,n-1}\|_{a_e}^2
\notag\\
&\quad
+2\mu_f\|\boldsymbol D(\boldsymbol e_u^{h,n+1})\|_{L^2\left(\Omega_f\right)}^2
+\|K^{1/2}\nabla e_\phi^{h,n+1}\|_{L^2\left(\Omega_p\right)}^2,
\label{eq:Dh_bulk_def}
\\
\mathcal D_{h,\Gamma}^{n+1}
&:=
\gamma\|\boldsymbol P_f\boldsymbol e_u^{h,n+1}\|_{L^2\left(\Gamma\right)}^2
+\gamma\|\boldsymbol P_p\boldsymbol e_\xi^{h,n+1}\|_{L^2\left(\Gamma\right)}^2
\notag\\
&\quad
+L\|\boldsymbol e_u^{h,n+1}\!\cdot\!\boldsymbol n_f\|_{L^2\left(\Gamma\right)}^2
+\|\boldsymbol e_\xi^{h,n+1}\!\cdot\!\boldsymbol n_p\|_{L^2\left(\Gamma\right)}^2
+L^{-1}\|e_\phi^{h,n+1}\|_{L^2\left(\Gamma\right)}^2.
\label{eq:Dh_Gamma_def}
\end{align}
And the grouped defect terms are split into bulk and interface contributions:
\begin{align}
\mathcal T_{\mathrm{bulk}}^{n+1}
&:=
\|\delta_t\Pi_u\boldsymbol u^{n+1}-\partial_t\boldsymbol u^{n+1}\|_{L^2\left(\Omega_f\right)}^2
+\|\boldsymbol e_u^{I,n+1}\|_{H^1\left(\Omega_f\right)}^2
+\|e_p^{I,n+1}\|_{L^2\left(\Omega_f\right)}^2
\notag\\
&\quad
+\|\delta_t\Pi_\xi\boldsymbol\xi^{n+1}-\partial_t\boldsymbol\xi^{n+1}\|_{L^2\left(\Omega_p\right)}^2
+\|\delta_t\Pi_{\phi}\phi^{n+1}-\partial_t\phi^{n+1}\|_{L^2\left(\Omega_p\right)}^2
+\|\boldsymbol\kappa^{n+1}\|_{H^1\left(\Omega_p\right)}^2
\notag\\
&\quad
+\|\nabla e_\phi^{I,n+1}\|_{L^2\left(\Omega_p\right)}^2
+\|\nabla\!\cdot\!\boldsymbol e_\xi^{I,n+1}\|_{L^2\left(\Omega_p\right)}^2,
\label{eq:T_bulk_def}
\\
\mathcal T_\Gamma^{n+1}
&:=
\|\boldsymbol e_u^{I,n+1}\|_{L^2\left(\Gamma\right)}^2
+\|e_\phi^{I,n+1}\|_{L^2\left(\Gamma\right)}^2
+\|\boldsymbol e_\xi^{I,n+1}\|_{L^2(\Gamma)}^2
\notag\\
&\quad
+\|\Pi_u\boldsymbol u^{\star}-\boldsymbol u^{n+1}\|_{L^2\left(\Gamma\right)}^2
+\|\Pi_\xi\boldsymbol\xi^{\star}-\boldsymbol\xi^{n+1}\|_{L^2\left(\Gamma\right)}^2
+\|\Pi_{\phi}\phi^{\star}-\phi^{n+1}\|_{L^2\left(\Gamma\right)}^2.
\label{eq:T_Gamma_def}
\end{align}

\begin{lemma}[One-step error inequality]
\label{lem:one_step_error}
Under the hypotheses of Theorem~\ref{thm:stability_cfl}, the discrete errors satisfy
\begin{equation}
\frac{1}{\Delta t}\big(\mathcal E_h^{n+1}-\mathcal E_h^n\big)
+\frac14\mathcal D_{h,\mathrm{bulk}}^{n+1}
+\gamma\|\boldsymbol P_f\boldsymbol e_u^{h,n+1}
-\boldsymbol P_p\boldsymbol e_\xi^{h,n+1}\|_{L^2(\Gamma)}^2
\le
C\Big(1+\frac{\Delta t}{h^2}\Big)\mathcal E_h^{n+1}
+
C\mathcal T_{\mathrm{bulk}}^{n+1}
+
Ch^{-1}\mathcal T_\Gamma^{n+1},
\label{eq:one_step_error}
\end{equation}
where \(\mathcal D_{h,\mathrm{bulk}}^{n+1}\) is defined in
\eqref{eq:Dh_bulk_def}, \(\mathcal T_{\mathrm{bulk}}^{n+1}\) in
\eqref{eq:T_bulk_def}, and \(\mathcal T_\Gamma^{n+1}\) in
\eqref{eq:T_Gamma_def}.
\end{lemma}
Namely, the error energy at time level $n+1$, together with the bulk dissipation and the tangential interface mismatch, is controlled by the current error energy multiplied by the factor $1+\frac{\Delta t}{h^2}$, plus the bulk consistency defects and the \(h^{-1}\)-weighted interface consistency defects. The weight \(h^{-1}\) arises from the inverse trace estimates used to control the normal traces of the discrete poroelastic error variables.

\begin{proof}
We test \eqref{eq:fluid_error_new} with
\((\boldsymbol v_h,q_h)=(\boldsymbol e_u^{h,n+1},e_p^{h,n+1})\), and
\eqref{eq:biot_error_new} with
$
(\boldsymbol v_{p,h},\boldsymbol\zeta_h,\psi_h)
=
(\delta_t\boldsymbol e_\eta^{h,n+1},
 \boldsymbol e_\xi^{h,n+1},
 e_\phi^{h,n+1})$.
 In the fluid equation, the pressure-divergence terms cancel
 $$
-\left(e_p^{h,n+1}, \nabla \cdot \boldsymbol e_u^{h,n+1}\right)_f+\left(\nabla \cdot \boldsymbol e_u^{h,n+1}, e_p^{h,n+1}\right)_f=0 .
$$
 While in the Biot
equation, the following terms cancel:
\[
-\alpha(e_\phi^{h,n+1},\nabla\!\cdot\!\boldsymbol e_\xi^{h,n+1})_p
+\alpha(\nabla\!\cdot\!\boldsymbol e_\xi^{h,n+1},e_\phi^{h,n+1})_p=0,
\]
\[
\langle e_\phi^{h,n+1},\boldsymbol e_\xi^{h,n+1}\!\cdot\!\boldsymbol n_p\rangle_\Gamma
-
\langle \boldsymbol e_\xi^{h,n+1}\!\cdot\!\boldsymbol n_p,e_\phi^{h,n+1}\rangle_\Gamma
=0.
\]
Using equation \eqref{eq:error_kinematic_relation}, we rewrite the elasticity term as:
\[
a_e(\boldsymbol e_\eta^{h,n+1},\boldsymbol e_\xi^{h,n+1})
=
a_e(\boldsymbol e_\eta^{h,n+1},\delta_t\boldsymbol e_\eta^{h,n+1})
+
a_e(\boldsymbol e_\eta^{h,n+1},\boldsymbol\kappa^{n+1}).
\]
The weak kinematic relation in the error equation produces the term
$(\boldsymbol e_\xi^{h,n+1}, \delta_t \boldsymbol e_\eta^{h,n+1})_p
-\|\delta_t \boldsymbol e_\eta^{h,n+1}\|_{L^2(\Omega_p)}^2
=(\boldsymbol{\kappa}^{n+1}, \delta_t \boldsymbol e_\eta^{h,n+1})_p$,
which cancels exactly with the kinematic defect contribution
$(\boldsymbol\kappa^{n+1},\boldsymbol v_{p,h})_p$ in
$\mathfrak R_p^{n+1}$ when tested with
$\boldsymbol v_{p,h}=\delta_t\boldsymbol e_\eta^{h,n+1}$.
We define the reduced poroelastic residual
\begin{equation}\label{eq:Rp_tilde_def}
\widetilde{\mathfrak R}_p^{n+1}
:=
\mathfrak R_p^{n+1}
-
(\boldsymbol\kappa^{n+1},\boldsymbol v_{p,h})_p,
\end{equation}
so that after the cancellation, only
$\widetilde{\mathfrak R}_p^{n+1}$ and the elastic cross term
$a_e(\boldsymbol e_\eta^{h,n+1},\boldsymbol\kappa^{n+1})$ remain.
Applying the BDF2 identities \eqref{eq:bdf2_energy_identity} and
\eqref{eq:bdf2_elastic_identity}, we obtain
\begin{align}
\frac{1}{\Delta t}\big(\mathcal E_h^{n+1}-\mathcal E_h^n\big)
+\mathcal D_{h,\mathrm{bulk}}^{n+1}
+\mathcal D_{h,\Gamma}^{n+1}
=
\mathcal I_h^{n+1}
+\mathfrak R_f^{n+1}(\boldsymbol e_u^{h,n+1})
+\widetilde{\mathfrak R}_p^{n+1}(\delta_t\boldsymbol e_\eta^{h,n+1},
\boldsymbol e_\xi^{h,n+1},e_\phi^{h,n+1})
\notag\\
+\,a_e(\boldsymbol e_\eta^{h,n+1},\boldsymbol\kappa^{n+1}),
\label{eq:error_energy_identity}
\end{align}
where \(\mathcal I_h^{n+1}\) is the sum of the explicit interface terms in
the error equations
\begin{align}
\mathcal I_h^{n+1}
&:=
\gamma\big\langle \boldsymbol P_f \boldsymbol e_\xi^{h,\star},
\boldsymbol P_f \boldsymbol e_u^{h,n+1}\big\rangle_\Gamma
+\big\langle
L(\boldsymbol e_u^{h,\star}\!\cdot\!\boldsymbol n_f)
-e_\phi^{h,\star},
\boldsymbol e_u^{h,n+1}\!\cdot\!\boldsymbol n_f
\big\rangle_\Gamma
\notag\\
&\quad
+\gamma\big\langle \boldsymbol P_p \boldsymbol e_u^{h,\star},
\boldsymbol P_p \boldsymbol e_\xi^{h,n+1}\big\rangle_\Gamma
+\big\langle
\boldsymbol e_\xi^{h,\star}\!\cdot\!\boldsymbol n_p,
\boldsymbol e_\xi^{h,n+1}\!\cdot\!\boldsymbol n_p
\big\rangle_\Gamma
\notag\\
&\quad
+\big\langle
-\boldsymbol e_u^{h,\star}\!\cdot\!\boldsymbol n_p
+L^{-1}e_\phi^{h,\star},
e_\phi^{h,n+1}
\big\rangle_\Gamma,
\label{eq:Ih_def}
\end{align}
with the AB2 extrapolants
\[
\boldsymbol e_u^{h,\star}=2\boldsymbol e_u^{h,n}-\boldsymbol e_u^{h,n-1},
\quad
\boldsymbol e_\xi^{h,\star}=2\boldsymbol e_\xi^{h,n}-\boldsymbol e_\xi^{h,n-1},
\quad
e_\phi^{h,\star}=2e_\phi^{h,n}-e_\phi^{h,n-1}.
\]
The interface contribution \(\mathcal I_h^{n+1}\) has the same algebraic
structure as \(\mathcal I^{n+1}\) in the stability proof. We apply the
identical AB2 residual decomposition. Define the error second-differences
\begin{equation}
\boldsymbol r_{e_u}^{n+1}:=\boldsymbol e_u^{h,n+1}-2\boldsymbol e_u^{h,n}+\boldsymbol e_u^{h,n-1},
\quad
\boldsymbol r_{e_\xi}^{n+1}:=\boldsymbol e_\xi^{h,n+1}-2\boldsymbol e_\xi^{h,n}+\boldsymbol e_\xi^{h,n-1},
\quad
r_{e_\phi}^{n+1}:=e_\phi^{h,n+1}-2e_\phi^{h,n}+e_\phi^{h,n-1}.
\label{eq:error_second_diff}
\end{equation}
These satisfy the AB2 residual identity \eqref{eq:ab2_residual_identity}:
\[
\boldsymbol e_u^{h,\star} = \boldsymbol e_u^{h,n+1} - \boldsymbol r_{e_u}^{n+1},
\quad
\boldsymbol e_\xi^{h,\star} = \boldsymbol e_\xi^{h,n+1} - \boldsymbol r_{e_\xi}^{n+1},
\quad
e_\phi^{h,\star} = e_\phi^{h,n+1} - r_{e_\phi}^{n+1}.
\]
Note also that the BDF2 second-difference terms in \(\mathcal D_{h,\mathrm{bulk}}^{n+1}\) are
precisely \(\frac{1}{4\Delta t}\|\boldsymbol r_{e_u}^{n+1}\|^2\),
\(\frac{1}{4\Delta t}\|\boldsymbol r_{e_\xi}^{n+1}\|^2\), etc.; they are produced by
the same BDF2 identity step as \(\mathcal E_h^{n+1}-\mathcal E_h^n\) but are
separate positive terms, not contained in the energy difference.

Substituting into \(\mathcal I_h^{n+1}\) and separating current-step from
residual parts gives
\[
\mathcal I_h^{n+1} = \mathcal I_{h,\mathrm{cur}}^{n+1} + \mathcal R_{h,\Gamma}^{n+1},
\]
where
\begin{align}
\mathcal I_{h,\mathrm{cur}}^{n+1}
&:=
\gamma\langle \boldsymbol P_f\boldsymbol e_\xi^{h,n+1},\boldsymbol P_f\boldsymbol e_u^{h,n+1}\rangle_\Gamma
+\langle L(\boldsymbol e_u^{h,n+1}\!\cdot\!\boldsymbol n_f)-e_\phi^{h,n+1},
\boldsymbol e_u^{h,n+1}\!\cdot\!\boldsymbol n_f\rangle_\Gamma
\notag\\
&\quad
+\gamma\langle \boldsymbol P_p\boldsymbol e_u^{h,n+1},\boldsymbol P_p\boldsymbol e_\xi^{h,n+1}\rangle_\Gamma
+\langle \boldsymbol e_\xi^{h,n+1}\!\cdot\!\boldsymbol n_p,
\boldsymbol e_\xi^{h,n+1}\!\cdot\!\boldsymbol n_p\rangle_\Gamma
\notag\\
&\quad
+\langle -\boldsymbol e_u^{h,n+1}\!\cdot\!\boldsymbol n_p
+L^{-1}e_\phi^{h,n+1},e_\phi^{h,n+1}\rangle_\Gamma,
\label{eq:Ih_cur_def}
\end{align}
and
\begin{align}
\mathcal R_{h,\Gamma}^{n+1}
&:=
-\gamma\langle \boldsymbol P_f\boldsymbol r_{e_\xi}^{n+1},\boldsymbol P_f\boldsymbol e_u^{h,n+1}\rangle_\Gamma
-\gamma\langle \boldsymbol P_p\boldsymbol r_{e_u}^{n+1},\boldsymbol P_p\boldsymbol e_\xi^{h,n+1}\rangle_\Gamma
\notag\\
&\quad
-L\langle \boldsymbol r_{e_u}^{n+1}\!\cdot\!\boldsymbol n_f,
\boldsymbol e_u^{h,n+1}\!\cdot\!\boldsymbol n_f\rangle_\Gamma
+\langle r_{e_\phi}^{n+1},
\boldsymbol e_u^{h,n+1}\!\cdot\!\boldsymbol n_f\rangle_\Gamma
\notag\\
&\quad
-\langle \boldsymbol r_{e_\xi}^{n+1}\!\cdot\!\boldsymbol n_p,
\boldsymbol e_\xi^{h,n+1}\!\cdot\!\boldsymbol n_p\rangle_\Gamma
+\langle \boldsymbol r_{e_u}^{n+1}\!\cdot\!\boldsymbol n_p,
e_\phi^{h,n+1}\rangle_\Gamma
-L^{-1}\langle r_{e_\phi}^{n+1},e_\phi^{h,n+1}\rangle_\Gamma.
\label{eq:Rh_Gamma_def}
\end{align}
Since \(\boldsymbol n_p=-\boldsymbol n_f\) and \(\boldsymbol P_f=\boldsymbol P_p\) on \(\Gamma\),
\[
\mathcal I_{h,\mathrm{cur}}^{n+1}
=
2\gamma\langle \boldsymbol P_f\boldsymbol e_\xi^{h,n+1},
\boldsymbol P_f\boldsymbol e_u^{h,n+1}\rangle_\Gamma
+L\|\boldsymbol e_u^{h,n+1}\!\cdot\!\boldsymbol n_f\|_{L^2(\Gamma)}^2
+\|\boldsymbol e_\xi^{h,n+1}\!\cdot\!\boldsymbol n_p\|_{L^2(\Gamma)}^2
+L^{-1}\|e_\phi^{h,n+1}\|_{L^2(\Gamma)}^2.
\]
The normal and pore-pressure interface terms appear identically in
\(\mathcal D_{h,\Gamma}^{n+1}\), so they cancel exactly. The tangential terms
combine via \(\|a\|^2-2\langle a,b\rangle+\|b\|^2=\|a-b\|^2\), giving the
exact identity
\begin{equation}
\mathcal D_{h,\Gamma}^{n+1} - \mathcal I_{h,\mathrm{cur}}^{n+1}
=
\gamma\|\boldsymbol P_f\boldsymbol e_u^{h,n+1}
-\boldsymbol P_p\boldsymbol e_\xi^{h,n+1}\|_{L^2(\Gamma)}^2.
\label{eq:Ih_cur_identity}
\end{equation}
Substituting the decomposition
\(\mathcal I_h^{n+1}=\mathcal I_{h,\mathrm{cur}}^{n+1}+\mathcal R_{h,\Gamma}^{n+1}\)
into \eqref{eq:error_energy_identity} and using \eqref{eq:Ih_cur_identity} gives
\begin{align}
&\frac{1}{\Delta t}\big(\mathcal E_h^{n+1}-\mathcal E_h^n\big)
+\mathcal D_{h,\mathrm{bulk}}^{n+1}
+\gamma\|\boldsymbol P_f\boldsymbol e_u^{h,n+1}
-\boldsymbol P_p\boldsymbol e_\xi^{h,n+1}\|_{L^2(\Gamma)}^2
\notag\\
&\qquad\le
\mathcal R_{h,\Gamma}^{n+1}
+\mathfrak R_f^{n+1}(\boldsymbol e_u^{h,n+1})
+\widetilde{\mathfrak R}_p^{n+1}(\delta_t\boldsymbol e_\eta^{h,n+1},
\boldsymbol e_\xi^{h,n+1},e_\phi^{h,n+1})
+a_e(\boldsymbol e_\eta^{h,n+1},\boldsymbol\kappa^{n+1}).
\label{eq:error_energy_reduced}
\end{align}
We now estimate \(\mathcal R_{h,\Gamma}^{n+1}\) term by term, in direct analogy
with the stability residual estimates \eqref{eq:R1_alt}--\eqref{eq:R7_alt},
with \((\boldsymbol r_u,\boldsymbol r_\xi,r_\phi,\boldsymbol u_h,\boldsymbol\xi_h,\phi_h)\)
replaced by \((\boldsymbol r_{e_u},\boldsymbol r_{e_\xi},r_{e_\phi},
\boldsymbol e_u^h,\boldsymbol e_\xi^h,e_\phi^h)\).
Using Cauchy--Schwarz, the discrete trace inequalities
\eqref{eq:disc_trace_fluid}--\eqref{eq:disc_trace_inverse_poro}, and Young's
inequality, we obtain
\begin{align}
\gamma|\langle \boldsymbol P_f\boldsymbol r_{e_\xi}^{n+1},
\boldsymbol P_f\boldsymbol e_u^{h,n+1}\rangle_\Gamma|
&\le
\frac{\rho_p}{32\Delta t}\|\boldsymbol r_{e_\xi}^{n+1}\|_{L^2(\Omega_p)}^2
+C\frac{\gamma^2\Delta t}{\rho_p h^2}\|\boldsymbol e_u^{h,n+1}\|_{L^2(\Omega_f)}^2
+\frac{\mu_f}{8}\|\boldsymbol D(\boldsymbol e_u^{h,n+1})\|_{L^2(\Omega_f)}^2,
\label{eq:Re1}
\\
\gamma|\langle \boldsymbol P_p\boldsymbol r_{e_u}^{n+1},
\boldsymbol P_p\boldsymbol e_\xi^{h,n+1}\rangle_\Gamma|
&\le
\frac{\rho_f}{32\Delta t}\|\boldsymbol r_{e_u}^{n+1}\|_{L^2(\Omega_f)}^2
+C\frac{\gamma^2\Delta t}{\rho_f h^2}\|\boldsymbol e_\xi^{h,n+1}\|_{L^2(\Omega_p)}^2,
\label{eq:Re2}
\\
L|\langle \boldsymbol r_{e_u}^{n+1}\!\cdot\!\boldsymbol n_f,
\boldsymbol e_u^{h,n+1}\!\cdot\!\boldsymbol n_f\rangle_\Gamma|
&\le
\frac{\rho_f}{32\Delta t}\|\boldsymbol r_{e_u}^{n+1}\|_{L^2(\Omega_f)}^2
+\frac{\mu_f}{8}\|\boldsymbol D(\boldsymbol e_u^{h,n+1})\|_{L^2(\Omega_f)}^2
\notag\\
&\quad
+C\frac{L^2\Delta t}{\rho_f h^2}\|\boldsymbol e_u^{h,n+1}\|_{L^2(\Omega_f)}^2
+C\frac{L^2\Delta t}{\rho_f}\|\boldsymbol D(\boldsymbol e_u^{h,n+1})\|_{L^2(\Omega_f)}^2,
\label{eq:Re3}
\\
|\langle r_{e_\phi}^{n+1},\boldsymbol e_u^{h,n+1}\!\cdot\!\boldsymbol n_f\rangle_\Gamma|
&\le
\frac{C_0}{32\Delta t}\|r_{e_\phi}^{n+1}\|_{L^2(\Omega_p)}^2
+\frac{\mu_f}{8}\|\boldsymbol D(\boldsymbol e_u^{h,n+1})\|_{L^2(\Omega_f)}^2
\notag\\
&\quad
+C\frac{\Delta t}{C_0 h^2}\|\boldsymbol e_u^{h,n+1}\|_{L^2(\Omega_f)}^2
+C\frac{\Delta t}{C_0}\|\boldsymbol D(\boldsymbol e_u^{h,n+1})\|_{L^2(\Omega_f)}^2,
\label{eq:Re4}
\\
|\langle \boldsymbol r_{e_\xi}^{n+1}\!\cdot\!\boldsymbol n_p,
\boldsymbol e_\xi^{h,n+1}\!\cdot\!\boldsymbol n_p\rangle_\Gamma|
&\le
\frac{\rho_p}{32\Delta t}\|\boldsymbol r_{e_\xi}^{n+1}\|_{L^2(\Omega_p)}^2
+C\frac{\Delta t}{\rho_p h^2}\|\boldsymbol e_\xi^{h,n+1}\|_{L^2(\Omega_p)}^2,
\label{eq:Re5}
\\
|\langle \boldsymbol r_{e_u}^{n+1}\!\cdot\!\boldsymbol n_p,
e_\phi^{h,n+1}\rangle_\Gamma|
&\le
\frac{\rho_f}{32\Delta t}\|\boldsymbol r_{e_u}^{n+1}\|_{L^2(\Omega_f)}^2
+C\frac{\Delta t}{\rho_f h^2}\|e_\phi^{h,n+1}\|_{L^2(\Omega_p)}^2
+\frac{k_0}{8}\|\nabla e_\phi^{h,n+1}\|_{L^2(\Omega_p)}^2,
\label{eq:Re6}
\\
L^{-1}|\langle r_{e_\phi}^{n+1},e_\phi^{h,n+1}\rangle_\Gamma|
&\le
\frac{C_0}{32\Delta t}\|r_{e_\phi}^{n+1}\|_{L^2(\Omega_p)}^2
+C\frac{\Delta t}{C_0 L^2 h^2}\|e_\phi^{h,n+1}\|_{L^2(\Omega_p)}^2
+\frac{k_0}{8}\|\nabla e_\phi^{h,n+1}\|_{L^2(\Omega_p)}^2.
\label{eq:Re7}
\end{align}
The fractions \(1/32\) and \(1/8\) are chosen so that the BDF2
second-difference terms on the right-hand sides are bounded by the
corresponding terms in \(\mathcal D_{h,\mathrm{bulk}}^{n+1}\), and the
gradient terms are absorbed by the viscous and Darcy dissipation.
Combining \eqref{eq:Re1}--\eqref{eq:Re7} with \eqref{eq:K_coercive} gives
\begin{equation}
\mathcal R_{h,\Gamma}^{n+1}
\le
\frac{1}{4}\mathcal D_{h,\mathrm{bulk}}^{n+1}
+C_1\Delta t\|\boldsymbol D(\boldsymbol e_u^{h,n+1})\|_{L^2(\Omega_f)}^2
+C_2\frac{\Delta t}{h^2}\mathcal E_h^{n+1},
\label{eq:Rh_Gamma_est}
\end{equation}
where \(C_1\) and \(C_2\) are the same constants as in
\eqref{eq:def_C1}--\eqref{eq:def_C2}.
Under the condition \(\Delta t\le\Delta t_0\), the term
\(C_1\Delta t\|\boldsymbol D(\boldsymbol e_u^{h,n+1})\|^2\) is absorbed into
the viscous part of \(\mathcal D_{h,\mathrm{bulk}}^{n+1}\), exactly as in
\eqref{eq:small_dt_absorb_corrected}, so
\begin{equation}
\mathcal R_{h,\Gamma}^{n+1}
\le
\frac{1}{2}\mathcal D_{h,\mathrm{bulk}}^{n+1}
+C_2\frac{\Delta t}{h^2}\mathcal E_h^{n+1}.
\label{eq:Ih_est}
\end{equation}
For the fluid residual, using \eqref{eq:discrete_korn} and
\eqref{eq:div_from_korn}, Cauchy-Schwarz and Young's inequality, we have:
\begin{align}
|\mathfrak R_f^{n+1}(\boldsymbol e_u^{h,n+1})|
&\le
\rho_f\|\delta_t\Pi_u\boldsymbol u^{n+1}-\partial_t\boldsymbol u^{n+1}\|_{L^2\left(\Omega_f\right)}
\|\boldsymbol e_u^{h,n+1}\|_{L^2\left(\Omega_f\right)}
\notag\\
&\quad
+2\mu_f\|\boldsymbol e_u^{I,n+1}\|_{H^1\left(\Omega_f\right)}\,
\|\boldsymbol e_u^{h,n+1}\|_{H^1\left(\Omega_f\right)}
+\|e_p^{I,n+1}\|_{L^2\left(\Omega_f\right)}\,
\|\nabla\!\cdot\!\boldsymbol e_u^{h,n+1}\|_{L^2\left(\Omega_f\right)}
\notag\\
&\quad
+\gamma\|\boldsymbol e_u^{I,n+1}\|_{L^2\left(\Gamma\right)}\,
\|\boldsymbol P_f\boldsymbol e_u^{h,n+1}\|_{L^2\left(\Gamma\right)}
+L\|\boldsymbol e_u^{I,n+1}\!\cdot\!\boldsymbol n_f\|_{L^2\left(\Gamma\right)}\,
\|\boldsymbol e_u^{h,n+1}\!\cdot\!\boldsymbol n_f\|_{L^2\left(\Gamma\right)}
\notag\\
&\quad
+\gamma\|\Pi_\xi\boldsymbol\xi^{\star}-\boldsymbol\xi^{n+1}\|_{L^2\left(\Gamma\right)}
\|\boldsymbol P_f\boldsymbol e_u^{h,n+1}\|_{L^2\left(\Gamma\right)}
\notag\\
&\quad
+\|L(\Pi_u\boldsymbol u^{\star}-\boldsymbol u^{n+1})\!\cdot\!\boldsymbol n_f
-(\Pi_{\phi}\phi^{\star}-\phi^{n+1})\|_{L^2\left(\Gamma\right)}
\|\boldsymbol e_u^{h,n+1}\!\cdot\!\boldsymbol n_f\|_{L^2\left(\Gamma\right)}
\notag\\
&\le
\frac18\mathcal D_{h,\mathrm{bulk}}^{n+1}
+
C\mathcal E_h^{n+1}
+
C\mathcal T_{\mathrm{bulk}}^{n+1}
+
Ch^{-1}\mathcal T_\Gamma^{n+1}.
\label{eq:Rf_test_bound}
\end{align}
For the reduced poroelastic residual \(\widetilde{\mathfrak R}_p^{n+1}\)
defined in \eqref{eq:Rp_tilde_def}, the kinematic defect
\((\boldsymbol\kappa^{n+1},\boldsymbol v_{p,h})_p\) has already been removed
by the cancellation described above.
We first integrate by parts for the term below:
\begin{align}
-\alpha(e_\phi^{I,n+1},\nabla\!\cdot\!\boldsymbol e_\xi^{h,n+1})_p
&=
\alpha(\nabla e_\phi^{I,n+1},\boldsymbol e_\xi^{h,n+1})_p
-\alpha\langle e_\phi^{I,n+1},
\boldsymbol e_\xi^{h,n+1}\!\cdot\!\boldsymbol n_p\rangle_\Gamma.
\label{eq:eI_phi_parts}
\end{align}
And then, we have:
\begin{align}
&|\widetilde{\mathfrak R}_p^{n+1}(\delta_t\boldsymbol e_\eta^{h,n+1},
\boldsymbol e_\xi^{h,n+1},e_\phi^{h,n+1})|
\\
&\le
\rho_p\|\delta_t\Pi_\xi\boldsymbol\xi^{n+1}-\partial_t\boldsymbol\xi^{n+1}\|_{L^2\left(\Omega_p\right)}
\|\boldsymbol e_\xi^{h,n+1}\|_{L^2\left(\Omega_p\right)}
\notag\\
&\quad
+C_0\|\delta_t\Pi_{\phi}\phi^{n+1}-\partial_t\phi^{n+1}\|_{L^2\left(\Omega_p\right)}
\|e_\phi^{h,n+1}\|_{L^2\left(\Omega_p\right)}
\notag\\
&\quad
+\alpha\|\nabla e_\phi^{I,n+1}\|_{L^2\left(\Omega_p\right)}\,
\|\boldsymbol e_\xi^{h,n+1}\|_{L^2\left(\Omega_p\right)}
+\alpha\|e_\phi^{I,n+1}\|_{L^2\left(\Gamma\right)}\,
\|\boldsymbol e_\xi^{h,n+1}\!\cdot\!\boldsymbol n_p\|_{L^2\left(\Gamma\right)}
\notag\\
&\quad
+\alpha\|\nabla\!\cdot\!\boldsymbol e_\xi^{I,n+1}\|_{L^2(\Omega_p)}
\|e_\phi^{h,n+1}\|_{L^2(\Omega_p)}
\notag\\
&\quad
+\gamma\|\boldsymbol e_\xi^{I,n+1}\|_{L^2(\Gamma)}
\|\boldsymbol P_p\boldsymbol e_\xi^{h,n+1}\|_{L^2(\Gamma)}
+\|\boldsymbol e_\xi^{I,n+1}\!\cdot\!\boldsymbol n_p\|_{L^2(\Gamma)}
\|\boldsymbol e_\xi^{h,n+1}\!\cdot\!\boldsymbol n_p\|_{L^2(\Gamma)}
\notag\\
&\quad
+\|e_\phi^{I,n+1}\|_{L^2(\Gamma)}
\|\boldsymbol e_\xi^{h,n+1}\!\cdot\!\boldsymbol n_p\|_{L^2(\Gamma)}
+\|\boldsymbol e_\xi^{I,n+1}\!\cdot\!\boldsymbol n_p\|_{L^2(\Gamma)}
\|e_\phi^{h,n+1}\|_{L^2(\Gamma)}
\notag\\
&\quad
+\gamma\|\Pi_u\boldsymbol u^{\star}-\boldsymbol u^{n+1}\|_{L^2(\Gamma)}
\|\boldsymbol P_p\boldsymbol e_\xi^{h,n+1}\|_{L^2(\Gamma)}
\notag\\
&\quad
+\|\Pi_\xi\boldsymbol\xi^{\star}-\boldsymbol\xi^{n+1}\|_{L^2\left(\Gamma\right)}
\|\boldsymbol e_\xi^{h,n+1}\!\cdot\!\boldsymbol n_p\|_{L^2\left(\Gamma\right)}
\notag\\
&\quad
+\|-(\Pi_u\boldsymbol u^{\star}-\boldsymbol u^{n+1})\!\cdot\!\boldsymbol n_p
+L^{-1}(\Pi_{\phi}\phi^{\star}-\phi^{n+1})\|_{L^2\left(\Gamma\right)}
\|e_\phi^{h,n+1}\|_{L^2\left(\Gamma\right)}
\notag\\
&\le
\frac18\mathcal D_{h,\mathrm{bulk}}^{n+1}
+
C\mathcal E_h^{n+1}
+
C\mathcal T_{\mathrm{bulk}}^{n+1}
+
Ch^{-1}\mathcal T_\Gamma^{n+1}.
\label{eq:Rp_test_bound}
\end{align}
Finally, for the elastic cross term, we have:
\begin{align}
|a_e(\boldsymbol e_\eta^{h,n+1},\boldsymbol\kappa^{n+1})|
&\le
C\|\boldsymbol e_\eta^{h,n+1}\|_{a_e}\,\|\boldsymbol\kappa^{n+1}\|_{H^1\left(\Omega_p\right)}
\le
C\mathcal E_h^{n+1}
+
C\|\boldsymbol\kappa^{n+1}\|_{H^1\left(\Omega_p\right)}^2.
\label{eq:kappa_cross_bound_1}
\end{align}

After inserting \eqref{eq:Ih_est}, \eqref{eq:Rf_test_bound},
\eqref{eq:Rp_test_bound}, and
\eqref{eq:kappa_cross_bound_1} into \eqref{eq:error_energy_reduced}, and
absorbing the remaining dissipative pieces into the left-hand side under
the CFL condition \(\Delta t\le c_\ast h^2\), yields
\eqref{eq:one_step_error}. The BDF2 second-difference terms in
\(\mathcal D_{h,\mathrm{bulk}}^{n+1}\) that were used to absorb the AB2
residuals in \eqref{eq:Re1}--\eqref{eq:Re7} are then dropped as nonnegative
remainders.
\end{proof}

\subsection{A priori error estimate}
\label{subsec:error_residual_bounds_new}

\begin{lemma}[Bounds for the defect indicators]
\label{lem:residual_bounds_new}
Under \eqref{eq:reg_u_error}-\eqref{eq:reg_phi_error}, there exists a constant
\(C>0\), independent of \(h\), \(\Delta t\), and \(n\), such that
\begin{align}
\mathcal T_{\mathrm{bulk}}^{n+1}
&\le
C h^{2k}
+C\Delta t^3\int_{I_n}
\Big(
\|\partial_t^3\boldsymbol u(s)\|_{L^2\left(\Omega_f\right)}^2
+\|\partial_t^3\boldsymbol\xi(s)\|_{L^2\left(\Omega_p\right)}^2
+\|\partial_t^3\boldsymbol\eta(s)\|_{H^1\left(\Omega_p\right)}^2
+\|\partial_t^3\phi(s)\|_{L^2\left(\Omega_p\right)}^2
\Big)\,ds,
\label{eq:defect_bulk_pointwise}
\end{align}
and
\begin{align}
\mathcal T_\Gamma^{n+1}
&\le
C h^{2k+1}
+C\Delta t^3\int_{I_n}
\Big(
\|\partial_t^2\boldsymbol u(s)\|_{H^1\left(\Omega_f\right)}^2
+\|\partial_t^2\boldsymbol\xi(s)\|_{H^1\left(\Omega_p\right)}^2
+\|\partial_t^2\phi(s)\|_{H^1\left(\Omega_p\right)}^2
\Big)\,ds.
\label{eq:defect_Gamma_pointwise}
\end{align}
Consequently, the weighted defect sum satisfies
\begin{equation}
\Delta t\sum_{n=1}^{N-1}\big(\mathcal T_{\mathrm{bulk}}^{n+1}
+h^{-1}\mathcal T_\Gamma^{n+1}\big)
\le
C\big(h^{2k}+\Delta t^4+h^{-1}\Delta t^4\big).
\label{eq:defect_sum_weighted}
\end{equation}
Under the parabolic CFL condition \(\Delta t\le c_\ast h^2\) and \(1\le k\le 3\),
\begin{equation}
h^{-1}\Delta t^4\le c_\ast^4 h^{-1}h^8 = c_\ast^4 h^7 \le C h^{2k},
\label{eq:cfl_absorption}
\end{equation}
so
\begin{equation}
\Delta t\sum_{n=1}^{N-1}\big(\mathcal T_{\mathrm{bulk}}^{n+1}
+h^{-1}\mathcal T_\Gamma^{n+1}\big)
\le
C\big(h^{2k}+\Delta t^4\big).
\label{eq:defect_sum}
\end{equation}
\end{lemma}

\begin{proof}
We estimate the bulk and interface defect indicators separately.

For the fluid bulk time defect,
\[
\delta_t\Pi_u\boldsymbol u^{n+1}-\partial_t\boldsymbol u^{n+1}
=
\Pi_u(\delta_t\boldsymbol u^{n+1}-\partial_t\boldsymbol u^{n+1})
+
(\Pi_u\partial_t\boldsymbol u^{n+1}-\partial_t\boldsymbol u^{n+1}).
\]
Using the \(L^2\)-stability of \(\Pi_u\), Lemma~\ref{lem:bdf2_defect},
and \eqref{eq:Fortin_approx}, we obtain
\[
\|\delta_t\Pi_u\boldsymbol u^{n+1}-\partial_t\boldsymbol u^{n+1}\|_{L^2\left(\Omega_f\right)}^2
\le
C h^{2k+2}
+
C\Delta t^3\int_{I_n}\|\partial_t^3\boldsymbol u(s)\|_{L^2\left(\Omega_f\right)}^2\,ds.
\]
The bulk spatial fluid projection terms satisfy
\[
\|\boldsymbol e_u^{I,n+1}\|_{H^1\left(\Omega_f\right)}^2
+\|e_p^{I,n+1}\|_{L^2\left(\Omega_f\right)}^2
\le C h^{2k}.
\]
Similarly,
\[
\|\delta_t\Pi_\xi\boldsymbol\xi^{n+1}-\partial_t\boldsymbol\xi^{n+1}\|_{L^2\left(\Omega_p\right)}^2
\le
C h^{2k+2}
+
C\Delta t^3\int_{I_n}\|\partial_t^3\boldsymbol\xi(s)\|_{L^2\left(\Omega_p\right)}^2\,ds,
\]
and
\[
\|\delta_t\Pi_{\phi}\phi^{n+1}-\partial_t\phi^{n+1}\|_{L^2\left(\Omega_p\right)}^2
\le
C h^{2k+2}
+
C\Delta t^3\int_{I_n}\|\partial_t^3\phi(s)\|_{L^2\left(\Omega_p\right)}^2\,ds.
\]
For the Biot projection errors, we have
\[
\|\nabla e_\phi^{I,n+1}\|_{L^2\left(\Omega_p\right)}^2
+\|\nabla\!\cdot\!\boldsymbol e_\xi^{I,n+1}\|_{L^2\left(\Omega_p\right)}^2
\le
C h^{2k}.
\]
The kinematic defect is controlled by \eqref{eq:kappa_bound_H1}.
Collecting these bounds proves \eqref{eq:defect_bulk_pointwise}.

By the trace approximation estimate \eqref{eq:Fortin_trace_approx},
\[
\|\boldsymbol e_u^{I,n+1}\|_{L^2\left(\Gamma\right)}^2
\le C h^{2k+1}.
\]
Similarly, by the corresponding poroelastic trace estimates,
\[
\|e_\phi^{I,n+1}\|_{L^2\left(\Gamma\right)}^2
\le C h^{2k+1},
\qquad
\|\boldsymbol e_\xi^{I,n+1}\|_{L^2(\Gamma)}^2
\le C h^{2k+1}.
\]

For the AB2 interface defects, we write, for example,
\[
\Pi_\xi\boldsymbol\xi^{\star}-\boldsymbol\xi^{n+1}
=
(\Pi_\xi\boldsymbol\xi^{\star}-\boldsymbol\xi^{\star})
+
(\boldsymbol\xi^{\star}-\boldsymbol\xi^{n+1}).
\]
Since \(\Pi_\xi\), \(\Pi_u\), and \(\Pi_{\phi}\) are linear and time-independent,
they commute with AB2 extrapolation, i.e.\
\(\Pi_\xi\boldsymbol\xi^{\star} = \Pi_\xi(2\boldsymbol\xi^n-\boldsymbol\xi^{n-1})\),
and analogously for the other projections.
Hence the first term in each splitting is bounded by the corresponding
trace projection estimate, and the second by Lemma~\ref{lem:ab2_defect}
in \(H^1\) followed by the trace inequality:
$$
\|\Pi_\xi\boldsymbol\xi^{\star}-\boldsymbol\xi^{n+1}\|_{L^2\left(\Gamma\right)}^2
\le
C h^{2k+1}
+
C\Delta t^3\int_{I_n}\|\partial_t^2\boldsymbol\xi(s)\|_{H^1\left(\Omega_p\right)}^2\,ds.
$$
Applying the same argument to \(\Pi_u\boldsymbol u^{\star}\) and
\(\Pi_{\phi}\phi^{\star}\),
\begin{align*}
\|\Pi_u\boldsymbol u^{\star}-\boldsymbol u^{n+1}\|_{L^2\left(\Gamma\right)}^2
&\le
C h^{2k+1}
+
C\Delta t^3\int_{I_n}\|\partial_t^2\boldsymbol u(s)\|_{H^1\left(\Omega_f\right)}^2\,ds,
\\
\|\Pi_{\phi}\phi^{\star}-\phi^{n+1}\|_{L^2\left(\Gamma\right)}^2
&\le
C h^{2k+1}
+
C\Delta t^3\int_{I_n}\|\partial_t^2\phi(s)\|_{H^1\left(\Omega_p\right)}^2\,ds.
\end{align*}
Collecting these bounds proves \eqref{eq:defect_Gamma_pointwise}.

From \eqref{eq:defect_bulk_pointwise}, multiplying by \(\Delta t\) and summing
over \(n\) (using uniform overlap multiplicity of the intervals \(I_n\)) gives
\[
\Delta t\sum_{n=1}^{N-1}\mathcal T_{\mathrm{bulk}}^{n+1}
\le C(h^{2k}+\Delta t^4).
\]
From \eqref{eq:defect_Gamma_pointwise}, the \(h^{-1}\)-weighted interface
defect satisfies
\[
h^{-1}\mathcal T_\Gamma^{n+1}
\le C h^{2k} + Ch^{-1}\Delta t^3\int_{I_n}
\Big(
\|\partial_t^2\boldsymbol u(s)\|_{H^1(\Omega_f)}^2
+\|\partial_t^2\boldsymbol\xi(s)\|_{H^1(\Omega_p)}^2
+\|\partial_t^2\phi(s)\|_{H^1(\Omega_p)}^2
\Big)\,ds,
\]
so that
\[
\Delta t\sum_{n=1}^{N-1}h^{-1}\mathcal T_\Gamma^{n+1}
\le C(h^{2k}+h^{-1}\Delta t^4).
\]
Combining gives \eqref{eq:defect_sum_weighted}. Under the parabolic CFL
condition \(\Delta t\le c_\ast h^2\) and \(1\le k\le 3\), equation
\eqref{eq:cfl_absorption} shows that \(h^{-1}\Delta t^4=O(h^7)\le Ch^{2k}\),
which proves \eqref{eq:defect_sum}.
\end{proof}

\begin{theorem}[A priori error estimate in bulk energy norms]
\label{thm:error}
Assume \eqref{eq:reg_u_error}-\eqref{eq:reg_phi_error}, the hypotheses of
Theorem~\ref{thm:stability_cfl}, the second-order accurate initialization
condition
\begin{equation}
\mathcal E_h^1\le C\big(h^{2k}+\Delta t^4\big),
\label{eq:init_error_assumption_new}
\end{equation}
and \(1\le k\le 3\).
Then there exists a constant \(C_T>0\), independent of \(h\) and \(\Delta t\),
such that
\begin{align}
&\max_{2\le n\le N}
\Big(
\|\boldsymbol e_u^{h,n}\|_{L^2\left(\Omega_f\right)}^2
+\|\boldsymbol e_\xi^{h,n}\|_{L^2\left(\Omega_p\right)}^2
+\|e_\phi^{h,n}\|_{L^2\left(\Omega_p\right)}^2
+\|\boldsymbol e_\eta^{h,n}\|_{a_e}^2
\Big)
\notag\\
&\qquad
+\Delta t\sum_{n=1}^{N-1}
\Big(
\|\boldsymbol e_u^{h,n+1}\|_{H^1\left(\Omega_f\right)}^2
+\|e_\phi^{h,n+1}\|_{H^1\left(\Omega_p\right)}^2
+\gamma\|\boldsymbol P_f\boldsymbol e_u^{h,n+1}
-\boldsymbol P_p\boldsymbol e_\xi^{h,n+1}\|_{L^2\left(\Gamma\right)}^2
\Big)
\le
C_T\big(h^{2k}+\Delta t^4\big).
\label{eq:eh_error_bound_new}
\end{align}
Consequently, the total errors satisfy
\begin{align}
&\max_{2\le n\le N}
\Big(
\|\boldsymbol u^n-\boldsymbol u_h^n\|_{L^2\left(\Omega_f\right)}
+\|\boldsymbol\xi^n-\boldsymbol\xi_h^n\|_{L^2\left(\Omega_p\right)}
+\|\phi^n-\phi_h^n\|_{L^2\left(\Omega_p\right)}
+\|\boldsymbol\eta^n-\boldsymbol\eta_h^n\|_{a_e}
\Big)
\le
C_T\big(h^{k}+\Delta t^2\big),
\label{eq:total_error_bound_new}
\\
&\text{and}\\
&\left(
\Delta t\sum_{n=1}^{N-1}
\big(
\|\boldsymbol u^{n+1}-\boldsymbol u_h^{n+1}\|_{H^1\left(\Omega_f\right)}^2
+\|\phi^{n+1}-\phi_h^{n+1}\|_{H^1\left(\Omega_p\right)}^2
\big)
\right)^{1/2}
\le
C_T\big(h^{k}+\Delta t^2\big).
\label{eq:bulk_total_error_bound_new}
\end{align}
\end{theorem}

\begin{proof}
Starting from \eqref{eq:one_step_error}, summing over \(n=1,\dots,m\), and
using \eqref{eq:defect_sum}, we obtain
\begin{align}
\mathcal E_h^{m+1}
+\frac{\Delta t}{4}\sum_{n=1}^{m}\mathcal D_{h,\mathrm{bulk}}^{n+1}
+\gamma\Delta t\sum_{n=1}^{m}
\|\boldsymbol P_f\boldsymbol e_u^{h,n+1}
-\boldsymbol P_p\boldsymbol e_\xi^{h,n+1}\|_{L^2(\Gamma)}^2\\
\le
\mathcal E_h^1
+
C\Delta t\sum_{n=1}^{m}
\Big(1+\frac{\Delta t}{h^2}\Big)\mathcal E_h^{n+1}
+
C\big(h^{2k}+\Delta t^4\big).
\end{align}
Under the CFL condition \(\Delta t\le c_\ast h^2\), the factor
\(1+\Delta t/h^2\) is uniformly bounded, so the discrete Gronwall yields \eqref{eq:eh_error_bound_new}.

To derive the total error estimate, we use the decompositions
\eqref{eq:errsplit_u_new}-\eqref{eq:errsplit_phi_new}. For instance,
\[
\|\boldsymbol u^n-\boldsymbol u_h^n\|_{L^2\left(\Omega_f\right)}
\le
\|\boldsymbol e_u^{I,n}\|_{L^2\left(\Omega_f\right)}
+\|\boldsymbol e_u^{h,n}\|_{L^2\left(\Omega_f\right)},
\]
and similarly for \(\boldsymbol\xi\), \(\phi\), and \(\boldsymbol\eta\). The approximation
estimates \eqref{eq:Fortin_approx}, \eqref{eq:proj_eta}-\eqref{eq:proj_phi}
yield
\[
\|\boldsymbol e_u^{I,n}\|_{L^2\left(\Omega_f\right)}
+\|\boldsymbol e_\xi^{I,n}\|_{L^2\left(\Omega_p\right)}
+\|e_\phi^{I,n}\|_{L^2\left(\Omega_p\right)}
+\|\boldsymbol e_\eta^{I,n}\|_{a_e}
\le C h^k.
\]
Therefore, \eqref{eq:total_error_bound_new} follows from
\eqref{eq:eh_error_bound_new}. Moreover, by \eqref{eq:discrete_korn},
\[
\|\boldsymbol e_u^{h,n+1}\|_{H^1\left(\Omega_f\right)}
\le C_K\|\boldsymbol D(\boldsymbol e_u^{h,n+1})\|_{L^2\left(\Omega_f\right)},
\]
and thus the estimate \eqref{eq:bulk_total_error_bound_new} follows from
\eqref{eq:eh_error_bound_new}, \eqref{eq:Fortin_approx}, and
\eqref{eq:proj_phi}.
\end{proof}

\begin{remark}
\label{rem:pressure_omitted}
The estimate above is written in the bulk energy norms of the partitioned
scheme, together with the tangential interface mismatch. A separate optimal \(L^2(\Omega_f)\) estimate for the fluid pressure
requires an additional argument controlling
\(\delta_t\boldsymbol e_u^{h,n+1}\) in a norm compatible with the discrete
inf-sup condition. A sharp standalone fluid-pressure bound is possible in principle, but it needs extra work beyond the current proof, so we are not claiming it here.
\end{remark}

\begin{remark}[On the polynomial degree and interface trace losses]
\label{rem:trace_loss}
The estimates for \(\mathfrak R_f^{n+1}\) and \(\mathfrak R_p^{n+1}\) involve
interface consistency terms of the form
\(\|\Pi a^{\star}-a^{n+1}\|_{L^2(\Gamma)}\) paired with discrete variables
whose normal traces are controlled via the inverse trace inequality
\eqref{eq:disc_trace_inverse_poro}. After Cauchy--Schwarz and Young's
inequality, these pairings introduce weighted defect terms of the form
\(h^{-1}\|\Pi a^{\star}-a^{n+1}\|_{L^2(\Gamma)}^2\). Using the
approximation estimate \(\|\Pi a - a\|_{L^2(\Gamma)}^2 = O(h^{2k+1})\) and
the AB2 consistency bound \(O(\Delta t^4)\), the weighted terms contribute
\(O(h^{2k} + h^{-1}\Delta t^4)\). Under the parabolic CFL condition
\(\Delta t \lesssim h^2\), the second term is \(O(h^7)\). For the finite
element spaces used in this work (in particular Taylor--Hood
\(\mathbb P_2/\mathbb P_1\) with \(k=2\)), this contribution is absorbed
into the spatial error term. For higher-order spaces \(k \ge 4\), the term
\(h^{-1}\Delta t^4\) would require a separate treatment or a modified
statement of the estimate.
\end{remark}

\begin{remark}[On the fluid projection assumptions]
\label{rem:Fortin_Korn}
The error analysis in this section is conditional on two standard properties of
the discrete fluid space. First, the Stokes finite element pair
\((\boldsymbol V_{f,h},Q_{f,h})\) is assumed to admit a Fortin
interpolation/projection operator \(\Pi_u\) satisfying
\eqref{eq:Fortin_property}, \eqref{eq:Fortin_approx}, and
\eqref{eq:Fortin_trace_approx}. Second, the discrete fluid velocity space is
assumed to satisfy the discrete Korn inequality \eqref{eq:discrete_korn}. Under
these assumptions, the fluid residual terms can be controlled in a way that is
compatible with the natural energy estimate from
Theorem~\ref{thm:stability_cfl}.
Moreover, if the fluid discretization uses a standard conforming Stokes-stable pair such as Taylor–Hood or MINI on shape-regular meshes, then the Fortin projection and discrete Korn inequality required in Section 4 are already available in the literature \cite{Brenner-Scott:book, ScottZhang1990}.
\end{remark}

\section{Numerical validation}
\label{sec:numerics}

\subsection{Benchmark problem}
In this example, we considered a benchmark problem with manufactured solutions to examine the rates of convergence in time and space of the explicit splitting scheme. We solve the time-dependent Stokes--Biot system with added external forcing terms, {\sunny{given by the following}}:
$$
\left\{
\begin{array}{ll}
    \rho_f \partial_t \boldsymbol{u}=\nabla \cdot \boldsymbol{\sigma}_f\left(\boldsymbol{u}, p\right)+\boldsymbol{F}_f & \text { in } \Omega_f \times(0, T), \\
    \nabla \cdot \boldsymbol{u}=g_f & \text { in } \Omega_f \times(0, T), \\
    \partial_t \boldsymbol{\eta}=\boldsymbol{\xi} & \text { in } \Omega_p \times(0, T), \\
    \rho_p \partial_t \boldsymbol{\xi}=\nabla \cdot \boldsymbol{\sigma}_p\left(\boldsymbol{\eta}, \phi\right)+\boldsymbol{F}_e & \text { in } \Omega_p \times(0, T), \\
    \boldsymbol{u_p}=-\mathbb K\nabla \phi & \text { in } \Omega_p \times(0, T), \\
    C_0 \partial_t \phi+\alpha \nabla \cdot \boldsymbol{\xi}-\nabla \cdot (\mathbb K\nabla\phi)=F_d & \text { in } \Omega_p \times(0, T) .
\end{array}
\right.
$$
The FPSI problem is defined within {\sunny{the rectangular domain $\Omega = (0,1) \times (-1,1)$, where the fluid domain occupies the upper half of $\Omega$, i.e., $\Omega_f = (0,1) \times (0,1)$, and the solid domain occupies the lower half of $\Omega$, i.e., $\Omega_p = (0,1) \times (-1,0)$. The exact solution of this problem is given by:}}
\begin{align*}
& \boldsymbol{u}_{exact}=\pi \cos (\pi t)\left[\begin{array}{c}
-3 x+\cos (y) \\
y+1
\end{array}\right], \quad&&p_{exact}=e^t \sin (\pi x) \cos \left(\frac{\pi y}{2}\right)+2 \pi \cos (\pi t), \\
& \boldsymbol{\eta}_{exact}=\sin (\pi t)\left[\begin{array}{c}
-3 x+\cos (y) \\
y+1
\end{array}\right], \quad&&\phi_{exact}=e^t \sin (\pi x) \cos \left(\frac{\pi y}{2}\right).
\end{align*}
From the exact solutions, we can retrieve the corresponding forcing terms of $\boldsymbol{F}_f, g_f, \boldsymbol{F}_e$, and $F_d$:
\begin{equation}
\begin{aligned}
&\boldsymbol{F}_f=\left[\begin{array}{l}
\begin{aligned}
& \rho_f \pi^2 \sin (\pi t)(3 x-\cos y)+\pi e^t \cos (\pi x) \cos \left(\frac{\pi y}{2}\right)+\mu_f \pi \cos (\pi t) \cos y \\
& -\rho_f \pi^2 \sin (\pi t)(y+1)-\frac{\pi}{2} e^t \sin (\pi x) \sin \left(\frac{\pi y}{2}\right) 
\end{aligned}
\end{array}\right],  \\
&g_f=-2 \pi \cos (\pi t),\\
&\boldsymbol{F}_e=\left[\begin{array}{l}
\begin{aligned}
& \rho_p \pi^2 \sin(\pi t)(3 x-\cos y)+\alpha \pi e^t \cos (\pi x) \cos \left(\frac{\pi y}{2}\right)+\mu_p \sin(\pi t) \cos y \\
& -\rho_p \pi^2 \sin(\pi t)(y+1)-\alpha \frac{\pi}{2} e^t \sin (\pi x) \sin \left(\frac{\pi y}{2}\right) 
\end{aligned} 
\end{array}\right],  \\
&{F}_d=C_0 e^t \sin (\pi x) \cos \left(\frac{\pi y}{2}\right)-2 \alpha \pi \cos (\pi t)+\frac{5}{4} \pi^2 e^t \sin (\pi x) \cos \left(\frac{\pi y}{2}\right).
\end{aligned}
\end{equation}
{\sunny{We set the physical parameters all equal to one:}}
$$
\rho_p=\mu_p=\lambda_p=\alpha=C_0=\gamma=\rho_f=\mu_f=1, \quad \mathbb{K}=\mathbf{I}.
$$
{\sunny{Finite elements are used for spatial discretization. In particular, for the  fluid we use  Taylor-Hood elements $\mathbb{P}_2-\mathbb{P}_1$ for $\left(\boldsymbol{u}, p\right)$.}} For the Biot variables, we employ continuous \(\mathbb P_2\) elements for
the solid displacement \(\boldsymbol\eta\), continuous \(\mathbb P_2\)
elements for the structure velocity \(\boldsymbol\xi\), and continuous
\(\mathbb P_2\) elements for the pore pressure \(\phi\). The system is {\sunny{solved on the time interval $(0,T) = (0,0.1)$, and we evaluate the numerical error at $T = 0.1$}}. To compute convergence rates, we define the final time errors for {\sunny{structure displacement and velocity $(\boldsymbol{\eta},\boldsymbol{\xi})$, Darcy pressure $\phi$, and fluid velocity and pressure $\left(\boldsymbol{u}, p\right)$, as follows:}}
$$
\begin{aligned}
e_{\boldsymbol{\eta}}  & :=\left\|\boldsymbol{\eta}_h(T)-\boldsymbol{\eta}_{\text {exact }}(T)\right\|_{L^2\left(\Omega_p\right)}, \\
e_{\boldsymbol{\xi}}  & :=\left\|\boldsymbol{\xi}_h(T)-\boldsymbol{\xi}_{\text {exact }}(T)\right\|_{L^2\left(\Omega_p\right)} \\
e_\phi & :=\left\|\phi_h(T)-\phi_{\text {exact }}(T)\right\|_{L^2\left(\Omega_p\right)} \\
e_{\boldsymbol{u}} & :=\left\|\boldsymbol{u}_{f, h}(T)-\boldsymbol{u}_{\text {exact }}(T)\right\|_{L^2\left(\Omega_f\right)} \\
e_p & :=\left\|p_{f, h}(T)-p_{\text {exact}}(T)\right\|_{L^2\left(\Omega_f\right)} .
\end{aligned}
$$

\begin{table}[htbp]
\centering
\caption{Temporal errors for the fixed-domain Stokes-Biot scheme at $T=0.1$. The number in parentheses denotes the convergence rate.}
\label{tab:temporal_error_rate}
\renewcommand{\arraystretch}{1.15}
\begin{tabular}{cccccc}
\hline
$\Delta t$  & $u$ & $p$ & $\eta$ & $\xi$ & $\phi$ \\
\hline
$2.50\mathrm{e}-02$ 
& \begin{tabular}[c]{@{}c@{}}$5.99\mathrm{e}-04$\\ {--}\end{tabular}
& \begin{tabular}[c]{@{}c@{}}$2.60\mathrm{e}-02$\\ {--}\end{tabular}
& \begin{tabular}[c]{@{}c@{}}$3.27\mathrm{e}-03$\\ {--}\end{tabular}
& \begin{tabular}[c]{@{}c@{}}$1.92\mathrm{e}-02$\\ {--}\end{tabular}
& \begin{tabular}[c]{@{}c@{}}$1.09\mathrm{e}-03$\\ {--}\end{tabular}
\\[2mm]

$1.25\mathrm{e}-02$
& \begin{tabular}[c]{@{}c@{}}$1.50\mathrm{e}-04$\\ $(1.99)$\end{tabular}
& \begin{tabular}[c]{@{}c@{}}$6.50\mathrm{e}-03$\\ $(2.00)$\end{tabular}
& \begin{tabular}[c]{@{}c@{}}$8.26\mathrm{e}-04$\\ $(1.99)$\end{tabular}
& \begin{tabular}[c]{@{}c@{}}$5.03\mathrm{e}-03$\\ $(1.93)$\end{tabular}
& \begin{tabular}[c]{@{}c@{}}$2.80\mathrm{e}-04$\\ $(1.97)$\end{tabular}
\\[2mm]

$6.25\mathrm{e}-03$ 
& \begin{tabular}[c]{@{}c@{}}$3.80\mathrm{e}-05$\\ $(1.99)$\end{tabular}
& \begin{tabular}[c]{@{}c@{}}$1.62\mathrm{e}-03$\\ $(2.00)$\end{tabular}
& \begin{tabular}[c]{@{}c@{}}$2.07\mathrm{e}-04$\\ $(2.00)$\end{tabular}
& \begin{tabular}[c]{@{}c@{}}$1.28\mathrm{e}-03$\\ $(1.97)$\end{tabular}
& \begin{tabular}[c]{@{}c@{}}$7.03\mathrm{e}-05$\\ $(1.99)$\end{tabular}
\\[2mm]

$3.13\mathrm{e}-03$ 
& \begin{tabular}[c]{@{}c@{}}$9.59\mathrm{e}-06$\\ $(1.99)$\end{tabular}
& \begin{tabular}[c]{@{}c@{}}$4.06\mathrm{e}-04$\\ $(2.00)$\end{tabular}
& \begin{tabular}[c]{@{}c@{}}$5.19\mathrm{e}-05$\\ $(2.00)$\end{tabular}
& \begin{tabular}[c]{@{}c@{}}$3.23\mathrm{e}-04$\\ $(1.99)$\end{tabular}
& \begin{tabular}[c]{@{}c@{}}$1.76\mathrm{e}-05$\\ $(2.00)$\end{tabular}
\\
\hline
\end{tabular}
\end{table}

Table~\ref{tab:temporal_error_rate} shows the temporal convergence behavior of the proposed fixed-domain Stokes-Biot scheme at $T=0.1$. The errors in the fluid velocity $u$ and pressure $p$ decrease with rates very close to $2$, which is fully consistent with the second-order BDF2-AB2 time discretization. The displacement $\eta$ and structure velocity $\xi$ also display clear second-order trends, particularly on the finer time-step levels where the observed rates move closer to $2$. For the pore pressure $\phi$, the computed rates are also second order. This indicates that the method achieves the expected second-order accuracy in time overall. Small deviations from the ideal rate may be caused by the parallel splitting strategy itself. Since the interface coupling is enforced through extrapolated data from previous time steps, the scheme introduces a splitting residual at the fluid–poroelastic interface. This residual can affect the observed temporal rates, especially for variables that are more sensitive to the interface transmission conditions.

\begin{table}[htbp]
\centering
\caption{Spatial errors for the fixed-domain Stokes--Biot scheme at $T=0.1$ with fixed $\Delta t=10^{-4}$. The number in parentheses denotes the observed convergence rate.}
\label{tab:spatial_error_rate}
\renewcommand{\arraystretch}{1.15}
\begin{tabular}{cccccc}
\hline
$h$  & $u$ & $p$ & $\eta$ & $\xi$ & $\phi$ \\
\hline
$1.41\mathrm{e}-01$ 
& \begin{tabular}[c]{@{}c@{}}$5.48\mathrm{e}-05$\\ {--}\end{tabular}
& \begin{tabular}[c]{@{}c@{}}$3.12\mathrm{e}-03$\\ {--}\end{tabular}
& \begin{tabular}[c]{@{}c@{}}$1.35\mathrm{e}-06$\\ {--}\end{tabular}
& \begin{tabular}[c]{@{}c@{}}$5.94\mathrm{e}-05$\\ {--}\end{tabular}
& \begin{tabular}[c]{@{}c@{}}$1.47\mathrm{e}-04$\\ {--}\end{tabular}
\\[2mm]

$7.07\mathrm{e}-02$
& \begin{tabular}[c]{@{}c@{}}$7.40\mathrm{e}-06$\\ $(2.89)$\end{tabular}
& \begin{tabular}[c]{@{}c@{}}$7.56\mathrm{e}-04$\\ $(2.05)$\end{tabular}
& \begin{tabular}[c]{@{}c@{}}$1.35\mathrm{e}-07$\\ $(3.33)$\end{tabular}
& \begin{tabular}[c]{@{}c@{}}$7.75\mathrm{e}-06$\\ $(2.94)$\end{tabular}
& \begin{tabular}[c]{@{}c@{}}$1.85\mathrm{e}-05$\\ $(2.99)$\end{tabular}
\\[2mm]

$3.54\mathrm{e}-02$ 
& \begin{tabular}[c]{@{}c@{}}$9.84\mathrm{e}-07$\\ $(2.91)$\end{tabular}
& \begin{tabular}[c]{@{}c@{}}$1.87\mathrm{e}-04$\\ $(2.01)$\end{tabular}
& \begin{tabular}[c]{@{}c@{}}$1.75\mathrm{e}-08$\\ $(2.95)$\end{tabular}
& \begin{tabular}[c]{@{}c@{}}$9.85\mathrm{e}-07$\\ $(2.98)$\end{tabular}
& \begin{tabular}[c]{@{}c@{}}$2.31\mathrm{e}-06$\\ $(3.00)$\end{tabular}
\\[2mm]

$1.77\mathrm{e}-02$ 
& \begin{tabular}[c]{@{}c@{}}$1.29\mathrm{e}-07$\\ $(2.93)$\end{tabular}
& \begin{tabular}[c]{@{}c@{}}$4.63\mathrm{e}-05$\\ $(2.01)$\end{tabular}
& \begin{tabular}[c]{@{}c@{}}$8.79\mathrm{e}-09$\\ $(0.99)$\end{tabular}
& \begin{tabular}[c]{@{}c@{}}$1.86\mathrm{e}-07$\\ $(2.41)$\end{tabular}
& \begin{tabular}[c]{@{}c@{}}$2.90\mathrm{e}-07$\\ $(3.00)$\end{tabular}
\\
\hline
\end{tabular}
\end{table}

To study spatial convergence, we ran the fixed-domain Stokes-Biot scheme on a sequence of meshes with $10, 20, 40$, and $80$ elements in each spatial direction. The final time was set to $T=0.1$, while the time step was fixed at $\Delta t=10^{-4}$ for all runs. As reported in Table~\ref{tab:spatial_error_rate}, the errors in the fluid velocity $u$ and fluid pressure $p$, structure displacement $\eta$, structure velocity $\xi$, and pore pressure $\phi$ all decrease steadily as the mesh is refined. For all variables, we observed the expected optimal convergence rates, indicating that the spatial discretization performs well for the coupled problem. In particular, the fluid variables behave consistently with the Taylor-Hood $P 2-P 1$ approximation used in the Stokes subproblem, while the poroelastic variables show nearly third-order decay. The convergence rate for the displacement error deteriorates on the last refinement level because the temporal error has reached approximately $O(10^{-8})$.

\subsection{2D blood flow in a moving domain}

This example is outside the scope of the stability and error analysis proved in
Sections 3--4. It is included only as a robustness and applicability test of the
partitioned strategy in a more realistic moving-domain Navier--Stokes--Biot
setting.

We next consider another benchmark problem motivated by blood flow in a straight artery. Different from the linear problem (Stokes--Biot) in a fixed domain studied in \cite{Parrow2026R, he2026lockfree}, we consider the Navier Stokes--Biot problem in a moving domain.

Let $R>0$ denote the lumen radius, $L>0$ the vessel length, and $r_p>0$ the thickness of the poroelastic wall. The fluid reference domain is $\Omega_f=(0,L)\times(-R,R)$, while the poroelastic wall region is given by $\Omega_p=(0,L)\times(-R-r_p, R)\cup(0,L)\times(R,R+r_p)$.
The moving fluid domain is denoted by $\Omega_f(t)$ and is obtained through an
ALE map
\[
\mathcal{A}_t:\widehat{\Omega}_f\rightarrow \Omega_f(t),
\qquad
\mathcal{A}_t(\widehat{\mathbf{x}})
=
\widehat{\mathbf{x}}+\boldsymbol{\eta}_{\mathrm{ALE}}(\widehat{\mathbf{x}},t),
\]
where $\boldsymbol{\eta}_{\mathrm{ALE}}$ is the ALE mesh displacement. The mesh
velocity is defined by
\[
\boldsymbol{w}
=
\partial_t \mathcal{A}_t
=
\partial_t \boldsymbol{\eta}_{\mathrm{ALE}} .
\]
In the time-discrete implementation, the mesh velocity is approximated by
\[
\boldsymbol{w}^{n+1}
=
\frac{
\boldsymbol{\eta}_{\mathrm{ALE}}^{n+1}
-
\boldsymbol{\eta}_{\mathrm{ALE}}^{n}
}{\Delta t}.
\]
The fluid inlet and outlet boundaries are defined as $$\Gamma_f^{\mathrm{in}}=\{(0,y)\mid -R<y<R\}, \qquad \Gamma_f^{\mathrm{out}}=\{(L,y)\mid -R<y<R\}.$$ 
The corresponding inlet and outlet boundaries of the poroelastic wall are 
\begin{align}
&\Gamma_p^{\mathrm{in}}=\{(0,y)\mid -R-r_p<y<-R \text{ or } R<y<R+r_p\}, \nonumber\\
&\Gamma_p^{\mathrm{out}}=\{(L,y)\mid -R-r_p<y<-R \text{ or } R<y<R+r_p\}.\nonumber
\end{align} 
The external boundary of the wall is 
$$
\Gamma_p^{\mathrm{ext}}=\{(x,y)\mid 0<x<L,\; y=-R-r_p \text{ or } y=R+r_p\},
$$
and the fluid--poroelastic interfaces are located at 
$$
\Gamma_{fp}^{-}=(0,L)\times\{-R\}, \qquad \Gamma_{fp}^{+}=(0,L)\times\{R\}.
$$

In the lumen, the blood flow is modeled by the Navier Stokes equation on the moving domain $\Omega_f(t)$:
\begin{equation}
\rho_f
\left[
\left.\frac{\partial \boldsymbol{u}_f}{\partial t}\right|_{\hat{x}}
+
\left(
(\boldsymbol{u}_f-\boldsymbol{w})\cdot\nabla
\right)\boldsymbol{u}_f
\right]
-
\nabla\cdot\boldsymbol{\sigma}_f(\boldsymbol{u}_f,p_f)
=
\mathbf{f}_f
\qquad
\text{in } \Omega_f(t)\times(0,T].
\label{eq:benchmark_ALE_NS_momentum}
\end{equation}
Here, the time derivative is taken with respect to the fixed reference coordinate
$\widehat{x}$, and the convective velocity is the relative velocity
$\boldsymbol{u}_f-\boldsymbol{w}$.
In our numerical implementation, the fluid equation is not solved by deforming
the computational mesh during each time step. Instead, the Navier--Stokes equations on the moving physical domain are pulled back to the fixed reference domain $\widehat{\Omega}_f$.
Let
\[
F
=
\widehat{\nabla}\mathcal{A}_t
=
I+\widehat{\nabla}\boldsymbol{\eta}_{\mathrm{ALE}},
\qquad
J=\det F,
\]
where $\widehat{\nabla}$ denotes differentiation with respect to the reference
coordinate. Then physical-domain gradients are computed from reference-domain
gradients by
\[
\nabla_{\mathbf{x}}\boldsymbol{u}
=
\widehat{\nabla}\boldsymbol{u}\,F^{-1},
\]
and the volume integrals over the moving domain are transformed by
\[
\int_{\Omega_f(t)} g(\mathbf{x})\,d\mathbf{x}
=
\int_{\widehat{\Omega}_f}
g(\mathcal{A}_t(\widehat{\mathbf{x}}))\,J\,d\widehat{\mathbf{x}}.
\]
Therefore, the fluid weak form is assembled on the reference domain. 

The arterial wall is modeled by the Biot system. To represent the circumferential recoil that appears in the axisymmetric reduction of a three-dimensional cylindrical wall, the momentum equation is augmented by a spring term $\beta \boldsymbol{\eta}$ \cite{Parrow2026R}. Hence the structure equation takes the form
\begin{equation}
\rho_p \partial_{tt}\boldsymbol{\eta}
-
\nabla\cdot \boldsymbol{\sigma}_p(\boldsymbol{\eta},p_p)
+
\beta \boldsymbol{\eta}
=
\boldsymbol{f}_p
\qquad
\text{in } \Omega_p\times(0,T].
\label{eq:benchmark_structure_momentum}
\end{equation}
Under the assumption of axial symmetry, the two-dimensional problem is interpreted as the meridional section of a three-dimensional cylindrical tube. Along the horizontal symmetry axis, denoted by $\Gamma_f^{\mathrm{sym}}$, the symmetry condition
\begin{equation}
\boldsymbol{u}_f\cdot \mathbf{n}_f = 0
\qquad \text{on } \Gamma_f^{\mathrm{sym}}\times(0,T]
\end{equation}
is imposed. 

The benchmark employs the following boundary conditions:

\paragraph{Poroelastic displacement boundary condition}
The wall displacement is fixed at the inlet and outlet:
\begin{equation}
\boldsymbol{\eta}=\mathbf{0}
\qquad
\text{on }
\left(\Gamma_p^{\mathrm{in}}\cup \Gamma_p^{\mathrm{out}}\right)\times(0,T].
\end{equation}
Since the structure velocity satisfies $\boldsymbol{\xi}=\partial_t\boldsymbol{\eta}$, we also impose
\begin{equation}
\boldsymbol{\xi}=\mathbf{0}
\qquad
\text{on }
\left(\Gamma_p^{\mathrm{in}}\cup \Gamma_p^{\mathrm{out}}\right)\times(0,T].
\end{equation}

\paragraph{Poroelastic external-wall stress condition}
On the external wall boundary, the poroelastic wall is traction-free:
\begin{equation}
\boldsymbol{\sigma}_p\mathbf{n}_p=\mathbf{0}
\qquad
\text{on }
\Gamma_p^{\mathrm{ext}}\times(0,T].
\end{equation}

\paragraph{Poroelastic no-flux condition}
For the pore fluid, no normal Darcy flux is imposed through the inlet, outlet, and external wall boundaries:
\begin{equation}
\boldsymbol{q}\cdot\mathbf{n}_p=K\nabla\phi\cdot\mathbf{n}_p=0
\qquad
\text{on }
\left(\Gamma_p^{\mathrm{in}}\cup\Gamma_p^{\mathrm{out}}\cup\Gamma_p^{\mathrm{ext}}\right)\times(0,T].
\end{equation}

\paragraph{Fluid symmetry condition}
Along the symmetry boundary of the half-domain fluid problem, we impose
\begin{equation}
\boldsymbol{u}_f\cdot\mathbf{n}_f=0
\qquad
\text{on }
\Gamma_f^{\mathrm{sym}}\times(0,T].
\end{equation}
For the present half-domain geometry, this is implemented as $u_{f,y}=0$ on $y=0$; the corresponding tangential stress condition is treated naturally.

\paragraph{Fluid inlet condition}
A time-dependent pressure pulse is prescribed at the inlet:
\begin{equation}
\boldsymbol{\sigma}_f \mathbf{n}_f
=
- p_{\mathrm{in}}(t)\mathbf{n}_f
\qquad
\text{on }
\Gamma_f^{\mathrm{in}}\times(0,T],
\end{equation}
where
\begin{equation}
p_{\mathrm{in}}(t)=
\begin{cases}
\dfrac{P_{\max}}{2}
\left[
1-\cos\left(\dfrac{2\pi t}{T_{\max}}\right)
\right],
& t\le T_{\max},\\[8pt]
0, & t>T_{\max},
\end{cases}
\label{eq:benchmark_pin}
\end{equation}
with
\[
P_{\max}=13334\ \mathrm{dyn}/\mathrm{cm}^2,
\qquad
T_{\max}=0.003\ \mathrm{s}.
\]

\paragraph{Fluid outlet condition}
At the outlet, the normal component of the fluid traction is set to zero:
\begin{equation}
\left(\boldsymbol{\sigma}_f \mathbf{n}_f\right)\cdot \mathbf{n}_f=0
\qquad
\text{on }
\Gamma_f^{\mathrm{out}}\times(0,T].
\end{equation}
This condition is imposed as a natural boundary condition in the weak formulation.

The ALE displacement is obtained by extending the wall displacement from the
fluid-structure interface into the fluid mesh. In the present implementation,
this extension is computed by solving the harmonic extension:
\begin{equation}
-\Delta \boldsymbol{\eta}_{\mathrm{ALE}}
=
\mathbf{0}
\qquad
\text{in } \widehat{\Omega}_f,
\label{eq:benchmark_ALE_extension}
\end{equation}
with interface condition
\begin{equation}
\boldsymbol{\eta}_{\mathrm{ALE}}
=
\boldsymbol{\eta}
\qquad
\text{on } \Gamma_{fp}^{-}\cup\Gamma_{fp}^{+}.
\label{eq:benchmark_ALE_interface_bc}
\end{equation}
On the inlet and outlet, the ALE displacement is fixed,
\begin{equation}
\boldsymbol{\eta}_{\mathrm{ALE}}=\mathbf{0}
\qquad
\text{on } \Gamma_f^{\mathrm{in}}\cup\Gamma_f^{\mathrm{out}},
\label{eq:benchmark_ALE_inout_bc}
\end{equation}
and on the symmetry boundary the normal component of the mesh displacement is
constrained,
\begin{equation}
\boldsymbol{\eta}_{\mathrm{ALE}}\cdot\mathbf{n}_f=0
\qquad
\text{on } \Gamma_f^{\mathrm{sym}}.
\label{eq:benchmark_ALE_sym_bc}
\end{equation}
In our second-order time-stepping scheme, the interface displacement used to move
the fluid mesh is extrapolated explicitly.
Thus, the ALE boundary condition at the interface is imposed as
\[
\boldsymbol{\eta}_{\mathrm{ALE}}^{n+1}
=
2\boldsymbol{\eta}_{\Gamma}^{n}
-
\boldsymbol{\eta}_{\Gamma}^{n-1}
\qquad
\text{on } \Gamma_{fp}^{-}\cup\Gamma_{fp}^{+}.
\]
After solving the ALE extension problem, the deformation gradient $
F^{n+1}
=
I+\widehat{\nabla}\boldsymbol{\eta}_{\mathrm{ALE}}^{n+1}
$
is computed on the reference fluid mesh. The fluid weak form is then assembled
using
\[
J^{n+1}=\det F^{n+1},
\qquad
(F^{n+1})^{-1}.
\]
The mesh velocity
\[
\boldsymbol{w}^{n+1}
=
\frac{
\boldsymbol{\eta}_{\mathrm{ALE}}^{n+1}
-
\boldsymbol{\eta}_{\mathrm{ALE}}^{n}
}{\Delta t}
\]
enters the transformed convection term through the relative velocity
$\boldsymbol{u}_f-\boldsymbol{w}$. In this way, the effect of the moving physical
fluid domain is included without solving the Navier--Stokes system directly on a
permanently moved mesh. All the physical parameters used in the benchmark are listed in Table~\ref{tab:physical_parameters_benchmark} \cite{Parrow2026R, he2026lockfree}.

\begin{table}[h!]
\centering
\caption{Physical parameters for the 2D blood flow benchmark.}
\label{tab:physical_parameters_benchmark}
\begin{tabular}{l|l|l|l}
\hline
Parameter & Symbol & Units & Reference value \\
\hline
Radius & $R$ & cm & $0.5$ \\
Length & $L$ & cm & $6$ \\
Poroelastic wall density & $\rho_p$ & $\mathrm{g}/\mathrm{cm}^3$ & $1.1$ \\
Fluid density & $\rho_f$ & $\mathrm{g}/\mathrm{cm}^3$ & $1.0$ \\
Dynamic viscosity & $\mu_f$ & $\mathrm{g}/(\mathrm{cm}\cdot \mathrm{s})$ & $0.035$ \\
Spring coefficient & $\beta$ & $\mathrm{dyn}/\mathrm{cm}^4$ & $4\mathrm{e}+06$ \\
Storage coefficient & $c_0$ & $\mathrm{cm}^2/\mathrm{dyn}$ & $10^{-3}$ \\
Permeability & $K$ & $\mathrm{cm}^2$ & $10^{-6}\mathbf{I}$ \\
Lam\'e coefficient & $\mu_p$ & $\mathrm{dyn}/\mathrm{cm}^2$ & $5.575\mathrm{e}+05$ \\
Lam\'e coefficient & $\lambda_p$ & $\mathrm{dyn}/\mathrm{cm}^2$ & $1.7\mathrm{e}+06$ \\
BJS coefficient & $\gamma$ & $\mathrm{g}/(\mathrm{cm}^2\cdot \mathrm{s})$ & $10^3$ \\
Biot--Willis constant & $\alpha$ & -- & $1$ \\
Robin coupling constant & $L$ & -- & $1000$ \\
\hline
\end{tabular}
\end{table}

The numerical simulation was implemented in FEniCS on a half-domain by exploiting axial symmetry. Since the present benchmark involves a moving-domain Navier--Stokes--Biot model, rather than the Stokes--Biot formulation considered in the referenced works \cite{Parrow2026R, he2026lockfree}, we validate our solution through mesh convergence tests. On coarser meshes, we observed that the pressure wave propagates more slowly, which is indicative of numerical dissipation. Therefore, the results presented below correspond to the finest mesh. Specifically, the fluid domain was discretized using a structured triangular mesh with axial resolution $n_x=180$ and radial resolution $n_y=15$, yielding a total of $5400$ triangular elements. The poroelastic wall used the same axial resolution $n_x=180$ and radial resolution of $n_y=10$, resulting in \(3600\) triangular elements. Taylor--Hood \(P_2/P_1\) elements were used for the fluid variables, \(P_2\) elements were used for the structure displacement and structure velocity, and \(P_1\) elements for the pore pressure. For the ALE mesh problem, we use \(P_2\) elements. The time step was chosen to be $\Delta t=10^{-6}$, and the simulation was run until the final time $T=0.014\,\mathrm{s}$. In the code, the ALE displacement is first computed on the reference fluid mesh.
Then the tensor $
F=I+\widehat{\nabla}\boldsymbol{\eta}_{\mathrm{ALE}}
$
is projected onto a tensor-valued finite element space and used as the deformation
gradient of the ALE map. The Navier--Stokes equations are assembled on the
reference fluid domain by inserting $J=\det F$ and $F^{-1}$ into the weak form.
In particular, the mass term is multiplied by $J$, the ALE convection term uses
the transformed relative velocity $F^{-1}(\boldsymbol{u}_f-\boldsymbol{w})$, the
stress and pressure terms are transformed using $F^{-1}$. The time derivative and structure coupling are discretized by BDF2, while the nonlinear convection is treated semi-implicitly using the previous-time velocity $\boldsymbol{u}_f^n$ for stability.

In Figure \ref{fig:four_snapshots}, we report the snapshots of the pressure (fluid pressure and pore pressure superimposed) at different time moments, which agrees well with the work reported in \cite{Parrow2026R}.
\begin{figure}[htbp!]
    \centering
\includegraphics[width=0.95\textwidth]{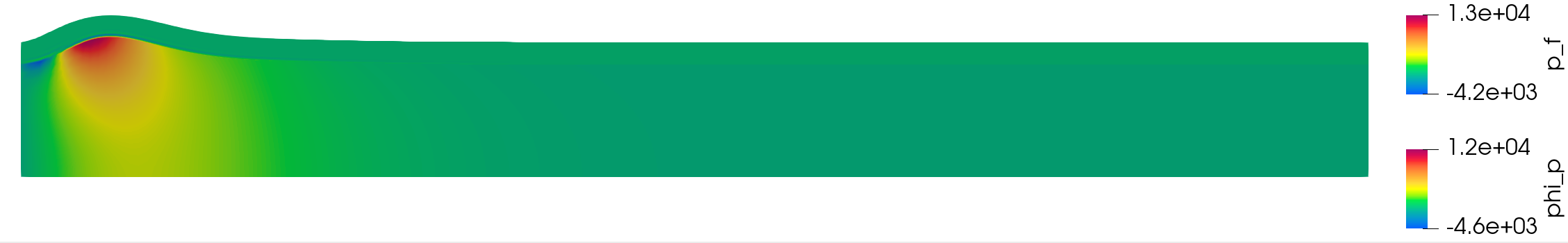}
    \hfill
\includegraphics[width=0.95\textwidth]{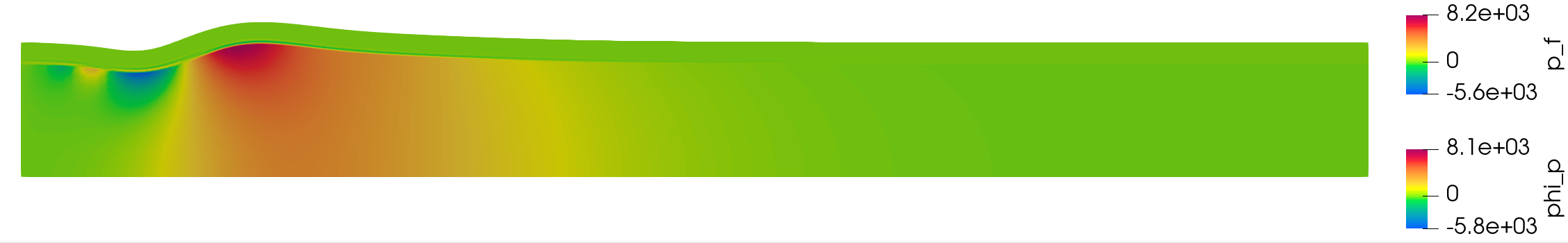}
\includegraphics[width=0.95\textwidth]{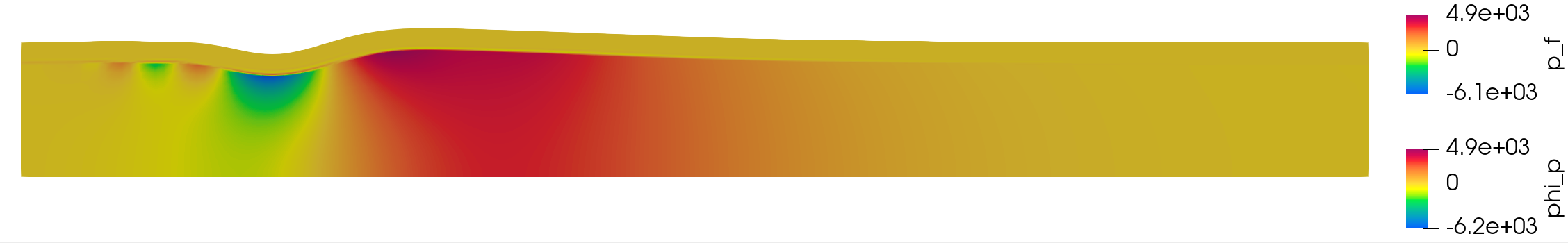}
\includegraphics[width=0.95\textwidth]{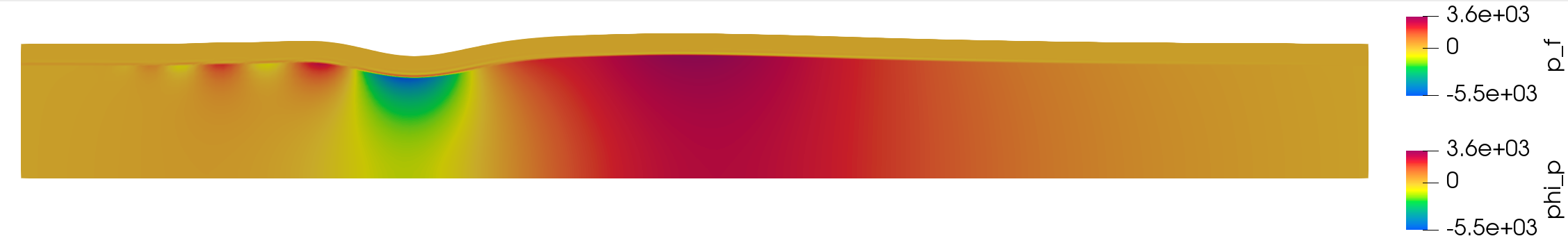}
\caption{Snapshots showing the fluid pressure $p_f$ and pore pressure $\phi_p$ at $T=0.0035\,\mathrm{s}$, $0.007\,\mathrm{s}$, $0.0105\,\mathrm{s}$, and $0.014\,\mathrm{s}$. The deformation is magnified by a factor of 5 for visualization.}
    \label{fig:four_snapshots}
\end{figure}
Figure \ref{fig:continuity_pressure} shows a portion of the computational domain at $t=0.014\,\mathrm{s}$. The mesh displacement has been magnified by a factor of $50$ to make the deformation visible. The pressure fields are superimposed on the deformed mesh.
The smooth transition of the pressure-like fields near the interface is
consistent with the normal-stress coupling imposed in the model.
\begin{figure}[htbp!]
    \centering
\includegraphics[width=0.95\textwidth]{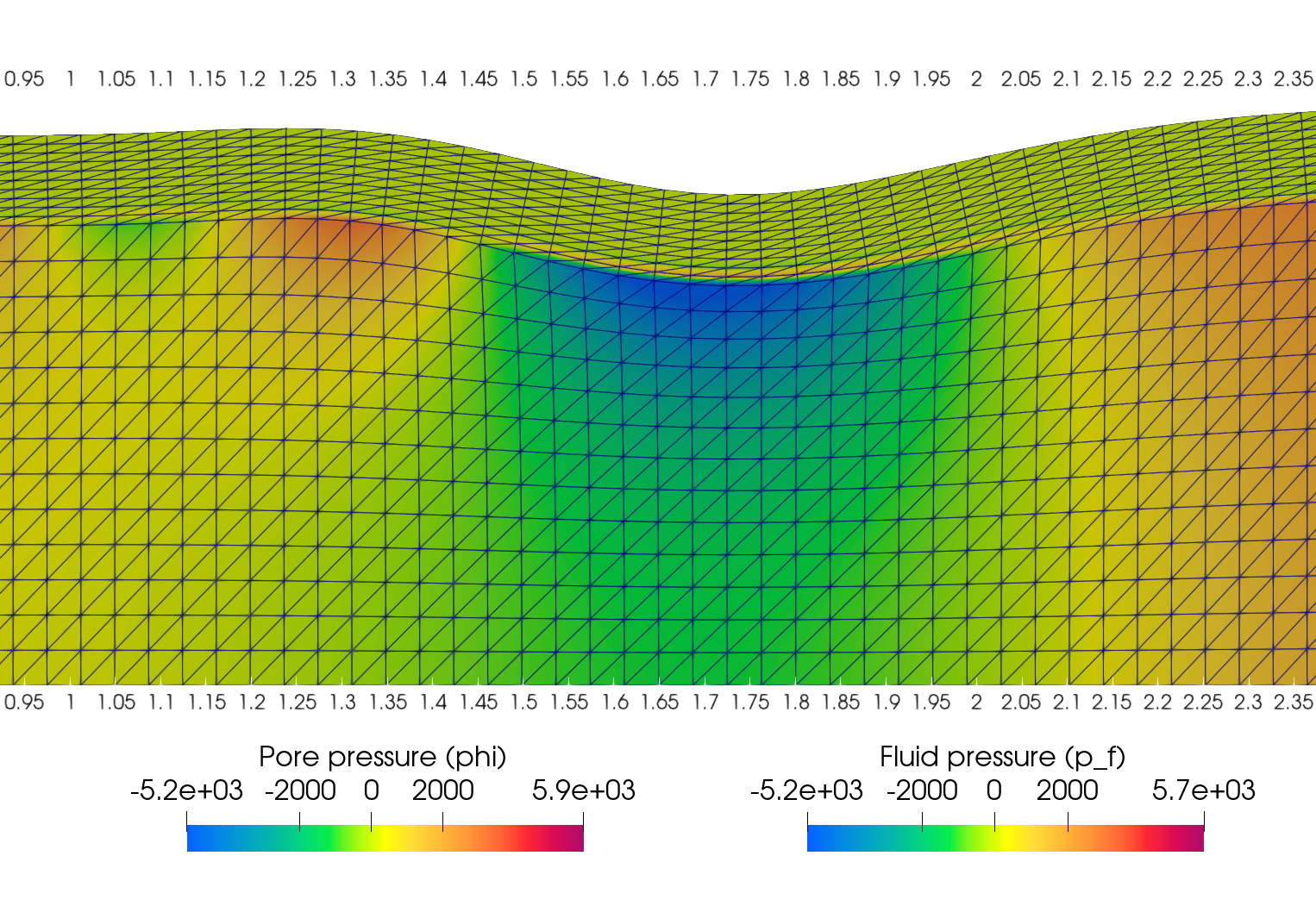}
\caption{Pressure fields near the fluid--poroelastic interface at $t=0.014\,\mathrm{s}$.
The mesh displacement is magnified by a factor of 5 for visualization. The pore
pressure $\phi$ in the poroelastic wall and the fluid pressure $p_f$ in the
lumen are superimposed on the deformed mesh.}
    \label{fig:continuity_pressure}
\end{figure}
In Figure \ref{fig:four_snapshots_v}, we also report the snapshots of the velocities (fluid velocity and poroelastic structure velocity superimposed) at different time moments.
\begin{figure}[htbp!]
    \centering
\includegraphics[width=0.95\textwidth]{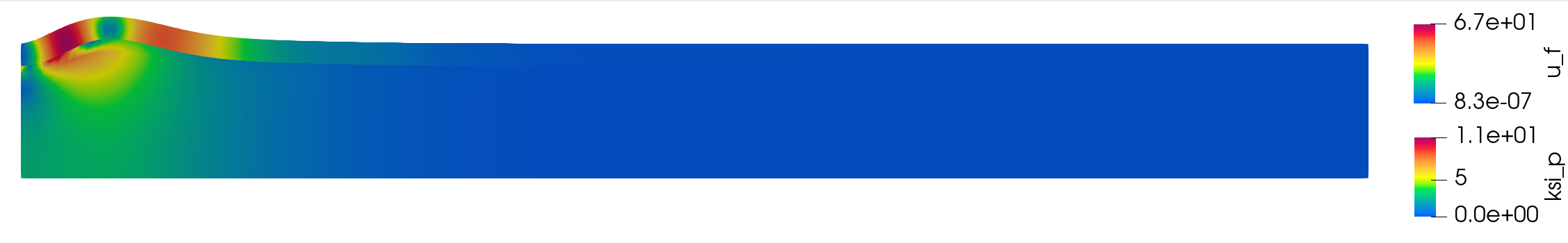}
    \hfill
\includegraphics[width=0.95\textwidth]{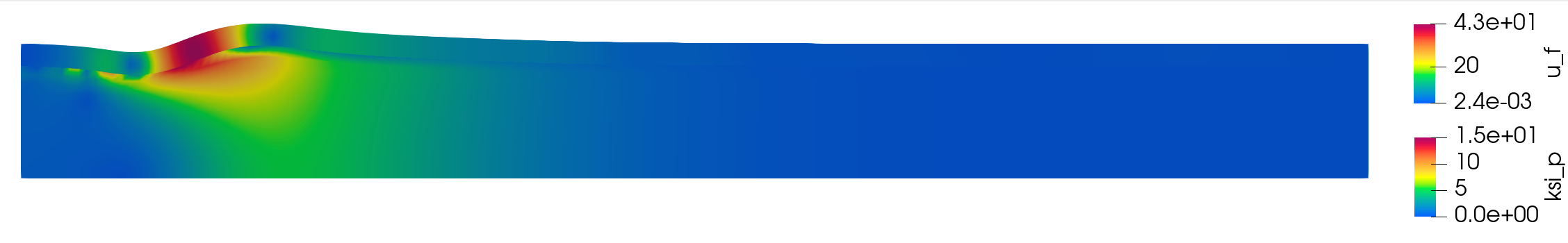}
\includegraphics[width=0.95\textwidth]{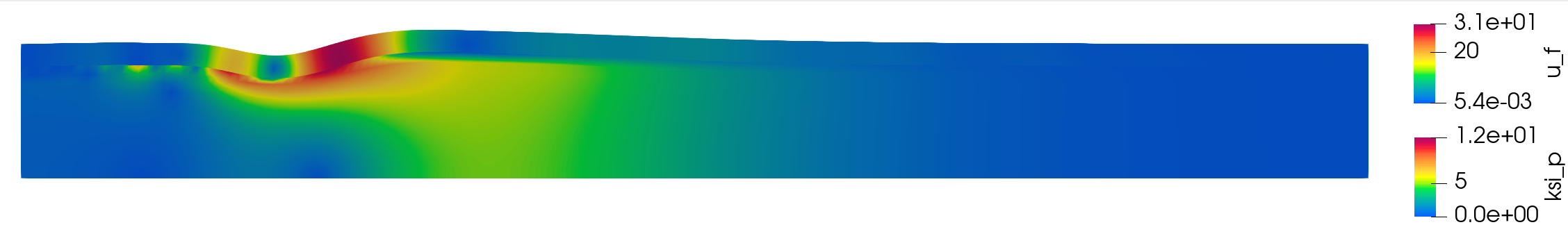}
\includegraphics[width=0.95\textwidth]{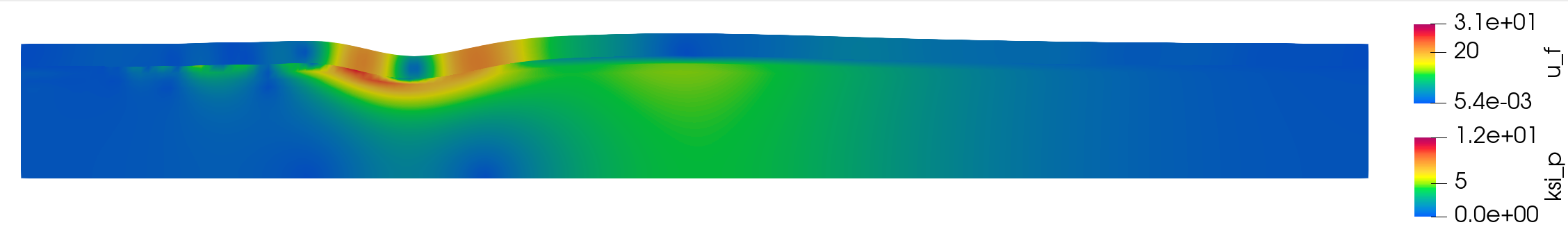}
\caption{Snapshots showing the fluid velocity $\boldsymbol{u}_f$ and poroelastic structure velocity $\boldsymbol{\xi}_p$ at $T=0.0035\,\mathrm{s}$, $0.007\,\mathrm{s}$, $0.0105\,\mathrm{s}$, and $0.014\,\mathrm{s}$. The deformation is magnified by a factor of 5 for visualization.}
    \label{fig:four_snapshots_v}
\end{figure}

We also report the pressure and fluid velocity solutions at the fluid--structure interface $y=0.5$ and along the channel centerline $y=0$. Figures~\ref{fig:1d_snapshots_v} and~\ref{fig:1d_snapshots_p} show snapshots of the fluid velocity $\boldsymbol{u}_f$ and pressure $p_f$, respectively, at $T=0.007\,\mathrm{s}$, $0.0105\,\mathrm{s}$, and $0.014\,\mathrm{s}$. In each figure, the top row corresponds to the solutions at the fluid--structure interface, while the bottom row corresponds to the solutions at the channel centerline. In addition, Figure~\ref{fig:1d_snapshots_d} presents the corresponding $x$- and $y$-components of the structure displacement $\boldsymbol{\eta}$ at the fluid--structure interface $y=0.5$ at the same time instances.

\begin{figure}[htbp!]
    \centering
\includegraphics[width=0.31\textwidth]{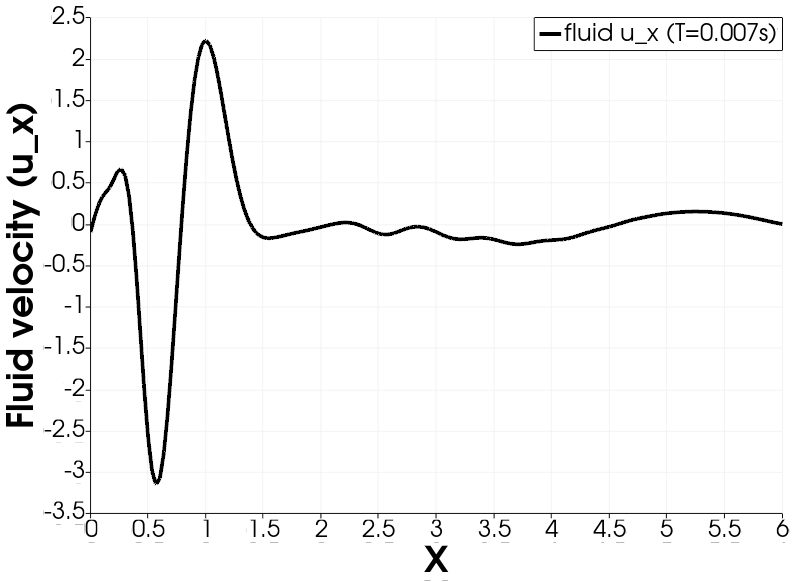}
\includegraphics[width=0.31\textwidth]{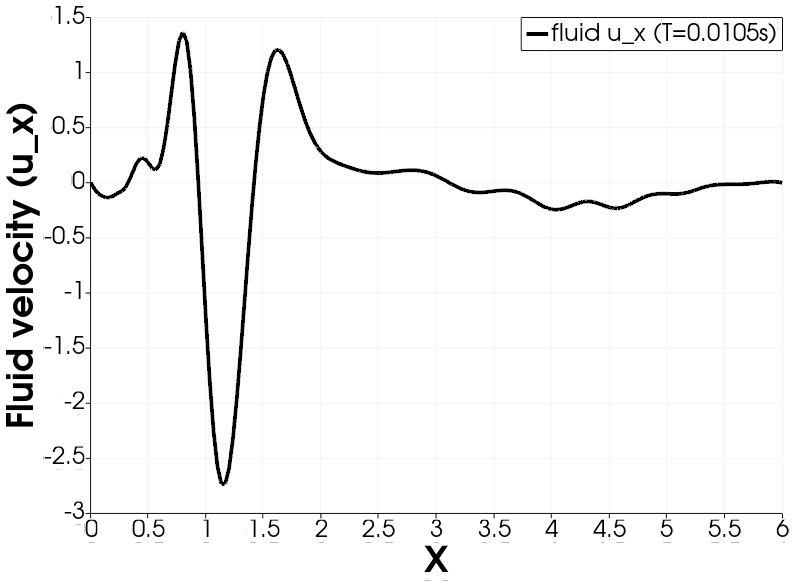}
\includegraphics[width=0.31\textwidth]{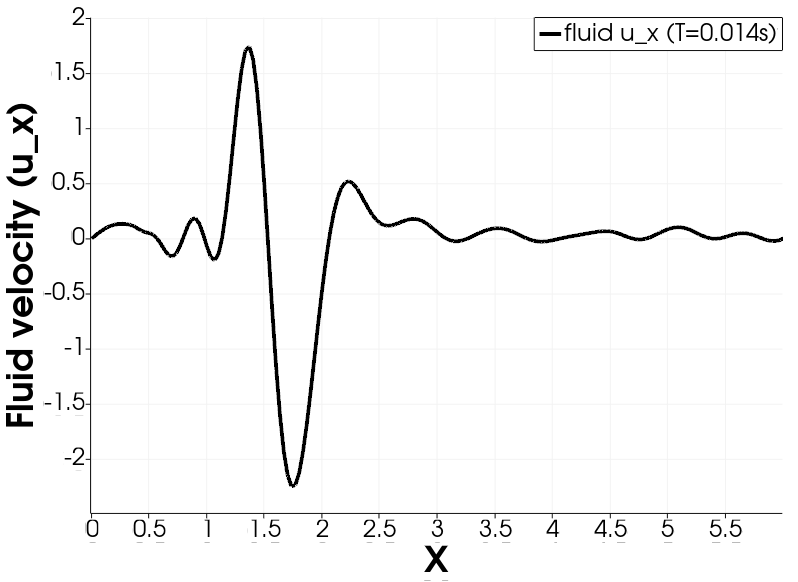}
\includegraphics[width=0.31\textwidth]{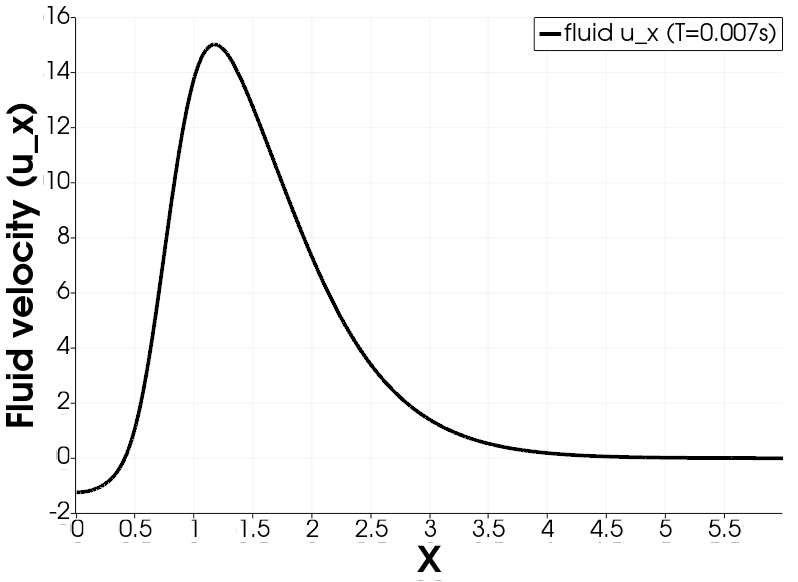}
\includegraphics[width=0.31\textwidth]{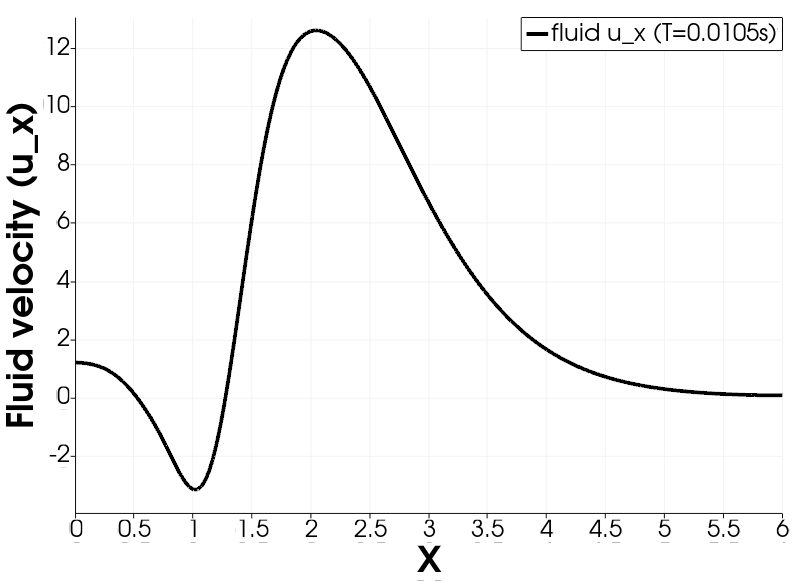}
\includegraphics[width=0.31\textwidth]{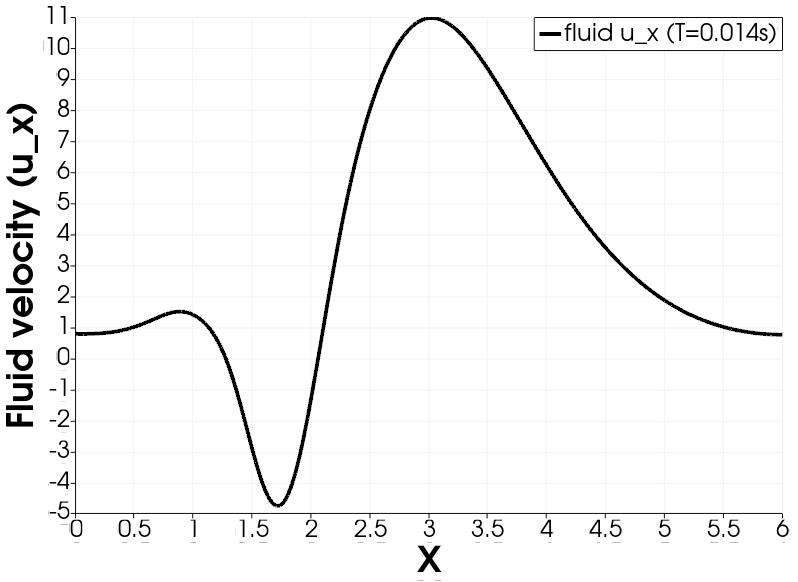}
\caption{Snapshots showing the fluid velocity $\boldsymbol{u}_f$ at the fluid–structure interface $y=0.5$ (Top row) and along the channel centerline $y=0$ (bottom row), at $T=0.007\,\mathrm{s}$, $0.0105\,\mathrm{s}$, and $0.014\,\mathrm{s}$ respectively. }
    \label{fig:1d_snapshots_v}
\end{figure}

\begin{figure}[htbp!]
    \centering
\includegraphics[width=0.31\textwidth]{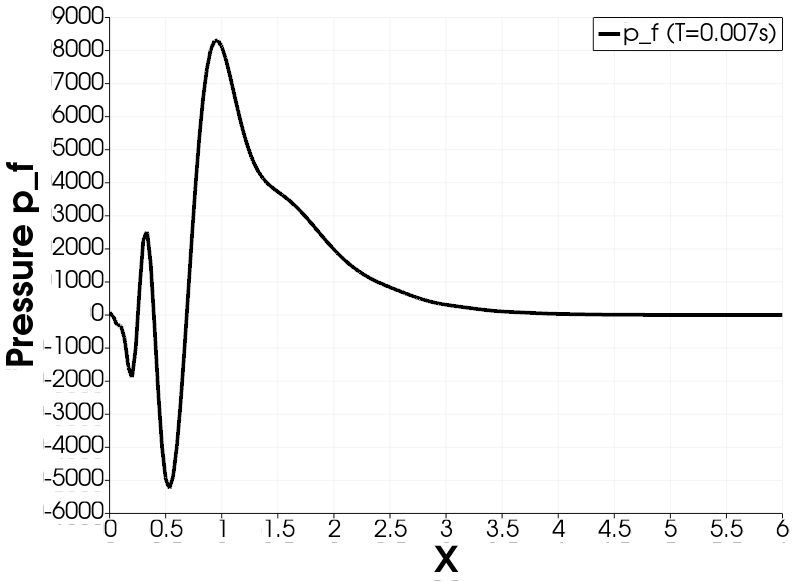}
\includegraphics[width=0.31\textwidth]{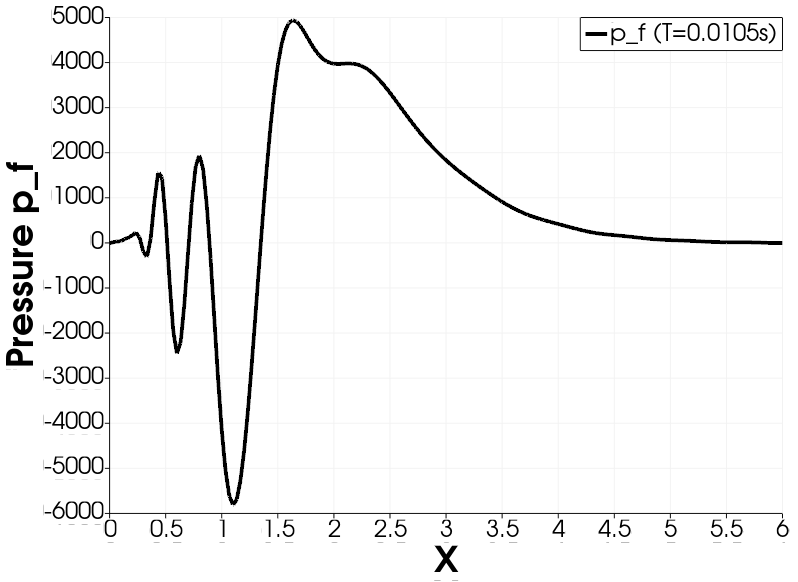}
\includegraphics[width=0.31\textwidth]{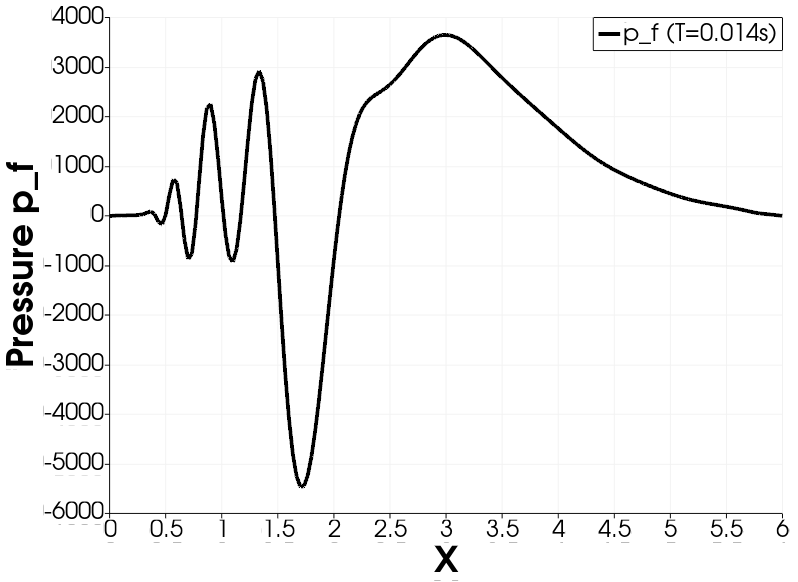}
\includegraphics[width=0.31\textwidth]{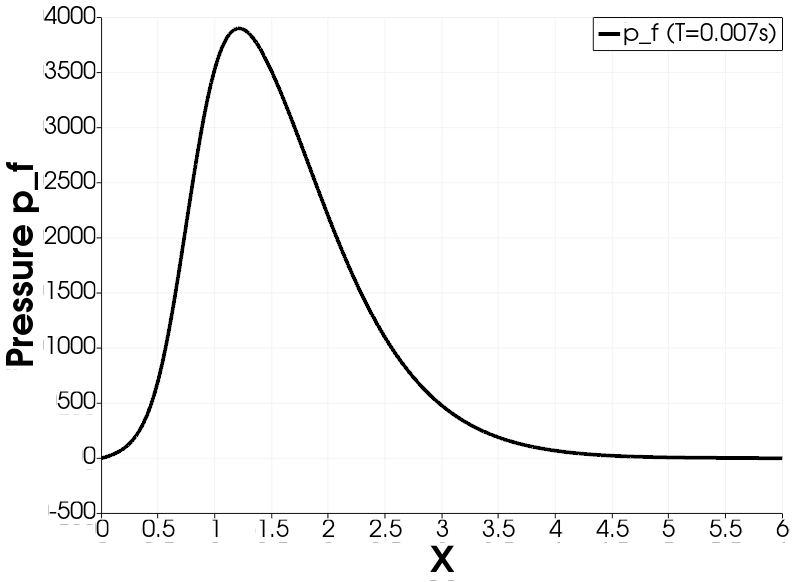}
\includegraphics[width=0.31\textwidth]{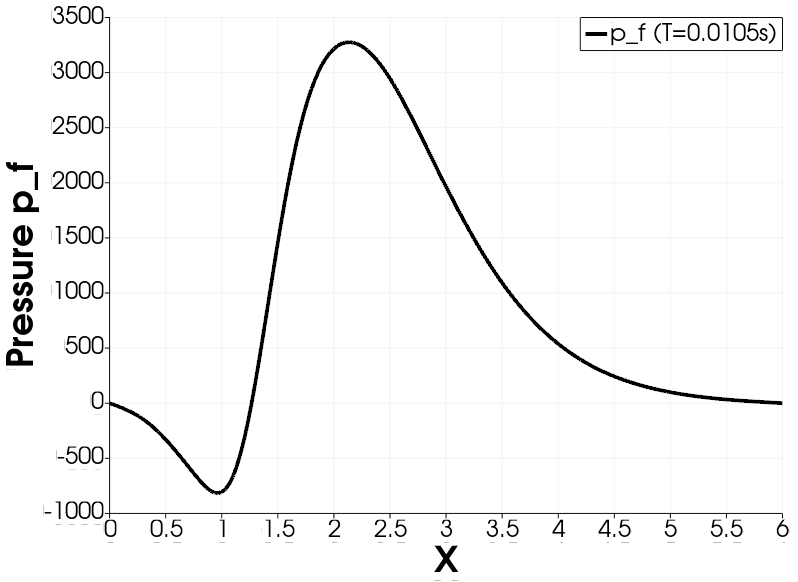}
\includegraphics[width=0.31\textwidth]{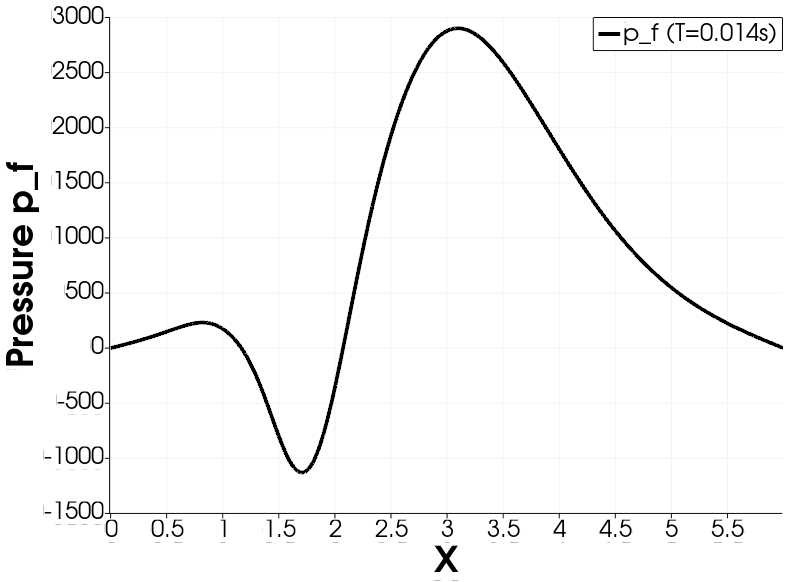}
\caption{Snapshots showing the fluid pressure $p_f$ at the fluid–structure interface $y=0.5$ (Top row) and along the channel centerline $y=0$ (bottom row), at $T=0.007\,\mathrm{s}$, $0.0105\,\mathrm{s}$, and $0.014\,\mathrm{s}$ respectively. } 
    \label{fig:1d_snapshots_p}
\end{figure}

\begin{figure}[htbp!]
    \centering
\includegraphics[width=0.31\textwidth]{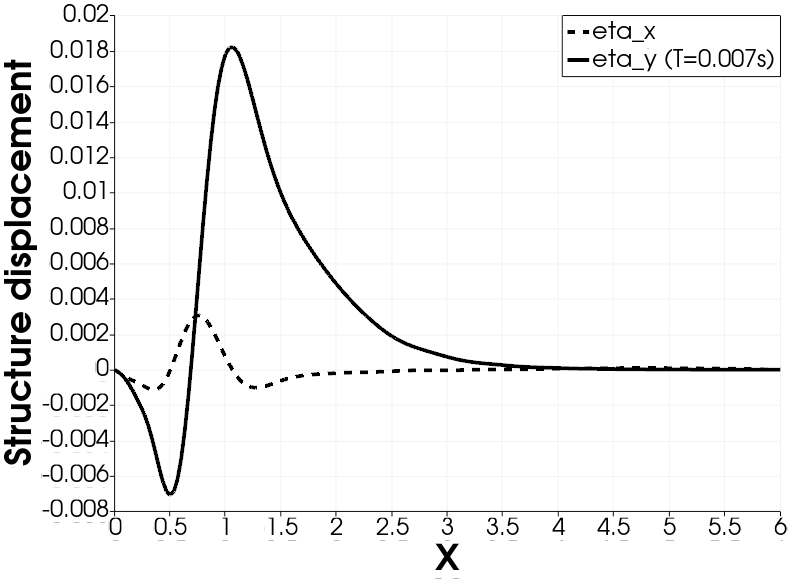}
\includegraphics[width=0.31\textwidth]{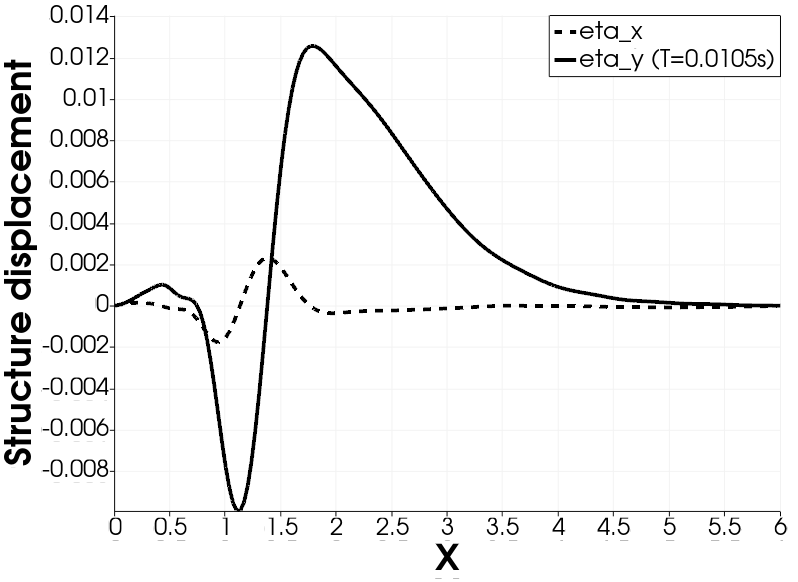}
\includegraphics[width=0.31\textwidth]{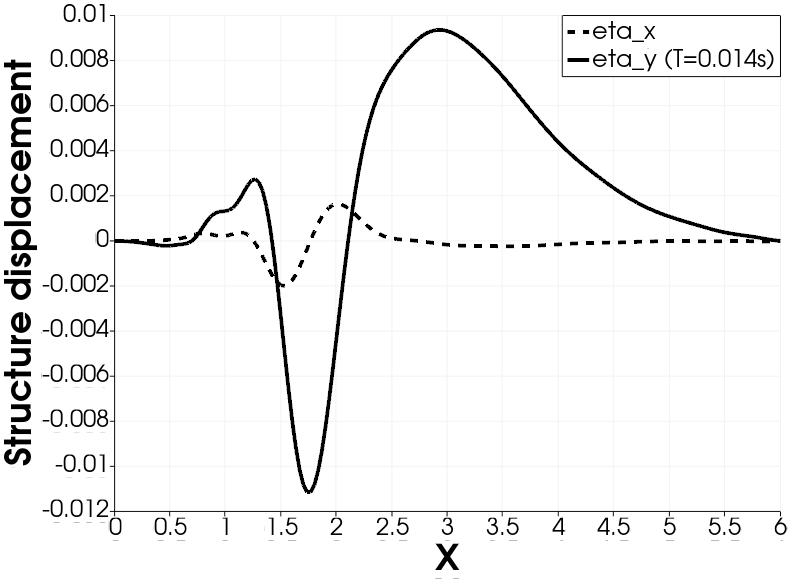}
\caption{Snapshots showing the $x ,y$ components of structure displacement $\boldsymbol{\eta}$ at the fluid–structure interface $y=0.5$, at $T=0.007\,\mathrm{s}$, $0.0105\,\mathrm{s}$, and $0.014\,\mathrm{s}$ respectively. } 
    \label{fig:1d_snapshots_d}
\end{figure}

\section{Conclusion}
\label{sec:main6}
We have developed and analyzed a fully discrete second-order explicit splitting scheme for fluid-poroelastic structure interaction problems governed by the Stokes-Biot system in a fixed domain. 
The method is based on a Robin reformulation of the interface conditions \cite{Wang2026ExplicitSplitting, wang2025} and combines BDF2 time stepping in the subdomains with AB2 extrapolation of the interface data. 
This yields a partitioned algorithm in which the fluid and poroelastic subproblems can be solved independently and in parallel at each time step and the scheme remains second order in time.

The main theoretical contribution of this work is a rigorous stability and error analysis for this explicit splitting strategy. 
By exploiting BDF2 energy identities and a careful treatment of the extrapolated interface residuals, we derived a closed discrete stability estimate under a CFL condition. 
We then introduced suitable projection operators in the fluid and poroelastic subdomains and used them to split the total error into approximation, consistency, and discrete components. 
This led to a discrete error energy inequality and, ultimately, to an a priori error estimate in the bulk energy norms induced by the partitioned formulation, together with the tangential interface mismatch controlled by the Robin coupling.
Under the stated regularity assumptions and second-order initialization, the method shows second-order accuracy in time together with optimal spatial convergence for the polynomial degrees covered by the analysis.

In the numerical experiment with a manufactured solution test, we observed that the computed temporal errors for the fluid velocity and pressure, structure displacement and velocity, and pore pressure are all close to second order. 
The spatial experiments also show nearly second order convergence for the corresponding finite element spaces, including Taylor-Hood approximation for the Stokes subproblem and compatible finite element spaces for the poroelastic variables. 
Overall, the results indicate that the proposed method offers a practical compromise between modularity, parallel efficiency, and provable accuracy.

Several extensions remain of interest. 
On the analytical side, it would be useful to investigate whether the CFL restriction can be relaxed or improved for particular choices of interface parameters and finite element spaces. 
On the modeling side, an important next step is to extend the present framework beyond fixed domains to moving-interface and geometrically nonlinear fluid-poroelastic interaction problems. 
Such developments would broaden the applicability of second-order explicit partitioned schemes to more realistic multiphysics settings.

\section*{Acknowledgement}
{\sunny{
\v{C}ani\'{c}'s research has been supported in part by the
National Science Foundation under grants DMS-2408928, DMS-2247000 and by the  U.S. Department of Energy, Office of Science, Office of Advanced Scientific Computing Research's Applied Mathematics Competitive Portfolios program under Contract No. AC02-05CH11231.
Wang’s research has been supported in part by the National Science Foundation under grant DMS-2247001, CPRIT Texas under grant RP260780, and by Simons Foundation Travel Award.}}

\bibliographystyle{siamplain}

\bibliography{merged_references_clean}
\end{document}